\documentclass[10pt,reqno]{amsart}
\usepackage[utf8]{inputenc}
\usepackage[T1]{fontenc}
\usepackage[utf8]{inputenc}
\usepackage[english]{babel}
\usepackage{amsmath}
\usepackage{amssymb}
\usepackage{amsfonts}
\usepackage{dsfont}
\usepackage{float}
\usepackage{graphicx}
\usepackage{wrapfig}
\usepackage{mathtools}
\usepackage{bbm}
\usepackage{amsthm}
\usepackage{ifthen}
\usepackage{graphicx}
\usepackage{hyperref}
\usepackage[ruled,vlined]{algorithm2e}
\usepackage{xcolor}
\usepackage{mathtools}
\usepackage{empheq}
\usepackage{mathrsfs}
\usepackage{multicol}

\newtheorem{theorem}{Theorem}[section]

\newtheorem{lemma}[theorem]{Lemma}
\newtheorem{cor}[theorem]{Corollary}

\newtheorem{assumptions}[theorem]{Assumptions}
\theoremstyle{definition}
\newtheorem{definition}[theorem]{Definition}
\theoremstyle{remark}
\newtheorem{remark}[theorem]{Remark}

\newcommand{\bb}[1]{\mathbb{#1}}

\newcommand{\R}{\bb{R}}

\newcommand{\norm}[1]{\left\lVert#1\right\rVert}

\chardef\bslash=`\\ 

\numberwithin{equation}{section}

\newcommand{\N}{\mathbb{N}}
\newcommand{\E}{\mathbb{E}}
\newcommand{\PP}{\mathbb{P}}

\def\bm{\left( \begin{array}{cc}}
\def\endm{\end{array}\right)}

 \providecommand{\norm}[1]{\lVert#1 \rVert}

\newcommand{\be}{\begin{equation}}
\newcommand{\ee}{\end{equation}}
\newcommand{\ba}{\left(\begin{array}{c}}
\newcommand{\ea}{\end{array}\right)}
\newcommand{\bea}{\begin{eqnarray}}
\newcommand{\eea}{\end{eqnarray}}
\newcommand{\bee}{\begin{eqnarray*}}
\newcommand{\eee}{\end{eqnarray*}}
\newcommand{\ben}{\begin{enumerate}}
\newcommand{\een}{\end{enumerate}}

\newcommand{\vertiii}[1]{{\left\vert\kern-0.25ex\left\vert\kern-0.25ex\left\vert #1 
		\right\vert\kern-0.25ex\right\vert\kern-0.25ex\right\vert}}

\usepackage[letterpaper, margin=3cm]{geometry}

\title[Wave equation and DNN]{Bounds on the approximation error for deep neural networks applied to dispersive models: Nonlinear waves
}

\author{Claudio Mu\~noz}  
	\address{Departamento de Ingenier\'{\i}a Matem\'atica and Centro
de Modelamiento Matem\'atico (UMI 2807 CNRS), Universidad de Chile.}
	\email{cmunoz@dim.uchile.cl}
	\thanks{C.M. was partially funded by Chilean research grants ANID 2022 Exploration 13220060, FONDECYT 1231250, and Basal CMM FB210005.}
\date{\today}

\author{Nicol\'as Valenzuela}
\address{Departamento de Ingenier\'{\i}a Matem\'atica, Universidad de Chile.}
	\email{nvalenzuela@dim.uchile.cl}
	\thanks{N.V. was partially funded by Chilean research grants ANID 2022 Exploration 13220060, FONDECYT 1231250, a Latin America PhD Google Fellowship and Basal CMM FB210005.}

\begin{document}

\begin{abstract}
We present a comprehensive framework for deriving rigorous and efficient bounds on the approximation error of deep neural networks in PDE models characterized by branching mechanisms, such as waves, Schr\"odinger equations, and other dispersive models. This framework utilizes the probabilistic setting established by Henry-Labord\`ere and Touzi. We illustrate this approach by providing rigorous bounds on the approximation error for both linear and nonlinear waves in physical dimensions $d=1,2,3$, and analyze their respective computational costs starting from time zero. We investigate two key scenarios: one involving a linear perturbative source term, and another focusing on pure nonlinear internal interactions. 
\end{abstract}

\maketitle


\section{Introduction and Main Results}

\subsection{Approximation of PDEs via DNN}

Deep neural networks have proven to be useful tools for solving partial differential equations, given their ability to act as a \textit{universal approximator} of continuous functions with compact support, or in mathematical terms, the space of deep neural networks is \textit{dense} in the space of continuous functions with compact support \cite{Hornik,Leshno}. Nevertheless, the complexity of the neural network depends on the number of degrees of freedom that define it. Depending on the continuous function, the approximation by neural networks may have a number of degrees of freedom that grows exponentially with the dimension of the domain of said function, a phenomenon known as the curse of dimensionality.

\medskip

In recent years, various impressive results suggest that for some classes of PDEs, deep neural networks succeed in approximating their solutions, and furthermore, the number of degrees of freedom of the neural network is at most polynomial in the dimension of the problem, and inversely proportional to the accuracy of the approximation. Such results span from numerical to theoretical perspectives, with a wide range of applications. Making a fair review of all these works is far from optimal. However, among the PDEs studied with these methods, linear and semilinear parabolic equations \cite{Han,E,Hure,Hutz2,Beck1,Beck2,Jentz1}, stochastic time dependent models \cite{Grohs2,Beck2,Berner1}, linear and semilinear time-independent elliptic equations \cite{LLP,Grohs}, a broad range of fluid equations \cite{Lye,MJK,Mishra2}, and even time dependent/independent nonlocal equations \cite{PLK,Raissi0,Gonnon,Val22,Val23,Castro} stand out. We refer to these works for further developments in this fast growing area of research. Additionally, one can classify the learning methods used to prove the previous results. Among the deep learning methods studied to date, Monte Carlo methods, Multilevel Picard Iterations (MLP) \cite{Hutz2,Beck}, Physics Informed Neural Networks (PINNs) \cite{Raissi2,MJK,Mishra}, and even neural networks approximating infinite-dimensional operators, such as DeepOnets \cite{Chen,Lu,Castro2,CalderonDO}, have been prominent. Other methods in the literature are for instance mentioned in the recent review \cite{Beck23} and the monograph \cite{JKW}.

\medskip

In view of the previous results, it is of key interest to get general, rigorous and efficient deep learning methods and approaches to deal with wave type models, or dispersive equations. By different reasons, these models are not so well-understood from the DNN point of view. In this work, we will present a general framework for considering these models; the essential element being the existence of a Duhamel formulation and a Branching Mechanism to represent solutions in a probabilist fashion. We shall concentrate efforts in linear and nonlinear wave models, leaving other cases such as Schr\"odinger or KdV models for forthcoming works. For the case of wave equations, several authors have considered in detail the DNN approximation of waves, primarily from a numerical point of view. In \cite{MMN}, PINNs were used to solve the wave equation. One of the first rigorous results in this direction (i.e., using PINNs to approximate dispersive models) was obtained in \cite{BKM22}, based in previous work by Mishra and Molinaro \cite{MM22}. See also \cite{Cox, ACL24} for convergence and approximation results for wave type models in a more functional setting. We consider our work to be, to date, one of the first to rigorously solve the nonlinear wave equation using neural networks obtained by Monte Carlo methods.  While finishing this work, we became aware of the very recent work \cite{LBK24}, where rigorous error estimates are computed for PINNs approximating the semillinear wave equation
\[
\partial_t^2 u -\Delta u + a(x) \partial_t u + f(x,u)=0, \quad x\in \Omega, \quad t\in [0,T]. 
\]
It is assumed that $a\geq 0$ a.e. is nontrivial. Compared with these results, our methods below are different and do not involve the use of PINNs. We also consider in principle $a\equiv 0$. Additionally we shall also provide quantitative complexity bounds for the number of required DNNs, a key step to decide if the approximation rule is satisfactory to be numerically implemented. 

\medskip

\subsection{Setting} Let $d=1,2,3$ be the standard physical dimensions. Consider the classical wave equation posed in $\R\times \R^d$:
\begin{equation}\label{eq:linear_wave}
	\begin{cases}
		\partial_{tt} u - \Delta u = F & \hbox{ in } \R\times \R^d,\\
		u(0,\cdot) = f_1 & \hbox{ in } \R^d,\\
		\partial_t u(0,\cdot) = f_2 & \hbox{ in } \R^d.
	\end{cases}
\end{equation}
Here, $u=u(t,x)\in\mathbb R$, $t\geq0$ (without loss of generality), $x\in\mathbb R^d$ is the unknown function, usually modeling some (linear/nonlinear) oscillatory behavior.  On the other hand, in many applications one has source terms that modify the dynamics. Precisely, in \eqref{eq:linear_wave} $F=F(t,x,u)$ is a source term that may also depend on $u$ itself. We shall assume $F:\R\times\R^d \times \R\to \R$ continuous. As for the initial data, we shall assume that $f_2 : \R^d \to \R$ is bounded and continuous, and $f_1 : \R^d \to \R$ is in the class $C^2$ and bounded.\footnote{This extra regularity is technical but we believe that can be lifted with some extra work; however, it will not be the purpose of this already long work.} By making the change of variable $u(t,x) = U(t,x) + f_1(x)$, and $\widetilde{F} = F + \Delta f_1$, we obtain the following simplified problem
\begin{equation}\label{eq:linear_wave2}
	\begin{cases}
		\partial_{tt} U - \Delta U =  \widetilde{F} &  \hbox{ in } \R\times \R^d,\\
		U(0,\cdot) = 0   & \hbox{ in } \R^d,\\
		\partial_t U(0,\cdot) =  f_2  & \hbox{ in } \R^d.
	\end{cases}
\end{equation}
The solution $u$ of \eqref{eq:linear_wave} can be directly obtained from $U$ by solving \eqref{eq:linear_wave2}.  

\medskip

This paper is our first statement on the deep neural network approximation of solutions to the wave equation, in the linear and nonlinear setting, always attempting to preserve the numerical cost of this approximation. This work will be devoted, specially in the nonlinear setting, to the problem of local in time generalization error, or the large time approximation under small data. This is somehow natural, since nonlinear waves may have blow up solutions if the data is large, or the time is sufficiently large. In mathematical terms, we will look for suitable \emph{artificial neural networks} of ReLu type for which a suitable solution, in a sense to be defined below, of \eqref{eq:linear_wave2}, is approximated to any accuracy while preserving the budget of approximation in a certain fixed interval $[0,T]$. It turns out that by several reasons, directly related to the hyperbolic nature of the model, this is a difficult question, somehow more complex than other previously treated cases such as diffusive and elliptic models, and will require additional assumptions and care to obtain satisfactory bounds.

\medskip

 In order to state the main results, we need to introduce the precise notion of solution of \eqref{eq:linear_wave2} required to establish our main results.  First of all, we state the Duhamel's formula for the problem \eqref{eq:linear_wave2}. Recall that $C_b^0([0, \infty)\times\R^d,\R)$ denotes the set of bounded, real-valued continuous functions $f=f(t,x)$, $t\in [0,T)$, $x\in\R^d$. For $(r,z) \in \R_+ \times \R^d$, we consider the fundamental solution of wave in $\mathbb R^d$ : 
\[
\begin{aligned}
	g_2(r,dz) &: = (2\pi)^{-d} \int \frac{\sin(r|\xi|)}{|\xi|} e^{i\xi\cdot z} d\xi,
\end{aligned}
\]
where the integral is taken on $\R^d_\xi$. Formally, $U=g_2$ solves ($\delta_0$ is the space-time Dirac delta)
\[
	\begin{cases}
		\partial_{tt} U - \Delta U =  \delta_0 &  \hbox{ in } \R\times \R^d,\\
		U(0,\cdot) = 0   & \hbox{ in } \R^d,\\
		\partial_t U(0,\cdot) =  0 & \hbox{ in } \R^d.
	\end{cases}
\]
Using Kirchhoff's formula, one gets\footnote{$1_A$ denotes the indicator function of the set $A$.} 
	\begin{equation}\label{eq:g2}
	g_2(r,dz) = \begin{cases}
		\displaystyle \frac 12 1_{ \{ |z|<r \} }dz & \hbox{for } d=1,\vspace{0.1cm} \\ \vspace{0.1cm}
		\displaystyle (2\pi)^{-1} (r^2-|z|^2)^{-\frac 12} 1_{ \{ |z|<r \}} dz & \hbox{for } d=2,\\
		\displaystyle (4\pi r)^{-1} \sigma_r(dz) & \hbox{for } d=3.
	\end{cases}
	\end{equation}
Here $\sigma_r(dz)$ denotes the surface area on $\partial B(0,r)$. Notice that $\norm{g_2(t,\cdot)}_{L^1(\R^d)} = t$, for $t>0.$ The notion of solution for \eqref{eq:linear_wave2} considered in this paper will be given by a continuous, bounded in time and space Duhamel type solution:

\begin{theorem}[\cite{HLT21}]
	Let $d=1,2,3$. Let $\widetilde{F}$, $f_2$ be bounded continuous functions. Then the problem \eqref{eq:linear_wave2} has a unique solution  $U\in C_b^0([0, \infty)\times\R^d,\R)$ of the Duhamel's representation
	\begin{equation}\label{eq:Duhamel}
		U(t,x)=  (g_2(t,\cdot ) * f_2)(x)  + \int_0^{t} (\widetilde{F}(t-s,\cdot)*g_2(s,\cdot))(x)ds.
	\end{equation}
\end{theorem}

From now on we shall omit the tilde in $\widetilde F$, and simply denote it $F$. In  \cite{HLT21}, Henry-Labordère and Touzi obtained a probabilistic representation of  solutions to \eqref{eq:Duhamel}. In Subsection \ref{sec:1.3}, and later in Theorem \ref{prop:sol-nonlinear}, we will describe in detail this probabilistic representation, obtained via a branching mechanism, the basis for the approximation by artificial neural networks. 

\section{Representation of waves}\label{sec:2}

\subsection{Duhamel's formula for general wave models} In this section we recall some well-known facts about the existence of Henry-Labordère and Touzi probabilistic solutions to the Duhamel's representation. We closely follow  \cite{HLT21} in the particular case of the wave equation. First of all, we introduce the $\mathbb{C}^2$-valued function $\hat{g}:=(\hat{g}_1,\hat{g}_2)$.
\[
\hat{g}(t,\xi) := (2\pi)^{-\frac d2} e^{tB(\xi)^T}\hbox{e}_1, \qquad t\geq 0,\quad  \xi \in \R^d,\quad B(\xi) := \begin{pmatrix}
	0 & 1\\
	-|\xi|^2 & 0
\end{pmatrix},
\]
where $(\hbox{e}_1,\hbox{e}_2)$ is the canonical basis of $\R^2$. Recall that this problem is well-posed in the sense that for all $t \geq 0$, $\hat g(t,\cdot) \in \mathcal S'$. One can then introduce the so-called Green functions as the inverse Fourier transform with respect to the space variable:
\[
g(t,\cdot) := \mathfrak{F}^{-1}\hat{g}(t,\cdot), \qquad t \in [0,T),
\]
in the distributional sense. In this setting, one has for $(r,z) \in \R_+ \times \R^d$,
\[
\begin{aligned}
g_1(r,dz) = (2\pi)^{-d} \int \cos(r|\xi|) e^{i\xi \cdot z}d\xi = \partial_r g_2(r,dz),\quad	g_2(r,dz) = (2\pi)^{-d} \int \frac{\sin(r|\xi|)}{|\xi|} e^{i\xi\cdot z} d\xi.
\end{aligned}
\]
The value of $g_2$ is already given in \eqref{eq:g2}. Notice that in this case, by the form of the Green functions that $g_1(t,\cdot) \notin L^1(\R)$, even  when $d=1$. For this reason \eqref{eq:linear_wave2} is introduced, in order to use only the Green function $g_2$. 

\medskip

Some additional notation is required. Let $\mu_2(t,\cdot)$ be given by
	\begin{equation}\label{eq:mu2}
	\mu_2(t,dz) := \frac{g_2(t,dz)}{t}, \quad t>0, \quad \hbox{$g_2$ as in \eqref{eq:g2}}.
	\end{equation}
	Notice that $\mu_2(t,\cdot)$ defines a probability measure on $\R^d$. Let $Z_{t,2}$ be the random variable whose density is $\mu_2(t,\cdot)$, namely,
	\begin{equation}\label{Zt2}
	\PP \left( Z_{t,2} \in dz \right) = \mu_2(t,dz).
	\end{equation}
	The random variable $Z_{t,2}$ satisfying \eqref{Zt2} exists in dimension $d=1,2,3$, see Remark \ref{rem:Z}. Finally, for $x \in \R^d$ define 
	\begin{equation}\label{Xt2}
	X_{t,2}:= x + Z_{t,2},
	\end{equation}
	chosen independent of $\tau$.

\begin{remark}[On the existence of $Z_{t,2}$]\label{rem:Z} From a change of variable $z \to \frac{z}{t}$ in the definition of $g_2(t,dz)$ in \eqref{eq:g2}, the random variable $Z_{t,2}$ has the form $Z_{t,2}:= tZ$, where:
		\begin{itemize}
			\item in dimension $d=1$, $Z$ has a uniform distribution on $[-1,1]$,
			\item in dimension $d=2$, the law of $Z$ is defined by the density $\frac{1}{2\pi}\frac{1}{\sqrt{1-z^2}}1_{ \{ |z|<1 \} }$,
			\item in dimension $d=3$, the law of $Z$ is $\frac{1}{4\pi}\mu_{S^2}(dz)$, where $\mu_{S^2}$ denotes the volume measure on the unit sphere.
		\end{itemize}
 In the three cases $Z$ takes values in the centered unit radius closed ball $\overline{B}(0,1)$, centered unit radius open ball $B(0,1)$ and centered unit radius sphere $\partial B(0,1)$ respectively, and therefore for $d=1,2,3$ we have that $|Z|\leq 1$.
\end{remark}

%
Following \cite{HLT21}, $g_2$ is part of a more general framework, that we explicit here in view of forthcoming extensions of this work.

\begin{lemma}
	Fix $T>0$. For $d=1,2,3,$ the following are satisfied:
	\begin{itemize}
		\item[(i)] $(t,x) \longmapsto (g_2(t,\cdot)*\phi)(x)$ is continuous on $[0,T)\times\R^d$, for all bounded continuous function $\phi$ on $\R^d$.
		\item[(ii)] $g_2(t,\cdot)$ may be represented by a signed measure 
		\[
		g_2(t,dx)=g_2^+(t,dx) - g_2^-(t,dx),
		\] 
		with total variation measure $|g_2|:=g_2^+ + g_2^-$, absolutely continuous with respect to some probability measure $\mu_2$. 
		\item[(iii)] Additionally, the corresponding densities $\gamma_2,\gamma_2^+$ and $\gamma_2^-$, defined by
		\[
		\begin{aligned}
		g_2^{+}(t,dx) = &~{} \gamma_2^+(t,x)\mu_2(t,dx), \\
		g_2^-(t,dx) = &~{} \gamma_2^-(t,x)\mu_2(t,dx).
		\end{aligned}
		\]
		\item[(iv)]	$\gamma_2  :=   \gamma_2^+ - \gamma_2^{-}$ satisfies, for all $t \in [0,T]$,
		\[
		\|\gamma_2(t,\cdot)\|_{\infty}<\infty, \quad \gamma_2(t,\R^d)<\infty, \quad  \hbox{and} \quad \gamma_2(\cdot,\R^d) \in L^1([0,t]).
		\]
	\end{itemize}
\end{lemma}

\begin{proof}
(i) and (ii) are a consequence of \eqref{eq:g2}. Finally, (iii) follows from \eqref{eq:mu2} and $\gamma_2(t,x) = t$. 
\end{proof}

\subsection{Probabilistic representation of linear wave models} \label{sec:1.3}
For a probabilistic representation of the solution of \eqref{eq:linear_wave2} we need to introduce some random variables{\color{black}. Following \cite{HLT21}}, let $(\Omega,\mathcal{F},\PP)$ be a probability space \emph{supporting two given random variables $\tau$ and $X_{t,2}$, with the additional requirement:}

\begin{assumptions}[Assumption on $\tau$] One has that  
	\begin{equation}\label{def_rho}
	\PP(\tau \in dt) = \rho(t) 1_{t\geq 0},
	\end{equation}
	for some $C^0(\R_+,\R)$ density function $\rho>0$ on $(0,\infty)$. We shall denote 
	\begin{equation}\label{rho_bar}
	\overline{\rho}(t):= \int_t^{\infty} \rho(s)ds.
	\end{equation}
\end{assumptions}

The next result gives the probabilistic representation of the solution of \eqref{eq:linear_wave2}. Again, for a general form of this theorem, the reader can consult \cite{HLT21}.

\begin{theorem}\label{prop:sol-linear}
	Let $f_2$, $\widetilde{F}$ be bounded continuous functions. Then the unique $C_b^0([0,{\color{black}\infty)}\times \R^d,\R)$ solution of the problem \eqref{eq:Duhamel} is given by
	\begin{equation}\label{eq:sol-linear}
	U(t,x) = \E\left[1_{\{\tau\geq t\}} \frac{t}{\overline{\rho}(t)} f_2(X_{t,2}) + 1_{\{\tau<t\}} \frac{\tau}{\rho(\tau)} \widetilde{F}(t-\tau,X_{\tau,2})\right], \quad t<{\color{black}\infty}, \quad x \in \R^d.
	\end{equation}
\end{theorem}

\begin{remark}\label{rem:sol-linear}
The representation \eqref{eq:sol-linear} is only valid when $\widetilde F$ does not depend on $U$. For the representation in the non-linear case, see Theorem \ref{prop:sol-nonlinear} bellow.
\end{remark}

\begin{remark}
To obtain \eqref{eq:sol-linear} one uses from Duhamel's representation \eqref{eq:Duhamel} that $g_2(t,dz) = t \mu_2(t,dz)$, and the convolution has been considered as the expected value of a function depending on the random variable $X_{t,2}$. The time-dependent integral that appears on \eqref{eq:Duhamel} has been considered as the expected value of a function depending on the random variable $\tau$.
\end{remark}

\begin{remark}[On the meaning of \eqref{eq:sol-linear}] \label{rem:Mean-SL} The random variables $\tau$ and $(X_{t,2})_{t\geq 0}$ in \eqref{Xt2} can be seen as the life-length and the trajectory of a simulated particle that was born at time $0$, and starting at the position $x$. This particle will be assigned a value under the action of the Duhamel's wave dynamics. The solution $U(t,x)$ in \eqref{eq:sol-linear} will reflect the dynamics of the particle in the sense that $U$ at time $t$ and position $x$ will be the expected value of the assigned value of such a particle that was born at time $0$ at the position $x$. Depending on the survival time of the particle within the time interval $[0,t]$, the assigned value will take two different options:

\medskip

\begin{enumerate}

\item If the particle is alive in time $t$ (i.e. $\tau \geq t$), then stays at the position $X_{t,2}$ at time $t$. Then we will assign to the particle the value $\frac{t}{\overline \rho(t)} f_2(X_{t,2})$. The wave dynamics is included in the initial condition $f_2$, and in the Green function $g_2$ at time $t$, represented as the law of $X_{t,2}$ and the magnitude $t$, followed by \eqref{eq:mu2}.

\medskip

\item If the particle did not survive by time $t$ (i.e. $\tau < t$), then it is at the position $X_{\tau,2}$ at time $\tau$. The particle will be assigned the value $\frac{\tau}{\rho(\tau)} \widetilde F(t-\tau,X_{\tau,2})$. In this case, the dynamics is given by the source term $\widetilde F$ at time $t-\tau$ and by the Green function $g_2$ at time $\tau$, represented as the law of $X_{\tau,2}$ and the magnitude $\tau$.
\end{enumerate}
\end{remark}
Later, in Theorem \ref{thrm:main_linear}, we provide a DNN approximation result in the case of linear waves \eqref{eq:sol-linear}, that serves as first approach to the more demanding nonlinear case.

\section{Nonlinear waves: the branching mechanism}\label{Sec:2b}

 As we said in Remark \ref{rem:sol-linear}, Theorem \ref{prop:sol-linear} is not valid when $\widetilde F = \widetilde F(t,x,U)$ or even $\widetilde F = \widetilde F(t,x,U,\nabla U)$. In this section we present the results of Herny-Labordère and Touzi related to a non-linear source term of the form
\begin{equation}\label{eq:source-nonlinear}
F(t,x,u) := \sum_{j=0}^{\infty} q_jc_j(t,x)u^j(t,x),
\end{equation}
where for each $j\in \N$, $c_j(t,x):(0,\infty)\times \R^d \to \R$ is a bounded continuous function, and for all $j \geq 0$ one has 
\begin{equation}\label{q_j}
q_j \geq 0, \quad \sum_{j\geq 0} q_j = 1.
\end{equation}
 The series defining $F$ in \eqref{eq:source-nonlinear} may not be convergent in general. For that reason, and following \cite{HLT21}, we shall assume the following conditions ensuring a well-defined $F$:

\begin{assumptions}\label{OnF}
The series defining the nonlinearity $F$ in \eqref{eq:source-nonlinear} is well-defined, in the following sense: the series $H_1(s)$, defined as
\[
H_1(s) := \sum_{j=0}^{\infty} q_j \norm{c_j}_{\infty} s^j, \qquad \norm{c_j}_{\infty}:= \sup_{(t,x)\in (0,\infty)\times \R^d} |c_j(t,x)|,
\]
has strictly positive radius of convergence $R_1 \in (0,\infty]$. In addition, it is required that
\begin{enumerate}
\item $0<r_1 := \norm{f_2}_{\infty} \norm{\gamma_2}_{\infty} < R_1$.
\item There exist constants $T,s_1,r_1>0$, $r_1<s_1$, such that 
\[
 \frac1T\int_{r_1}^{s_1} \frac{ds}{H_1(s)} = \norm{\gamma_2}_{\infty}.	
\]
\end{enumerate}
\end{assumptions}
The last assumption ensures that $F$ in \eqref{eq:source-nonlinear} is nontrivial.

\medskip

All the cases considered in this work (Theorem \ref{MT1}) will satisfy these hypotheses, but recall that more general nonlinearities may be also allowed. For the statement of Theorem \ref{prop:sol-nonlinear}, describing a stochastic representation of nonlinear wave models, thanks to \eqref{q_j} we first introduce a random variable $J$ defined by its probabilities $\PP(J=j) = q_j$. Then we consider the so-called \textit{branching mechanism}:

\begin{definition}[Branching mechanism, see \cite{HLT21}]\label{BM}
Let $t\geq 0$, $k\in \cup_{n\in \N}\N^n$, $T^t_k\geq 0$ and $X_{T_k^t}^k \in\R^d$ be such that the following are satisfied:
\begin{enumerate}
\item Start with a particle indexed by 0. Let $\tau_0$ be an i.i.d copy of $\tau$ and let $J_0$ be an i.i.d. copy of $J$. If ($a\wedge b := \min\{a,b\}$)
\[
T^t_0:=\tau_0 \wedge t<t,
\]
 then the particle 0 branches into $J_0$ offspring particles indexed by $(0,1),\ldots,(0,J_0)$.

 \medskip

\item For any particle indexed by an $n$-tuple $k$ we do the same framework. First denote $k_{-}$ as {\bf the ancestor of $k$}, which is just the vector $k$ after deleting the last component. Let $\tau_k$ and $J_k$ i.i.d. copies of $\tau$ and $J$, respectively. Define 
\begin{equation}\label{eq:Ttk}
T^t_k := (T^t_{k_{-}} + \tau_k) \wedge t, \qquad T_{0_{-}}^t:= 0.
\end{equation}
If $T^t_k < t$, then the particle $k$ branches into $J_k$ offspring particles indexed by the $(n+1)$-tuples $(k,1),\ldots,(k,J_k)$.

\medskip

\item Denote by $K_t$ the set of particles $k$ such that $T_{k}^t = t$, and $\overline K_t$ be the set of particles $k$ such that $T_{k}^t \leq t$.

\medskip

\item For any particle $k$ define $Z_{T_k^t}^k$ as the random variable with conditional probability
\begin{equation}\label{eq:DTk}
\PP \left( Z_{T_k^t}^k \in dz \Big| \Delta T_k^t \right) = \mu_2\left( \Delta T_k^t,dz \right), \hspace{.5cm} \hbox{where} \hspace{.5cm} \Delta T_k^t := T_k^t - T_{k_-}^t.
\end{equation}
Finally define
\begin{equation}\label{eq:XTk}
X_0^0=x, \qquad X_{T_{0_{-}}}^{0_{-}}:= 0, \qquad  X_{T_k^t}^k := X^{k_{-}}_{T_{k_{-}}^t} + Z_{T_{k}^t}^t.
\end{equation}
\end{enumerate}
\end{definition}
The solution to the nonlinear wave equation in terms of a branching process is given by the following result.

\begin{theorem}[Theorem 3.2 in \cite{HLT21}]\label{prop:sol-nonlinear} Let $f_2 \in C_b^0(\R^d,\R)$ and $c_j \in C_b^0((0,\infty)\times \R^d,\R)$, $T^t_k$ as in \eqref{eq:Ttk}, $\Delta_k^t$ as in \eqref{eq:DTk} and $X_{T_k^t}^k$ given by \eqref{eq:XTk}. Then, under Assumptions \ref{OnF}, one has that
\begin{equation}\label{eq:sol-nonlinear}
U(t,x) = \E \left[ \prod_{k \in K_t} \frac{\Delta T_k^t}{\overline \rho(\Delta T_k^t)}f_2\left( X_{T_k^t}^k\right) \prod_{k \in \overline K_t \setminus K_t} \frac{\Delta T_k^t}{\rho(\Delta T_k^t)} c_{J_k}\left(t - T_k^t, X_{T_k^t}^k\right)\right],
\end{equation}
is a well-defined $C_b^0([0,\infty)\times \R^d,\R)$ function solving the non-linear wave equation in its Duhamel form \eqref{eq:Duhamel}, 
with source term $\widetilde F=F$ as in \eqref{eq:source-nonlinear}.
\end{theorem}

Some remarks are in order:

\begin{remark}[On the meaning of \eqref{eq:sol-nonlinear}] The identity \eqref{eq:sol-nonlinear} looks more difficult to understand than the one appearing in the linear case \eqref{eq:sol-linear}. However, its meaning follows a similar idea. Indeed, as in the linear case, and given the branching mechanism, one can see the random variables $\tau_k$ and $X_{T_k^t}^k$ as the life-length and the position at time $T_k^t$ of a particle (with label $k$) that was born at time $T_{k_-}^t$ starting at position $X_{T_{k_-}^t}^{k_-}$. Here, $k_-$ indicates the parent of $k$. Every particle will carry a value depending on whether or not survives at time $t$, similar as in Remark \ref{rem:Mean-SL}. The solution $U(t,x)$ in \eqref{eq:sol-nonlinear} will have a relation with all particles that were born up to time $t$, in the sense that $U$ at time $t$ and position $x$ will be the expected value of the multiplication of the assigned value of all the particles. Every particle $k$ will have two different options:

\medskip

\begin{enumerate}
\item If $k$ is alive in time $t$, then it was living for a time $t-T_{k_-}^t$ within the interval $[0,t]$. Therefore, the particle will carry the value $\frac{t-T_{k_-}^t}{\overline \rho(t-T_{k_-}^t)}f_2 \left(X_{T_k^t}^k \right)$, where we have considered the life time of the particle in $[0,t]$, and its position at time $T_k^t = t$. This is precisely the first term in the RHS of \eqref{eq:sol-nonlinear}.

\medskip

\item If $k$ died before $t$, then it lived for a time $\tau_k$ within the interval $[0,t]$. Therefore, $k$ will carry the value $\frac{\tau_k}{\rho(\tau_k)} c_{J_k}(t-T_{k}^t,X_{T_k^t}^k)$, $c_j$  as in \eqref{eq:source-nonlinear}. In this case we have also considered the life time of $k$ in $[0,t]$, and its position at time $T_k^t < t$. In addition, we have considered the time $t-T_{k}^t$ where the particle is not alive inside the interval of time $[0,t]$. This is precisely the second term in the RHS of \eqref{eq:sol-nonlinear}.
\end{enumerate}

\medskip
As said before, the assigned value of each particle $k$ is similar to the behavior of the simulated particle in the linear case, with the difference that each $k$ (excepting the first particle) will be born at a random time $T_{k_-}^t$ and it will start at a position $X_{T_{k_-}^t}^{k}$ (also randomly chosen). Finally, recall that every particle $k$ will take a value determined by exactly one of the terms that define the source term $F$, depending on how many particles branch from $k$.
\end{remark}

\begin{remark} If one wants to recover the formula obtained in the linear case \eqref{eq:sol-linear}, \eqref{eq:sol-nonlinear} becomes 
\[
U(t,x) = \E\left[ \prod_{k \in \overline K_t} \left({\bf 1}_{\{k \in K_t\}} \frac{\Delta T_{k}^t}{\overline \rho (T_k^t)} f_2 \left(X_{T_k^t}^k\right) + {\bf 1}_{\{k \notin K_t\}} \frac{\Delta T_k^t}{\rho (T_k^t)} c_{J_k}\left(t-T_k^t,X_{T_k^t}^k\right)\right)\right].
\]
In this new representation, one can see that every particle $k$ alive at time $\leq t$ (the set $\overline K_t$) splits into two disjoint sets, with one particular outcome for each case.
\end{remark}

In order to know how to better compute and/or estimate the term \eqref{eq:sol-nonlinear}, it is interesting to consider two particular but very important cases:

\begin{remark}[Two important cases of $F$]\label{rem:special_cases}
For an arbitrary source term $F$ of the form \eqref{eq:source-nonlinear}, the solution $U$ in \eqref{eq:sol-nonlinear} may be not easy to work with, principally because one has arbitrary multiplications over the random sets $K_t$ and $\overline K_t \setminus K_t$. In order to gain some insights, we shall assume first that $f_1$ is harmonic, such that $\widetilde F= F$, and we present two different cases of \eqref{eq:sol-nonlinear} with natural simplifications:

\medskip

\begin{enumerate}
\item If $F(t,x,u) = c(t,x) u(t,x)$, a particle of the branching mechanism can only branch into one offspring particle. This implies that cardinal of $K_t$ is just $|K_t|=1$. Then, \eqref{eq:sol-nonlinear} takes the simplified form 
\[
U(t,x) = \E\left[\frac{\Delta T_N^t}{\overline \rho(\Delta T_N^t)}f(X_{T_N^t}^N)\prod_{n=0}^{N-1} \frac{\Delta T_n^t}{\rho(\Delta T_n^t)} c(t-T_n^t,X_{T_n^t}^n)  \right],
\]
where the indexes $k$ can be taken as nonnegative integers $n$, instead of $\N$-valued vectors, and $N = \min\{n\in\N: T_n^t \geq t\} = \min\{n\in\N:T_n^t = t\}$ is a nonnegative integer-valued random-variable easier to estimate than in the general case.

\medskip

\item If $F(t,x,u)=c(t,x)u^2(t,x)$, then a particle can only branch into two offspring particles. Let $N$ be the random variable that counts the number of particles on $K_t$. It is not difficult to see that the number of particles on $\overline K_t \setminus K_t$ is exactly $N-1$. Then \eqref{eq:sol-nonlinear} takes the simplified form
\[
U(t,x) = \E\left[\prod_{i=1}^{N}{\bf 1}_{\{k_i \in K_t\}}\frac{\Delta T_{k_i}^t}{\overline \rho\left(\Delta T_{k_t}^t\right)}f\left(X_{T_{k_i}^t}^{k_i}\right)\prod_{j=1}^{N-1}{\bf 1}_{\{\overline k_j \notin K_t \}} \frac{\Delta T_{\overline k_j}^t}{\rho\left(\Delta T_{\overline k_j}^t\right)} c\left(t-T_{\overline k_j}^t,X_{T_{\overline k_j}^t}^{\overline k_j}\right)  \right].
\]
\end{enumerate}
\end{remark}
Now that we have a suitable understanding of the branching mechanism, we are ready to state our main results.

\section{Main results} 

Consider the full framework introduced in Sections \ref{sec:2} and \ref{Sec:2b}. Before stating the main results of this paper, we need to state some hypotheses on the approximation of the functions $f$ and $c$ via deep neural networks. 

\begin{assumptions}\label{ass2} Fix a dimension $d \in \{1,2,3\}$ and let $f,F$ be bounded continuous functions such that Theorem \ref{prop:sol-nonlinear} is satisfied. Fix $F$ and $p$ be
\[
F(t,x,u): =c(t,x) u^p, \quad p=0,1,2,3,\ldots
\]
with $c$ nontrivial. Let $B,T,\beta\geq 0$, $p,r,q \in \N$, $q\geq 2$, and $\alpha \geq 2$. For every $\varepsilon \in (0,1]$ and $t \in [0,T]$ let $\Phi_{f,d,t,\varepsilon}, \Phi_{c,d,t,\varepsilon} \in \emph{\textbf{N}}$ be deep neural networks (see Section \ref{sec:3} for definitions) satisfying the following conditions:
\smallskip
\begin{enumerate}
\item[(a)] Continuous realizations $\mathcal R$ of DNNs for the initial data and model coefficients. For all $x \in \mathbb R^d$ and $t \in [0,T]$,
\[
\mathcal R(\Phi_{f,d,t,\varepsilon}) \in C(\R^d,\R), \qquad \mathcal R(\Phi_{c,d,t,\varepsilon}) \in C\left([0,T] \times \R^d,\R\right).
\]
\item[(b)] Local DNN approximation: for all $|x|\leq 2T$ and $s \in [0,T]$,
	\be\label{eq:approx_f1}
	|f(x)-\mathcal{R}(\Phi_{f,d,t,\varepsilon})(x)| \leq \varepsilon,
	\ee
	and,
	\be\label{eq:approx_c1}
	|c(s,x)-\mathcal{R}(\Phi_{c,d,t,\varepsilon})(s,x)| \leq \varepsilon.
	\ee
\item[(c)] A priori estimates on the number of parameters $\mathcal{P}$ of the DNN approximation of the initial data (Definition \ref{def:PDHW}):
	\be\label{eq:cota_param_phiu_nl}
	\mathcal{P}(\Phi_{f,d,t,\varepsilon}) \leq Bd^p \varepsilon^{-\alpha}, \hspace{1cm} \mathcal{P}(\Phi_{c,d,t,\varepsilon}) \leq Bd^p \varepsilon^{-\alpha}.
	\ee
\item[(d)] Bounded hidden layers. One has the following priori estimates on the number of hidden layers $\mathcal{H}$ of the DNN approximation of the initial data (Definition \ref{def:PDHW}):
	\be\label{eq:dim_dnn_encontradas_nl}
	\mathcal{H}(\Phi_{f,d,t,\varepsilon}) \leq Bd^p \varepsilon^{-\beta}, \hspace{1cm} \mathcal{H}(\Phi_{c,d,t,\varepsilon}) \leq Bd^p \varepsilon^{-\beta}.
	\ee
\item[(e)] No blow-up in finite time.  There exists $\delta_0>0$, such that either $\|f\|_{L^\infty(\mathbb R^d)} +\|c\|_{L^\infty(0,T,\R^d)} < \delta_0$, or $0<T<\delta_0$.
\end{enumerate}
\end{assumptions}

\begin{remark}
By assuming that $f$ and $c$ are bounded continuous, and using item $(b)$ of previous assumptions, one has the following bounds for all $t \in [0,T]$:
\[\norm{\mathcal R(\Phi_{f,d,t,\varepsilon})}_{\infty} \leq \norm{f}_{\infty} + 1, \qquad \norm{\mathcal R(\Phi_{c,d,t,\varepsilon})}_{\infty}\leq \norm{c}_{\infty} + 1.
\]
Additionally, notice that we only require approximation near the so-called light cone of the solution. 
\end{remark}

\begin{remark}
Assumption $(e)$ above is natural in view that $T>0$ will be a fixed but arbitrary time and NLW with quadratic or higher nonlinearities may have singularity formation if the data is large enough. Additionally, the power-like character ensures real-analyticity of the nonlinearity, avoiding some cases where even small data solutions blow-up in finite time.
\end{remark}

\begin{theorem}[Universal approximation]\label{MT1}
Let $U=U(t,x)$, $x\in \mathbb R^d$, $t\in [0,T]$ be given by \eqref{eq:sol-nonlinear}. Under Assumptions \ref{ass2}, the following is satisfied. For all $d =1,2,3$, $t \in [0,T]$, $\varepsilon \in (0,1]$ there exists a DNN $\Phi_{d,t,\varepsilon} \in \textbf{N}$, and constants $\widetilde{B}:=\widetilde{B}(B,T,\alpha,\beta)>0$, $\eta:=\eta(p,\alpha,\beta)>0$  and $\gamma:=\gamma(p,\alpha,\beta)>0$, such that the following are satisfied:
	
\begin{enumerate}
\item	 Continuous dependence and arbitrary approximation inside the light-cone: one has $\mathcal{R}(\Phi_{d,t,\varepsilon}) \in C(\mathbb R^d, \R)$ and
	\be\label{eq:main_eps_nl}
	\| U(t) - \mathcal{R}(\Phi_{d,t,\varepsilon})\|_{L^\infty(B(0,t))} 
	\leq \varepsilon.
	\ee
\item Moreover, one has the parameter estimates:
	\be\label{eq:main_param_nl}
	\mathcal{P}(\Phi_{d,t,\varepsilon}) \leq \widetilde{B}d^{\eta} \varepsilon^{-\eta}, \quad \mathcal{H}(\Phi_{d,t,\varepsilon}) \leq \widetilde{B} d^{\gamma} \varepsilon^{-\gamma}.
	\ee
\end{enumerate}
In particular, the solution $U$ does not blow up in finite time in the region $|x|<t$, $0\leq t\leq T$.
\end{theorem}

It is interesting to notice that one has control inside the light cone $|x|<t$. This is a natural property satisfied by classical waves. Mathematically, it is unlikely to have full approximation with control in space unless some particular spatial weights are used. In our case, we are able to discard weights (see Theorem \ref{thrm:main_linear} of a proof using that technique) and obtain full control inside the light cone under the condition $T$ finite.    

\medskip

To prove Theorem \ref{MT1}, we shall follow the following program: in Section \ref{sec:3} we introduce and describe the Deep Neural Network results needed for the proof of the main results. We need new estimates for the multiplication of functions (Lemmas \ref{lemma:mult_DNNs} and \ref{lemma:DNN_mult_k}, and Corollary \ref{lemma:mult_dnn_final}). Although somehow standard, these results are strongly necessary due to the product representation in branching processes, and quantify in better terms the product of DNNs and their approximation by a new DNN. 

\medskip

Later, in Section \ref{sec:4}, we start the proof of Theorem \ref{MT1} by considering first the case $p=0$, that is, the linear case, where the simpler representation \eqref{eq:sol-linear} holds. Notice that this case differs from the nonlinear one and the representation is different. Theorem \ref{thrm:main_linear} is the main result in this case. 

\medskip

Later, in Section \ref{sec:5}, we consider the proof of Theorem \ref{MT1} in the case $p=1$, the linear perturbative case, that we believe is deeply necessary to better understand the proof in the general case $p\geq 2$. Finally, Section \ref{sec:6} contains the proof of Theorem \ref{MT1} in the truly nonlinear case, $p\geq 2$.

\medskip

This work leaves several more than interesting open questions. First of all, due to the extension of this paper, the numerical implementation of the proposed method will be done elsewhere. Several works performing numerical approximation of Montecarlo motivated DNN approximations are present in the literature, see e.g. \cite{Val23} for the case of the fractional Laplacian, an operator with less strong ``oscillatory behavior'' compared with waves. Additionally, considering more general nonlinearities requires better probabilistic counting techniques in the case of branching mechanisms, a process that we expect to obtain in future works.

\medskip

Even if the dimensions $d=1,2,3$ are widely regarded as the most prominent physical dimensions where one studies waves, the main drawback of this work is the fact that dimensions larger than 4 do require additional developments, in view that the fundamental solution of linear waves does not enjoy enough integrability properties as required in the probabilist setting. Currently we are exploring this case, extending the framework described in this paper. This step is required to fully comprehend the cost of  approximation in the case of dimensions $d$ large.  

%


\newpage

\appendix

\section*{Supplementary Material}

\section{Deep Neural Networks: A quick review and new results}\label{sec:3}

\subsection{Neural Network setting} In this section we introduce some elemental concepts for the construction of the neural networks and their principal components. There are several definitions on the architecture of neural networks in the literature, see e.g. \cite{Elb21, Gro23,Hutz20}.

\begin{definition}[Neural Network]\label{def:DNN}
Let $H \in \N$, let $k_0,\ldots,k_{H+1} \in \N$. 
\begin{enumerate}
\item A Neural Network $\Phi$ is a sequence of matrices and vectors (called weights and biases, respectively) given by
\begin{equation}\label{def_NN}
\Phi := ((W_i,b_i))_{i=1}^{H+1} \in \prod_{i=1}^{H+1} \left(\R^{k_i \times k_{i-1}} \times \R^{k_i}\right).
\end{equation}
\item We denote $H$ as the number of hidden layers (represented by $1,2,\ldots H$), and $k_0, k_1,k_2,\ldots,k_{H+1} \in \N$ as the deepness (vector size) of each layer. $\mathbb R^{k_0}$ is the input space, and $\mathbb R^{k_{H+1}}$ will be the output space. We will say that the NN is deep if $H$ is large enough, depending on the considered problem.

\item We define the space of all NNs ${\bf N}$ as
\[
{\bf N} := \bigcup_{H \in \N} \bigcup_{(k_0,\ldots,k_{H+1}) \in \N^{H+2}} \prod_{i=1}^{H+1} \left(\R^{k_i \times k_{i-1}} \times \R^{k_i}\right).
\]
\end{enumerate}
\end{definition}

We can naturally define a continuous function from any neural network $\Phi \in {\bf N}$. Next Definition is devoted to the characterization of such a continuous function. First, let $\sigma:\R \to \R$, $\sigma(z) = \max\{0,z\}$ be a ReLU activation function.

\begin{definition}[Realization of a NN] Define 
\[
\begin{matrix}
\mathcal R: &{\bf N} &\longrightarrow &\displaystyle \bigcup_{k,l \in \N} C(\R^{k},\R^{l}) \\
& \Phi & \longmapsto & \mathcal R (\Phi) \equiv A_{H+1} \circ \sigma \circ A_{H} \circ \sigma \circ \ldots \circ \sigma \circ A_1,
\end{matrix}
\]
where for $i=1,\ldots,H+1$, $A_i(x) := W_i x + b_i$ are linear affine transformations, with $\sigma$ acting component-wise. For $\Phi = ((W_i,b_i))_{i=1}^{H+1} \in {\bf N}$ we say that $\mathcal R(\Phi) \in C(\R^{k_0},\R^{k_{H+1}})$ is the realization of $\Phi$.
\end{definition}
Notice that we can construct the realization of $\Phi \in \N$ given any continuous function $\sigma : \R \to \R$ acting component-wise. Although, in this paper we will only focus on the ReLU activation function.  
\begin{definition}[DNN parameters]\label{def:PDHW}
For each $\Phi \in \N$ as in Definition \ref{def:DNN}, we denote
\begin{enumerate}
\item $\mathcal P(\Phi)$ be the number of non-zero entries of the weights and biases of $\Phi$.
\item $\mathcal D(\Phi) = (k_0,\ldots,k_{H+1})$ be the vector of dimensions of $\Phi$.
\item $\mathcal H(\Phi) = H$ be the number of hidden layers.
\item $\mathcal W(\Phi) = \max_{i=0,\ldots,H+1} k_i$ be the width of $\Phi$.
\end{enumerate}
\end{definition}

\begin{remark}
From Definition \ref{def:PDHW} and \eqref{def_NN}, one can easily see that for any $\Phi \in \N$,
\[
\mathcal P(\Phi) \leq (\mathcal H(\Phi) + 1) \mathcal W(\Phi) \left( \mathcal W(\Phi) + 1 \right).
\]
The RHS corresponds to have every possible entry in the NN with a nonzero value.
\end{remark}

\subsection{Review of DNN algebra} Basic operations between neural networks, such as the sum, the composition, among others, have been exhaustive studied in literature \cite{Elb21, Gro23,Hutz20}. In this section we present some Lemmas that we will use along this paper. First we define two operations between vectors.
\begin{definition}\label{def:Op_vec} Let $\displaystyle {\bf D} = \bigcup_{H \in \N} \N^{H+2}$ be the set of all possible space dimensions of a NN.
\begin{enumerate}
\item Let $H_1,H_2 \in \N$. We denote by $\odot : {\bf D} \times {\bf D} \to {\bf D}$ the function defined such that for $\alpha = (\alpha_{0},\ldots,\alpha_{H_1+1}) \in \N^{H_1+2}$ and $\beta = (\beta_0,\ldots,\beta_{H_2+1}) \in \N^{H_2+2}$ satisfies
\begin{equation}\label{suma_o}
\alpha \odot \beta = (\beta_0,\ldots,\beta_{H_2}, \beta_{H_2+1}+\alpha_{0},\alpha_1,\ldots,\alpha_{H_1+1}) \in \N^{H_1+H_2+3}.
\end{equation}
This function will naturally represent the concatenation (or composition) of NNs.
\item Let $H \in \N$. We denote by $\boxplus: {\bf D}\times {\bf D}\to {\bf D}$ the function defined such that for $\alpha = (\alpha_0,\ldots,\alpha_{H+1}) \in \N^{H+2}$ and $\beta=(\beta_0,\ldots,\beta_{H+1}) \in \N^{H+2}$ satisfies
\[
\alpha \boxplus \beta = (\alpha_0,\alpha_1+\beta_1,\ldots ,\alpha_H+\beta_H,\beta_{H+1}) \in \N^{H+2}.
\]
This function will naturally represent the sum of NNs, and it is an associative property.
\item Let $H \in \N$. Define $\mathfrak n_H$ as
\begin{equation}\label{NH}
\mathfrak n_H = (1,\underbrace{2,\ldots,2}_{H-\text{times }},1) \in \N^{H+2}.
\end{equation}
This  already introduced particular vector (see e.g. \cite{Hutz20}), represent the minimum cost of representing the identity via a NN of free parameter $H$ (Lemma \ref{lemma:DNN_id}).
\end{enumerate}
\end{definition}
The advantage of using ReLU DNNs is that the identity function on $\R^d$ can be \emph{exactly represented} as the realization of a neural network with ReLU activation function.
\begin{lemma}[Representation of the identity function on $\R$]\label{lemma:DNN_id}
Let $H \in \N$. Define ${\bf Id}_{\R} : \R \to \R$ be the function such that for any $x \in \R$, ${\bf Id}_{\R}(x) = x$. Then there exists $\Sigma_{H} \in {\bf N}$ such that
\[
\mathcal R(\Sigma_{H}) \equiv {\bf Id}_{\R} \in \mathcal R\big( \{\Phi \in {\bf N}: \mathcal D(\Phi) = \mathfrak n_{H}\} \big).
\]
Moreover, the number of hidden layers, the number of nonzero components of the NN, and the width of the NN are, respectively,
\[
\mathcal H(\Sigma_H) = H, \hspace{1cm} \mathcal P(\Sigma_H) \leq 2(H+1) \hspace{.5cm} \hbox{and} \hspace{.5cm} \mathcal W(\Sigma_H) \leq 2. 
\]
\end{lemma}
\begin{proof} See \cite{Hutz20}.
\end{proof}
An analogous result is also valid for the representation of the identity function in $\R^d$, $d\in \N$, by constructing a DNN $\Sigma_{H,d}$ in a similar way of the construction of $\Sigma_H$, with $\mathcal D(\Sigma_{H,d}) = d\mathfrak n_H$. We will not go into additional details. Other properties about neural networks establish that operations such as the sum and the composition of neural networks are well-defined from the operators in Definition \ref{def:Op_vec}. These results are summarized in the following Lemmas. The reader can also refer to \cite{Hutz20}.

\begin{lemma}[Composition of NNs]\label{lemma:DNN_comp}
Let $H_1,H_2 \in \N$, $\alpha \in \N^{H_1+2}$, $\beta \in \N^{H_2+2}$ with $\alpha_0 = \beta_{H_2+1}$. Let $\Phi_1, \Phi_2 \in {\bf N}$ satisfy $\mathcal D(\Phi_1) = \alpha$ and $\mathcal D(\Phi_2) = \beta$. There exists $\Psi \in {\bf N}$ such that
\[
\mathcal R(\Psi) \equiv \mathcal R(\Phi_1) \circ \mathcal R(\Phi_2) \in \mathcal R(\{\Phi \in {\bf N}: \mathcal D(\Phi) = \alpha \odot \beta\}).
\]
Moreover, one has (see Definition \ref{def:PDHW})
\[
\begin{aligned}
\mathcal H(\Psi) = &~{} H_1 + H_2 + 1,\\ 
\mathcal P(\Psi) \leq &~{} 2\mathcal P(\Phi_1) + 2\mathcal P(\Phi_2), \hspace{.3cm} \hbox{and}\\
\mathcal W(\Phi) \leq &~{} \max\left\{2\alpha_0,\mathcal W(\Phi_1),\mathcal W(\Phi_2)\right\}.
\end{aligned}
\]
\end{lemma}
Notice that the result above ensures that $\mathcal R(\Phi_1) \circ \mathcal R(\Phi_2)$ is indeed the exact realization of a NN. This is possible thanks to the particular sizes required as hypotheses in Lemma \ref{lemma:DNN_comp}. Now we describe a result for sums of NNs.

\begin{lemma}[Sum of NNs with the same length]\label{lemma:DNN_suma}
Let $M,H,p,q \in \N$, $h_i \in \R$, $\beta_i \in \N^{H+2}$, $\Phi_i \in {\bf N}$ satisfy $\mathcal D(\Phi_i) = \beta_i$, $i=1,\ldots,M$. There exists $\Psi \in {\bf N}$ such that
\[
\mathcal R(\Psi) \equiv \sum_{i=1}^{M} h_i \mathcal R(\Phi_i) \in \mathcal R\left(\left\{\Phi \in {\bf N}: \mathcal D(\Phi) = \underset{i=1}{\overset{M}{\boxplus}}\beta_i\right\}\right).
\] 
Moreover, recalling Definition \ref{def:PDHW} one has
\[
\mathcal H(\Psi) =H \hspace{1cm}\mathcal P(\Psi) \leq \sum_{i=1}^{M} \mathcal P(\Phi_i), \hspace{.5cm} \hbox{and} \hspace{.5cm}\mathcal W(\Psi) \leq \sum_{i=1}^{M}\mathcal W(\Phi_i).
\]
\end{lemma}
Now we present a result of NNs for affine transformations. For a detailed proof, the reader can consult, e.g. \cite{Hutz20}, Lemma 3.7. 	
\begin{lemma}[NNs for affine transformations]\label{lemma:DNN_affine}
Let $d,m \in \N$, $\lambda \in \R$, $b \in \R^d$, $a \in \R^m$, and $\Psi \in {\bf N}$ satisfying $\mathcal R(\Psi) \in C(\R^d,\R^m)$. Then it holds that
\[
\lambda (\mathcal R(\Psi)(\cdot + b) +a) \in \mathcal R(\{\Phi \in {\bf N}: \mathcal D(\Phi) = \mathcal D(\Psi)\}).
\]
\end{lemma}

This result, even if of auxiliar character, it is really useful in the proofs of Theorems \ref{thrm:main_linear} (approximation of linear waves) and \ref{MT1} (approximation of perturbed linear and nonlinear waves).

\medskip

The following result has been introduced in order to be able to sum DNNs with different lengths. In principle, it states that every DNN can extend its hidden layers without changing its realization.

\begin{lemma}[Extension of layers of a given NN]\label{lem:DNN_extension} Let $H_1, H_2 \in \N$ be nontrivial integers such that $H_1 < H_2$. Let $\alpha \in \N^{H_1+2}$ with $\alpha_{H_1+1}=1$, and let $\Psi \in {\bf N}$ satisfy $\mathcal D(\Psi) = \alpha$. Then there exists $\overline \Psi \in {\bf N}$ such that
\[
\mathcal R(\overline \Psi) \equiv \mathcal R(\Psi) \in \mathcal R \Big( \Big\{\Phi \in {\bf N}: \mathcal D(\Phi) =  \mathfrak n_{H_2-H_1-1} \odot \alpha\Big\} \Big).
\]
Moreover, the parameters of $\overline \Psi$ are only slightly worsened:  
\[
\mathcal H(\overline \Psi) = H_2, \hspace{.6cm} \mathcal P(\overline \Psi) \leq 2\mathcal P(\Psi) + 4(H_2-H_1), \hspace{.3cm} \hbox{and} \hspace{.3cm} \mathcal W(\overline \Psi) = \max\{2,\mathcal W(\Psi)\}.
\]
\end{lemma}

\begin{remark}
From \eqref{NH} and \eqref{suma_o} one has 
\[
\begin{aligned}
\mathfrak n_{H_2-H_1-1} \odot \alpha = &~{}  (\alpha_0,\alpha_1,\ldots, \alpha_{H_1},\alpha_{H_1+1}+1, 2,2,\ldots, 2, 1) \\
= &~{} (\alpha_0,\alpha_1,\ldots, \alpha_{H_1},2, 2,2,\ldots, 2, 1),
\end{aligned}
\]
where, in the last vector, the number 2 appears $H_2-H_1$ times.
\end{remark}

\begin{proof}[Proof of Lemma \ref{lem:DNN_extension}]
Lemma \ref{lem:DNN_extension} is a consequence of the identity representation by DNN (Lemma \ref{lemma:DNN_id}) and the composition of DNNs (Lemma \ref{lemma:DNN_comp}).
\end{proof}

\begin{lemma}[Sum of DNNs with different length]\label{lemma:DNN_suma_distH}
	Let $d,H_f,H_g \in \N$ with $H_f < H_g$, let $\alpha \in \N^{H_f+2}$, $\beta \in \N^{H_g+2}$, $f,g \in C(\R^d,\R)$ satisfy that $f \in \mathcal R(\{\Phi \in \textbf{N}: \mathcal D(\Phi)=\alpha \})$, and $g \in \mathcal R(\{\Phi \in \textbf{N}: \mathcal D(\Phi)=\beta \})$. Therefore there exists $\Psi \in \N$ that satisfy $\mathcal R(\Psi) \in C(\R^d,\R)$, $\mathcal R(\Psi) \equiv f+g$ and
	\[
	\mathcal R(\Psi) \in \mathcal R \big(\left\{ \Phi \in \textbf{N}: \mathcal D(\Phi) = (\mathfrak n_{H_g-H_f-1} \odot \alpha) \boxplus \beta \in \N^{H_g+2} \right\}\big).
	\]
	Moreover, $\mathcal H(\Psi) = H_g$, $\mathcal W(\Psi) = \max\{2,\mathcal W(\Phi_f)\} + \mathcal W(\Phi_g)$ and
	\[
	\mathcal P(\Psi) \leq \mathcal P(\Phi_g) + 2\mathcal P(\Phi_f) + 4(H_g-H_f).
	\]
\end{lemma}
\begin{proof}
Lemma \ref{lemma:DNN_suma_distH} follows from Lemmas \ref{lemma:DNN_suma} and \ref{lem:DNN_extension}.
\end{proof}

{\color{black} In addition to the sum and the composition, Neural Networks with the same length of hidden layers can also be parallelized, in the sense that a vector with each component a DNN is also a DNN.

\begin{lemma}[Parallelization of DNNs with the same length]\label{lemma:DNN_para}
Let $H,M \in \N$. For $i=1,\ldots,M$ let $\alpha_i \in \N^{H+2}$ and $\Phi_i \in {\bf N}$ be such that $\mathcal D(\Phi_i) = \alpha_i$. Therefore, there exists $\Psi \in {\bf N}$ such that for all $z = (z_1,\ldots,z_M) \in \R^{\sum_{i=1}^{M} \alpha_{i,0}}$, with $z_i \in \R^{\alpha_{i,0}}$ it satisfy
\[
\mathcal R(\Psi)(z) = (\mathcal R(\Phi_1)(z_1),\ldots,\mathcal R(\Phi_M)(z_M)), \quad 
\mathcal R(\Psi) \in \mathcal R\left(\left\{\Phi \in {\bf N}: \mathcal D(\Phi) = \sum_{i=1}^{M} \alpha_i \right\}\right).
\]
Moreover,
\[
\mathcal H(\Psi) = H, \hspace{1cm} \mathcal P(\Psi) = \sum_{i=1}^{M} \mathcal P(\Phi_i), \hspace{.5cm} \hbox{and} \hspace{.5cm} \mathcal W(\Psi) \leq \sum_{i=1}^{M} \mathcal W(\Phi_i).
\]
\end{lemma}


\begin{remark}
Let $\alpha_0 \in \N$. If $\alpha_{i,0}=\alpha_0$ for all $i=1,\ldots,M$, then one can construct $\Psi \in \N$ satisfying
\[
\begin{aligned}
& \mathcal R(\Psi) \equiv (\mathcal R(\Phi_1)(\cdot),\ldots,\mathcal R(\Phi_M)(\cdot))  \in \mathcal R\left(\left\{\Phi \in {\bf N}: \mathcal D(\Phi) = \left(\alpha_0,\sum_{i=1}^{M} \alpha_{i,1},\ldots,\sum_{i=1}^{M} \alpha_{i,H+1}\right)\right\}\right),
\end{aligned}
\]
and the same conditions on $\mathcal H(\Psi)$, $\mathcal P(\Psi)$ and $\mathcal W(\Psi)$.
\end{remark}

}

{\color{black}
We now verify that if we fix a component in the input layer of the realization of a DNN, the resulting function is still the realization of a DNN. To do so, we first prove the following Lemma
\begin{lemma}\label{lemma:DNN_ext_t}
	Let $t \in \R$, $t \geq 0$. Let $f_t \in C(\R^d,\R^{d+1})$ be the function that satisfies for all $x \in \R^d$, $f_t(x) = (t,x_1,\ldots,x_d)$. Therefore there exists $\Psi_t \in {\bf N}$ such that
	\[
	\mathcal R(\Psi_t) \equiv f_t \in \mathcal R(\{\Phi \in {\bf N}: \mathcal D(\Phi) = (d,2d+1,d+1)\}).
	\]
	Moreover
	\begin{equation}\label{calculo_simple}
	\mathcal H(\Psi_t) = 1, \hspace{1cm} \mathcal P(\Psi_t) = 2(2d+1) \hspace{.5cm} \hbox{and} \hspace{.5cm} \mathcal W(\Psi_t) = 2d+1.
	\end{equation}
\end{lemma}

\begin{proof}
Let $t \in \R$. Define $\Psi_t \in {\bf N}$ as
\[
\begin{aligned}
\Psi_t &= \left(\left(\begin{pmatrix}\vec{0}^T \\ \mathbb I_{\R^d} \\ -\mathbb I_{\R^d} \end{pmatrix},\begin{pmatrix}t \\ 0 \\ \vdots \\ 0\end{pmatrix}\right),\left(\begin{pmatrix}1 & 0 & \cdots & 0 \\ 0 & & & \\ \vdots & \mathbb I_{\R^d} & & -\mathbb I_{\R^d} \\ 0 & & &\end{pmatrix},\vec{0}\right)\right) \\
&\hspace{5cm}\in \left( \R^{(2d+1) \times d} \times \R^{2d+1} \right) \times \left( \R^{(d+1)\times (2d+1)} \times \R^{d+1} \right),
\end{aligned}
\]
where $\mathbb I_{\R^d}$ is the identity matrix. It is not difficult to see that for all $x \in \R^d$,
\[
\mathcal R(\Psi_t)(x) = f_t(x).
\]
Moreover, one easily has $\mathcal H(\Psi_t) = 1$, $\mathcal D(\Psi_t) = (d,2d+1,d+1)$, $\mathcal P(\Psi_t) = 2(2d+1)$ and $\mathcal W(\Psi_t) = 2d+1$, proving \eqref{calculo_simple}.
\end{proof}

\begin{cor}\label{cor:DNN_fixed_comp}
	Let $f : \R^{d+1} \to \R$ be a function such that $f = \mathcal R(\Phi_f) \in C(\R^{d+1},\R)$ for some $\Phi_f \in {\bf N}$. Therefore for all $t \in \R$, $f(t,\cdot) \in C(\R^d,\R)$ and 
	\[
	f(t,\cdot) \in \mathcal R \Big( \Big\{\Phi \in {\bf N}: \mathcal D(\Phi) = \mathcal D(\Phi_f) \odot (d,2d+1,d+1)\Big\} \Big).
	\]
\end{cor}

\begin{proof}
Corollary \ref{cor:DNN_fixed_comp} follows from Lemmas \ref{lemma:DNN_comp} and \ref{lemma:DNN_ext_t}.
\end{proof}

\subsection{Additional contributions in DNN algebra} The remaining Lemmas from this section are devoted to verify that the multiplication of neural networks approximates the multiplication of functions taking values on a ball centered at 0 with radius $r>0$, and even better, the multiplication of bounded neural networks can be approximated with a neural network on the ball $B(0,r)$.

\begin{lemma}\label{lemma:mult_DNNs} Let $r >0$ and $k \in \{2,3,\ldots\}$. Let $f_i : \R^d \to \R$, $i=1,\ldots,k$ be bounded continuous functions. Suppose that for all $\varepsilon \in (0,1)$ there exists $(\Phi_{r,i,\varepsilon})_{i=1}^{k} \subseteq {\bf N}$ such that for all $x \in B(0,r)$
\begin{equation}\label{eq:lem_DNN_mult_hyp}
\left|f_i(x) - \mathcal R(\Phi_{r,i,\varepsilon})(x)\right| \leq \varepsilon.
\end{equation}
Therefore for all $x \in B(0,r)$
\begin{equation}\label{eq:lem_DNN_mult}
\left|\prod_{i=1}^{k} f_i(x) - \prod_{i=1}^k \mathcal R(\Psi_{r,i,\varepsilon})(x)\right| \leq  C_k\varepsilon,
\end{equation}
where 
\begin{equation}\label{constante_fea}
C_k := \sum_{j=0}^{k-1}  \varepsilon^{j} \Bigg(\sum_{\substack{ i_1,i_2,\ldots, i_k \in\{0,1\}\\ i_1+i_2+\cdots +i_{k} = k-1-j} } \norm{f_{1}}_{\infty}^{i_1} \norm{f_{2}}_{\infty}^{i_2} \cdots \norm{f_{k}}_{\infty}^{i_k}\Bigg). 
\end{equation}
\end{lemma}

\begin{proof}[Proof of Lemma \ref{lemma:mult_DNNs}]
We will prove \eqref{eq:lem_DNN_mult} by induction. Let $r >0$ and for $k \in \N$ let $C=C_k$ be as in the statement of the lemma.  
First, notice from assumption \eqref{eq:lem_DNN_mult_hyp} that for all $i = 1,\ldots,k$:
\begin{equation}\label{eq:DNN_mult_cota_R}
|\mathcal R(\Phi_{r,i,\varepsilon})(x)| \leq |f_i(x)| + |f_i(x) - \mathcal R(\Phi_{r,i,\varepsilon})(x)| \leq \norm{f_i}_{\infty} + \varepsilon.
\end{equation}
For $k=2$ we have by triangle inequality that for all $x \in B(0,r)$
\[
\begin{aligned}
|f_1f_2 - \mathcal R(\Phi_{r,1,\varepsilon})\mathcal R(\Phi_{r,2,\varepsilon})|(x) &\leq |f_1 - \mathcal R(\Phi_{r,1,\varepsilon})|(x) |f_2(x)| + |\mathcal R(\Phi_{r,1,\varepsilon})(x)| |f_2 - \mathcal R(\Phi_{r,2,\varepsilon})|(x).
\end{aligned}
\]
Then using \eqref{eq:lem_DNN_mult_hyp} and \eqref{eq:DNN_mult_cota_R} we have
\[
\begin{aligned}
& |f_1(x)f_2(x) - \mathcal R(\Phi_{r,1,\varepsilon})(x)\mathcal R(\Phi_{r,2,\varepsilon})(x)|  \leq \norm{f_2}_{\infty}\varepsilon + (\norm{f_1}_{\infty} +\varepsilon)\varepsilon = C_2 \varepsilon.
\end{aligned}
\]
This proves \eqref{eq:lem_DNN_mult} for $k=2$. In order to elucidate where the constant $C$ comes from, we perform the case $k=3$:
\[
\begin{aligned}
& |f_1(x)f_2(x)f_3(x) - \mathcal R(\Phi_{r,1,\varepsilon})(x)\mathcal R(\Phi_{r,2,\varepsilon})(x)\mathcal R(\Phi_{r,3,\varepsilon})(x)| \\
& \quad \leq |f_1(x) - \mathcal R(\Phi_{r,1,\varepsilon})(x)| |f_2(x) f_3(x)| + | \mathcal R(\Phi_{r,1,\varepsilon})(x)| |f_2(x)-\mathcal R(\Phi_{r,2,\varepsilon})(x)|  |f_3(x)| \\
& \qquad + | \mathcal R(\Phi_{r,1,\varepsilon})(x)| |\mathcal R(\Phi_{r,2,\varepsilon})(x)|  |f_3(x)-\mathcal R(\Phi_{r,3,\varepsilon})(x)| \\
& \quad \leq \big(\|f_2\|_\infty\|f_3\|_\infty + \|f_3\|_\infty (\varepsilon + \|f_1\|_\infty) +(\varepsilon + \|f_1\|_\infty) (\varepsilon + \|f_2\|_\infty)  \big) \varepsilon \\
& \quad \leq  \big( \|f_2\|_\infty\|f_3\|_\infty + \|f_1\|_\infty\|f_3\|_\infty +\|f_1\|_\infty\|f_2\|_\infty  + \varepsilon (\|f_1\|_\infty + \|f_2\|_\infty + \|f_3\|_\infty ) + \varepsilon^2 \big) \varepsilon\\
& \quad = \varepsilon \sum_{j=0}^2  \varepsilon^{j} \Bigg(\sum_{\substack{ i_1,i_2,i_3 \in\{0,1\}\\ i_1+i_2 +i_3= 2-j} } \norm{f_{1}}_{\infty}^{i_1} \norm{f_{2}}_{\infty}^{i_2}  \norm{f_{3}}_{\infty}^{i_3}  \Bigg) .
\end{aligned}
\]
Now suppose that \eqref{eq:lem_DNN_mult} is satisfied for $k >2$. We have
\[
\begin{aligned}
\left|\prod_{i=1}^{k+1} f_i(x) - \prod_{i=1}^{k+1}\mathcal R(\Phi_{r,i,\varepsilon})(x)\right| &\leq \left|\prod_{i=1}^{k} f_i(x) - \prod_{i=1}^{k}\mathcal R(\Phi_{r,i,\varepsilon})(x)\right| |f_{k+1}(x)| \\
&\hspace{.5cm}+ \left|\prod_{i=1}^{k}\mathcal R(\Phi_{r,i,\varepsilon})(x)\right| |f_{k+1}(x) - \mathcal R(\Phi_{r,k+1,\varepsilon})(x)|.
\end{aligned}
\]
From \eqref{eq:DNN_mult_cota_R}, for all $x \in B(0,r)$ we have
\[
\left|\prod_{i=1}^{k} \mathcal R(\Phi_{r,i,\varepsilon})(x)\right| \leq \prod_{i=1}^{k} \left( \norm{f_i}_{\infty} + \varepsilon \right).
\]
Then, using the inductive hypothesis and \eqref{eq:lem_DNN_mult_hyp},
\[
\begin{aligned}
& \left|\prod_{i=1}^{k+1} f_i(x) - \prod_{i=1}^{k+1}\mathcal R(\Phi_{r,i,\varepsilon})(x)\right| \\
& \quad \leq \varepsilon \Bigg(  \norm{f_{k+1}}_{\infty} \sum_{j=0}^{k-1}  \varepsilon^{j} \Bigg(\sum_{\substack{ i_1,i_2,\ldots, i_k \in\{0,1\}\\ i_1+i_2+\cdots +i_k = k-1-j} } \norm{f_{1}}_{\infty}^{i_1} \norm{f_{2}}_{\infty}^{i_2} \cdots \norm{f_{k}}_{\infty}^{i_k}\Bigg) + \prod_{i=1}^{k} \left( \norm{f_i}_{\infty} + \varepsilon \right)  \Bigg) \\
& \quad = \varepsilon  \Bigg(  \norm{f_{k+1}}_{\infty} \sum_{j=0}^{k-1}  \varepsilon^{j} \Bigg(\sum_{\substack{ i_1,i_2,\ldots, i_k \in\{0,1\}\\ i_1+i_2+\cdots +i_k = k-1-j} } \norm{f_{1}}_{\infty}^{i_1} \norm{f_{2}}_{\infty}^{i_2} \cdots \norm{f_{k}}_{\infty}^{i_k}\Bigg) \\
& \hspace{1.5cm} + \sum_{j=0}^k  \varepsilon^{j} \Bigg(\sum_{\substack{ i_1,i_2,\ldots, i_k \in\{0,1\}\\ i_1+i_2+\cdots +i_k = k-j} } \norm{f_{1}}_{\infty}^{i_1} \norm{f_{2}}_{\infty}^{i_2} \cdots \norm{f_{k}}_{\infty}^{i_k}\Bigg) \Bigg) \\
 & \quad \leq  \varepsilon \sum_{j=0}^{k}  \varepsilon^{j} \Bigg(\sum_{\substack{ i_1,i_2,\ldots, i_k \in\{0,1\}\\ i_1+i_2+\cdots +i_{k+1} = k-j} } \norm{f_{1}}_{\infty}^{i_1} \norm{f_{2}}_{\infty}^{i_2} \cdots \norm{f_{k}}_{\infty}^{i_k}\norm{f_{k+1}}_{\infty}^{i_{k+1}}\Bigg).
\end{aligned}
\]
Therefore \eqref{eq:lem_DNN_mult} is valid for all $k \in \N$.
\end{proof}

{\color{black}
\begin{remark}[Special case for $C_k$]
The constant \eqref{constante_fea} can can be expressed in simpler terms in particular cases. Let $\ell \in \N$ with $\ell < k$ and consider the special case in \eqref{constante_fea} where $f_1= \cdots = f_{\ell} = f$ and $f_{\ell+1} = \cdots = f_{k} = g$. Then one can obtain a simpler bound for $C_k$ in \eqref{constante_fea}, given by
\[
C_k \leq (\|f\|_{\infty}+\varepsilon)^{\ell - 1} \left( \ell\|g\|^{k-\ell} + (k-\ell)(\|f\|_{\infty}+\varepsilon)(\|g\|_{\infty}+\varepsilon)^{k-\ell-1} \right).
\]
Indeed, first for any $0\leq j \leq k-1$ one gets
\begin{equation}\label{ck-1}
\begin{aligned}
&\sum_{\substack{ i_1,i_2,\ldots, i_{k} \in\{0,1\}\\ i_1+i_2+\cdots +i_{k} = k-1-j} }  \norm{f}_{\infty}^{i_1+\cdots+i_{\ell}} \norm{g}_{\infty}^{i_{\ell + 1}+\cdots+i_{k}} \\
&\hspace{3cm}= \sum_{i=0}^{n}  \begin{pmatrix} k-\ell\\ k-1-i-j \end{pmatrix} \begin{pmatrix} \ell\\ i \end{pmatrix} \|f\|_{\infty}^i \|g\|_{\infty}^{k-1-i-j} {\bf 1}_{\{\ell-1-j \leq i \leq k-1-j\}}.  
\end{aligned}
\end{equation}
To see \eqref{ck-1}, notice that
\[
\begin{aligned}
&\sum_{\substack{ i_1,i_2,\ldots, i_{k} \in\{0,1\}\\ i_1+i_2+\cdots +i_{k} = k-1-j} }  \norm{f}_{\infty}^{i_1+\cdots+i_{\ell}} \norm{g}_{\infty}^{i_{\ell+1}+\cdots+i_{k}} \\
&\hspace{2cm}= \sum_{m = 0}^{\ell} \sum_{\substack{ i_1,i_2,\ldots, i_{k} \in\{0,1\}\\ i_1+i_2+\cdots +i_{k} = k-1-j} } \|f\|_{\infty}^{m} \|g\|_{\infty}^{i_{\ell+1}+\cdots+i_{k}}{\bf 1}_{\{i_1+\cdots+i_{\ell} = m\}}\\
&\hspace{2cm}= \sum_{m=0}^{\ell} \|f\|_{\infty}^{m}\sum_{\substack{ i_{\ell+1},\ldots, i_{k} \in\{0,1\}\\ i_{\ell+1}+\cdots +i_{k} = k-1-m-j} } \|g\|_{\infty}^{i_{\ell+1}+\cdots+i_{k}} \sum_{i_{1},i_{2},\ldots,i_{\ell}\in\{0,1\}} {\bf 1}_{\{i_1+\cdots+i_\ell = m\}} \\
&\hspace{2cm} =\sum_{m = 0}^{\ell}\begin{pmatrix} \ell \\ m \end{pmatrix} \|f\|_{\infty}^{m} \|f\|^{k-1-m-j}_{\infty}\sum_{\substack{ i_{\ell+1},\ldots, i_{k} \in\{0,1\}\\ i_{\ell+1}+\cdots +i_{k} = k-1-m-j} } {\bf 1}_{\{0\leq k-1-m-j \leq k-\ell\}} \\
& \hspace{2cm} = \sum_{m=0}^{\ell} \begin{pmatrix} \ell \\ m \end{pmatrix} \begin{pmatrix} k-\ell \\ k-1-m-j \end{pmatrix} \|f\|_{\infty}^{m} \|g\|^{k-1-m-j}_{\infty} {\bf 1}_{\{\ell-1-j\leq m \leq k-1-j\}}.
\end{aligned}
\]
Coming back to the estimates for $C_k$, by using \eqref{ck-1} and the definition of $C_k$ in \eqref{constante_fea} one obtains
\[
C_k = \sum_{j=0}^{k-1} \sum_{i=0}^{\ell}  \begin{pmatrix} k-\ell\\ k-1-i-j \end{pmatrix} \begin{pmatrix} \ell\\ i \end{pmatrix} \|f\|_{\infty}^i\|g\|_{\infty}^{k-1-i-j}\varepsilon^j {\bf 1}_{\{\ell-1 \leq i+j \leq k-1\}}.
\]
As $i+j$ is conditioned to belong to an interval, it makes sense to introduce a variable $m = i + j$ such that $\ell-1 \leq m \leq k-1$. This implies that
\[
C_k = \sum_{m=\ell-1}^{k-1} \sum_{i = 0}^{\ell} \sum_{j=0}^{k-1} \begin{pmatrix} k-\ell\\ k-1-i-j \end{pmatrix} \begin{pmatrix} \ell\\ i \end{pmatrix} \|f\|_{\infty}^i\|g\|_{\infty}^{k-1-i-j}\varepsilon^j {\bf 1}_{\{j= m-i\}}{\bf 1}_{\{0 \leq j \leq k-1\}}.
\]
One of those sums can be cancelled by setting $j = m - i$, with the condition
 \[
 0 \leq j \leq k-1 \quad  \Longrightarrow \quad m-k+1 \leq i \leq m,
 \]
and therefore, $C_k$ can be written only from the sums involving the indexes $i$ and $m$:
 \[
 \begin{aligned}
 C_k &=  \sum_{m=\ell-1}^{k-1} \sum_{i= 0 }^{\ell} \begin{pmatrix} k-\ell\\ k-1-i-(m-i) \end{pmatrix} \begin{pmatrix} \ell\\ i \end{pmatrix} \|f\|_{\infty}^i\|g\|_{\infty}^{k-1-i-(m-i)}\varepsilon^{m-i} {\bf 1}_{\{m-k+1 \leq i \leq m\}}\\
 &= \sum_{m=\ell-1}^{k-1} \sum_{i= 0 \vee (m-k+1)}^{\ell \wedge m} \begin{pmatrix} k-\ell\\ k-1-m \end{pmatrix} \begin{pmatrix} \ell\\ i \end{pmatrix} \|f\|_{\infty}^i\|g\|_{\infty}^{k-1-m}\varepsilon^{m-i}.
 \end{aligned}
 \]
 From this notice the following two facts: 1) As $m \leq k-1$, then always $0 \vee (m-k+1) = 0$; and 2) $\ell \wedge m = \ell$ unless $m = \ell-1$. In that case, $\ell \wedge m = \ell-1$. Then
 \[
 \begin{aligned}
 C_k &= \begin{pmatrix} k-\ell\\ k-\ell \end{pmatrix} \|g\|_{\infty}^{k-\ell} \sum_{i=0}^{\ell-1} \begin{pmatrix} \ell\\ i \end{pmatrix} \|f\|_{\infty}^i \varepsilon^{\ell-1-i} \\
 & \quad + \sum_{m=\ell}^{k-1} \begin{pmatrix} k-\ell\\ k-1-m \end{pmatrix} \|g\|_{\infty}^{k-1-m} \sum_{i= 0 }^{\ell} \begin{pmatrix} \ell\\ i \end{pmatrix} \|f\|_{\infty}^i\varepsilon^{m-i} \\
&\leq \ell \|g\|_{\infty}^{k-\ell}\sum_{i=0}^{\ell-1} \begin{pmatrix} \ell-1\\ i \end{pmatrix} \|g\|_{\infty}^i \varepsilon^{\ell-1-i} + \sum_{m=\ell}^{k-1} \begin{pmatrix} k-\ell\\ k-1-m \end{pmatrix} \|g\|_{\infty}^{k-1-m} \varepsilon^{m-\ell} (\|f\|_{\infty}+ \varepsilon)^{\ell} \\
& = \ell\norm{g}_{\infty}^{k-\ell} (\norm{f}_{\infty}+\varepsilon)^{\ell-1} + (\norm{f}_{\infty}+\varepsilon)^{\ell}\sum_{m=0}^{k-1-\ell} \begin{pmatrix} k-\ell\\ k-1-\ell-m \end{pmatrix} \|g\|_{\infty}^{k-1-\ell-m} \varepsilon^{m}\\
&\leq  \ell\norm{g}_{\infty}^{k-\ell} (\norm{f}_{\infty}+\varepsilon)^{\ell-1} + (k-\ell)(\norm{f}_{\infty}+\varepsilon)^{\ell}\sum_{m=0}^{k-1-\ell} \begin{pmatrix} k-1-\ell\\ m \end{pmatrix} \|f\|_{\infty}^{k-1-\ell-m} \varepsilon^{m}.
 \end{aligned}
 \]
 Finally, one has
 \begin{equation}\label{ck-2}
 C_k \leq (\norm{f}_{\infty}+\varepsilon)^{\ell-1}\left(\ell\norm{g}_{\infty}^{k-\ell} + (k-\ell)(\norm{f}_{\infty}+\varepsilon)(\norm{g}_{\infty}+\varepsilon)^{k-\ell-1}\right).
 \end{equation}
 This bound will be particularly useful in the case of a nonlinear wave model.

\end{remark}
}

Let us recall the following result, proved by D. Yarotsky \cite{Dmitry}.

\begin{lemma}\label{lemma:DNN_mult_2}
Let $R>0$. For all $\varepsilon \in \left(0,\frac 12\right)$ there exists $\Phi_{R,\varepsilon} \in {\bf N}$ with $k_0=2$ and $k_{H+1}=1$ such that $\mathcal R(\Phi_{R,\varepsilon}):\mathbb R^2 \to \mathbb R$, and
\[
|xy - \mathcal R(\Phi_{R,\varepsilon})(x,y)| \leq \varepsilon, \quad \hbox{for all $x,y \in [-R,R]$.}
\]
Moreover, one has $\mathcal W(\Phi_{R,\varepsilon}) \leq 5$ and $\mathcal H(\Phi_{R,\varepsilon}) \leq C(\log\lceil R \rceil + \log \lceil \varepsilon^{-1}\rceil)$, for some universal $C > 0$.
\end{lemma}

In this paper, we need an improvement of the previous result, considering a general situation involving arbitrary $k$ multiplications as functions defined in $\mathbb R^k$. Indeed, we shall prove the following lemma.

{\color{black}
\begin{lemma}\label{lemma:DNN_mult_k}
Let $R \in (0,1)$, $k \in \N$ with $k \geq 2$. For all $\varepsilon \in \left(0, \frac 12\right)$ there exists $\Phi^k_{R,\varepsilon} \in {\bf N}$ such that for all $x \in [-R,R]^{k}$
\begin{equation}\label{eq:lemma_DNN_mult_eps}
\left| \prod_{i=1}^{k} x_i - \mathcal R\left(\Phi_{R,\varepsilon}^{k}\right)(x) \right| \leq (k-1)\varepsilon R^{k},
\end{equation}
and
\begin{equation}\label{eq:lemma_DNN_mult_R}
\left|  \mathcal R\left(\Phi_{R,\varepsilon}^{k}\right)(x) \right| \leq kR^{k}.
\end{equation}
Moreover, $\mathcal W(\Phi_{R,\varepsilon}^k) \leq 2k + 1$ and there exists $C>0$ such that
\begin{equation}\label{eq:lemma_DNN_mult_H}
\mathcal H(\Phi_{R,\varepsilon}^{k}) \leq C k \left( \log\lceil \varepsilon^{-1}\rceil + \log (k) + k \log \lceil R^{-1} \rceil \right).
\end{equation}
Finally, the amount of nonzero parameters of the DNN satisfies
\begin{equation}\label{cota en P}
\mathcal P(\Phi^k_{R,\varepsilon}) \leq C' k^4 \left(\log \lceil \varepsilon^{-1}\rceil + 1 + \log\lceil R^{-1} \rceil \right),
\end{equation}
for some universal $C' > 0$.
\end{lemma}

{\color{black} A similar result is also valid for $R>1$. In that case, by setting $\varepsilon_i = \varepsilon$ and $R_i = iR^i$, Lemma \ref{lemma:DNN_mult_k} can be concluded, with a sightly modification in \eqref{eq:lemma_DNN_mult_H} and \eqref{cota en P}, where $\log\lceil R\rceil$ will be placed instead of $\log\lceil R^{-1}\rceil$.}

\begin{proof}
Let $R \in (0,1)$, $k \in \N$ with $k \geq 2$ and $\varepsilon \in \left(0,\frac 12\right)$.  For $i \in \{1,...,k-1\}$ let 
\[
\varepsilon_i := \varepsilon R^{i+1} \in \left(0,\frac 12 \right), \hspace{1cm} R_i := \max\{R,i R^i\},
\] 
and let $\Phi_{R_i,\varepsilon_i} \in {\bf N}$ be the Deep Neural Network in $\mathbb R^2$ obtained by Lemma \ref{lemma:DNN_mult_2} and approximating $[-R_i,R_i]^2 \ni (x,y)\longmapsto xy\in [-R_i^2,R_i^2]$ with precision $\varepsilon_i$.  We define $y_i:=y_i(x_1,\ldots,x_{i+1},R,\varepsilon)$, $i \in \{1,\ldots,k-1\}$ as follows: for $i=1$,
\[
y_1 := \mathcal R\left(\Phi_{R_1,\varepsilon_1}\right)(x_1,x_2),
\]
and if $i \in \{2,\ldots,k-1\}$,
\[
y_i := \mathcal R\left(\Phi_{R_{i},\varepsilon_{i}}\right)\left(y_{i-1},x_{i+1}\right).
\]
We will prove for all $i \in \{1,\ldots,k-1\}$ the following two statements:
\begin{enumerate}
\item For all $x \in [-R,R]^{i+1}$,
\[
\left|\prod_{\ell = 1}^{i+1} x_{\ell} - y_i\right| \leq i \varepsilon R^{i+1},
\]
and
\[
\left|y_i\right| \leq (i+1)R^{i+1}.
\]
\item $y_i$ is the realization of some DNN $\Psi_i := \Psi_{i,R,\varepsilon} \in {\bf N}$ satisfying 
\[
\mathcal H(\Psi_i) = \sum_{\ell=1}^{i} \mathcal H\left(\Phi_{R_{\ell},\varepsilon_{\ell}}\right) + (i-1), \quad \hbox{and}\quad 
\mathcal W(\Psi_i) \leq 2i+3.
\]
\end{enumerate}
From Lemma \ref{lemma:DNN_mult_2}, $y_1$ satisfies 
\[
\left|x_1x_2 - y_1\right| = \left|x_1x_2 - \mathcal R\left(\Phi_{R_1,\varepsilon_1}\right)(x_1,x_2)\right| \leq \varepsilon_1 = \varepsilon R^2,
\]
and
\[
\left|y_1\right|\leq R^2 + \varepsilon_1 = (1+\varepsilon)R^2 \leq 2R^2.
\]
Therefore $(1)$ is true for $y_i$. Moreover, the definition of $y_1$ directly implies $(2)$. Given $y_{i-1}$ satisfying $(1)$ and $(2)$ we will prove the two statements for $y_i$. Indeed, for all $x \in [-R,R]^{i+1}$.
\[
\left|\prod_{\ell = 1}^{i+1} x_{\ell} - y_i\right| \leq\left|\prod_{\ell = 1}^{i} x_{\ell} - y_{i-1}\right|\left|x_{i+1}\right|+\left|y_{i-1}x_{i+1} - y_i\right|.
\]
Statement $(1)$ for $y_{i-1}$ and the definition of $R_i$ one has $(y_{i-1},x_{i+1}) \in [-R_i,R_i]^2$. Then
\[
|y_{i-1}x_{i+1} - y_i| = \left|y_{i-1}x_{i+1} - \mathcal R(\Phi_{R_i,\varepsilon_i})\left(y_{i-1},x_{i+1}\right)\right| \leq \varepsilon_i = \varepsilon R^{i+1}.
\]
Therefore 
\[
\left|\prod_{\ell = 1}^{i+1} x_{\ell} - y_i\right| \leq (i-1)\varepsilon R^{i+1} + \varepsilon R^{i+1} = i\varepsilon R^{i+1}.
\]
In addition, triangle inequality and statement $(1)$ implies that
\[
|y_i| \leq |y_{i-1} x_{i+1}| + \varepsilon_i \leq iR^{i+1} + \varepsilon R^{i+1} \leq (i+1)R^{i+1}.
\]
Then $y_i$ satisfies $(1)$. By assuming that $y_{i-1}$ is the realization of $\Psi_{i-1} \in {\bf N}$, $(y_{i-1},x_{i+1})$ can be seen as the realization of a DNN $\varphi_{i} \in {\bf N}$ that parallelizes $\Psi_{i-1}$ and the identity with $\mathcal H(\Psi_{i-1})$ hidden layers, $\Sigma_{\mathcal H(\Psi_{i-1})}$, that is, for all $x \in [-R,R]^{i+1}$
\[
\mathcal R(\varphi_{i})(x) = \left(\mathcal R(\Psi_{i-1})(x_1,\ldots,x_i),\mathcal R\left(\Sigma_{\mathcal H(\Psi_{i-1})}\right)(x_{i+1})\right).
\]
From Lemma \ref{lemma:DNN_para}, $\varphi_{i}$ satisfies $\mathcal H(\varphi_i) = \mathcal H(\Psi_{i-1})$ and
\[
\mathcal W(\varphi_i) \leq \mathcal W(\Psi_{i-1}) + \mathcal W\left(\Sigma_{\mathcal H(\Psi_{i-1})}\right) \leq \mathcal W(\Psi_{i-1}) + 2.
\]
Notice that $y_i$ is the composition of the DNNs $\Phi_{R_i,\varepsilon_i}$ and $\varphi_{i}$, namely, for all $x \in [-R,R]^{i+1}$
\[
y_i = \left(\mathcal R(\Phi_{R_i,\varepsilon_i}) \circ \mathcal R(\varphi_i)\right)(x).
\]
Therefore, Lemma \eqref{lemma:DNN_comp} implies that $y_i$ is the realization of some DNN $\Psi_i$ satisfying
\[
\mathcal H(\Psi_i) = \mathcal H(\Phi_{R_i,\varepsilon_i}) + \mathcal H(\Psi_{i-1}) + 1,
\]
and
\[
\mathcal W(\Psi_i) \leq \max\{4,\mathcal W(\Phi_{R_i,\varepsilon_i}),\mathcal W(\Psi_{i-1}) + 2\}.
\]
Statement $(2)$ for $y_{i-1}$ implies that 
\[
\mathcal H(\Psi_i) = \mathcal H(\Phi_{R_i,\varepsilon_i}) \sum_{\ell = 1}^{i-1} \mathcal H(\Phi_{R_{\ell},\varepsilon_{\ell}}) + i-2 + 1 = \sum_{\ell = 1}^{i} \mathcal H(\Phi_{R_{\ell},\varepsilon_{\ell}}) + i-1,
\]
and
\[
\mathcal W(\Psi_i) \leq \max\{4,5,2i + 1 + 2\} = 2i + 3.
\]
Therefore $y_i$ satisfies $(2)$. To conclude the proof of Lemma \eqref{lemma:DNN_mult_k}, define $\Phi_{R,\varepsilon}^k \in {\bf N}$ from its realization
\[
\mathcal R\left(\Phi_{R,\varepsilon}^k\right)(x) := y_{k-1} = \mathcal R(\Psi_{k-1})(x).
\]
$\Phi_{R,\varepsilon}^k$ satisfies \eqref{eq:lemma_DNN_mult_eps} and \eqref{eq:lemma_DNN_mult_R} from statement $(1)$ for $i=k-1$. Additionally, from statement $(2)$ it follows that
\[
\mathcal W\left(\Phi_{R,\varepsilon}^k\right) \leq 2k+1, \quad \hbox{and}\quad 
\mathcal H\left(\Phi_{R,\varepsilon}^k\right) = \sum_{i = 1}^{k-1} \mathcal H(\Phi_{R_{i},\varepsilon_{i}}) + k-2.
\]
Lemma \ref{lemma:DNN_mult_2} implies that for any $i \in \{1,\ldots,k-1\}$
\[
\begin{aligned}
\mathcal H(\Phi_{R_{i},\varepsilon_i}) &\leq C(\log\lceil R_i\rceil + \log\lceil \varepsilon_i^{-1}\rceil) \leq C(\log(i) + \log\lceil \varepsilon^{-1}\rceil + (i+1)\log\lceil R^{-1}\rceil).
\end{aligned}
\]
Therefore 
\[
\begin{aligned}
\mathcal H\left(\Phi_{R,\varepsilon}^k\right) &\leq C(k-1)(\log(k-1) + k\log\lceil R^{-1}\rceil + \log\lceil \varepsilon^{-1}\rceil) + k -2 \\
&\leq Ck(\log(k) + k\log\lceil R^{-1}\rceil + \log\lceil \varepsilon^{-1}\rceil).
\end{aligned}
\]
This proves \eqref{eq:lemma_DNN_mult_H}. Finally, a direct consequence of the previous results is the following: the bound $\mathcal W(\Phi^k_{R,\varepsilon}) \leq 2k+1$ and \eqref{eq:lemma_DNN_mult_H} imply that the amount of nonzero parameters of the DNN satisfies
\[
\begin{aligned}
\mathcal P\left(\Phi^k_{R,\varepsilon}\right) & \leq \left(\mathcal H\left(\Phi^k_{R,\varepsilon}\right) + 1\right) \mathcal W\left(\Phi^k_{R,\varepsilon}\right) \left(\mathcal W \left(\Phi^k_{R,\varepsilon}\right) + 1\right) \\
&\leq C' k^4 \left(\log \lceil \varepsilon^{-1}\rceil + 1 + \log\lceil R^{-1} \rceil \right),
\end{aligned}
\]
for some $C' > 0$. This proves \eqref{cota en P} and finishes the proof of Lemma \ref{lemma:DNN_mult_k}.
\end{proof}

}

We finish this section with a corollary on the product of DNNs, obtained directly from Lemmas \ref{lemma:DNN_mult_k}, \ref{lemma:DNN_comp} and \ref{lemma:DNN_para}.

\begin{cor}[DNN approximating the products of $k$ DNN with fixed deep]\label{lemma:mult_dnn_final}
Let $r>0$, $H,k \in \N$ and $R > 0$. Let $(\Phi_{i})_{i=1}^{k} \subseteq {\bf N}$ satisfy
\begin{enumerate}
\item[$(i)$] For all $i=1,\ldots,k$, $\mathcal H(\Phi_i) = H$.
\item[$(ii)$] For all $x \in B(0,r)$,
\[
|\mathcal R(\Phi_i)(x)| \leq R.
\]
\end{enumerate}
Therefore for all $\varepsilon \in (0,1)$ there exists $\Psi^{k}_{R,\varepsilon} \in {\bf N}$ such that for all $x \in B(0,r)$
\[
\left|\prod_{i=1}^{k} \mathcal R(\Phi_i)(x) - \mathcal R(\Psi^k_{R,\varepsilon})(x) \right| \leq \varepsilon.
\] 
Moreover, $\mathcal W(\Psi^k_{R,\varepsilon}) \leq \max\left\{ 2k+1,\sum_{i=1}^{k} \mathcal W(\Phi_i)\right\}$ and there exist universal constants $C,C'>0$ such that
\[
\begin{aligned}
\mathcal H(\Psi^k_{R,\varepsilon}) \leq Ck \left( k\log\lceil R\rceil + \log \lceil \varepsilon^{-1} \rceil + \log (k) \right) + H,
\end{aligned}
\]
and
%
\[
\mathcal P(\Psi^k_{R,\varepsilon}) \leq 2C'k^4 \left( \log\lceil \varepsilon^{-1}\rceil + 1 + \log \lceil R \rceil \right) + 2\sum_{i=1}^{k} \mathcal P(\Phi_i).
\]
\end{cor}

\newpage

\section{Universal approximation: the linear case}\label{sec:4}

The purpose of this section is to show a preliminary result that will be used to prove Theorem \ref{MT1}. In this setting, we shall only consider the linear wave model with given initial data and fixed source term.

\subsection{Setting} Let $f,F$ be bounded continuous functions. Let $u$ be the unique $C_b^0([0,\infty)\times \R^d,\R)$ solution of
	\[
	\begin{cases}
		\partial_{tt} u - \Delta u = F & \hbox{ on } (0,\infty)\times \R^d,\\
		u(0,\cdot) = 0 & \hbox{ on } \R^d,\\
		\partial_t u(0,\cdot) = f & \hbox{ on } \R^d,
	\end{cases}
	\]
	which by Theorem \ref{prop:sol-linear} is given for all $t \in [0,\infty)$, $x \in \R^d$ by the stochastic representation
	\be\label{eq:u_homo}
	U(t,x) = \E \left[\textbf{1}_{\{\tau \geq t\}}\frac{t}{\overline{\rho}(t)}f(x+tZ)\right] + \E \left[\textbf{1}_{\{\tau < t\}}\frac{\tau}{\rho(\tau)}F(t-\tau,x+tZ)\right].
	\ee
In order to prove our first result, let us state the main hypotheses needed:

\begin{assumptions}\label{ass1}
Fix a dimension $d\in \{1,2,3\}$ and let $f,F$ be bounded continuous functions such that \eqref{eq:u_homo} is satisfied. Let $B,T,\beta>0$, $p,r \in \N$, $q \in \N \cap [2,\infty)$ and $\alpha \in [2,\infty)$. For every $\varepsilon \in (0,1]$, let $\Phi_{f,d,\varepsilon}, \Phi_{F,d,\varepsilon} \in \emph{\textbf{N}}$ be deep neural networks satisfying the following conditions:
	\begin{enumerate}
	\item[(a)] Continuous realizations: for all $x \in \R^d$ and $t \in [0,T]$, 
	\[
	\mathcal{R}(\Phi_{f,d,\varepsilon}) \in C(\R^{d},\R), \quad \mathcal{R}(\Phi_{F,d,\varepsilon}) \in C(\R^{d+1},\R).
	\]
	\item[(b)]  Prescribed growth of realizations: for all $x \in \R^d$ and $t \in [0,T]$,
	\be\label{eq:R_f}
	|\mathcal{R}(\Phi_{f,d,\varepsilon})(x)| \leq Bd^p (1+|x|)^{pq},
	\ee
	and
	\be\label{eq:R_F}
	|\mathcal{R}(\Phi_{F,d,\varepsilon})(t,x)| \leq Bd^p (1+t+|x|)^{pq}.
	\ee
	\item[(c)] DNN approximation: for all $x \in \R^d$ and $t \in [0,T]$,
	\be\label{eq:delta_f}
	|f(x)-\mathcal{R}(\Phi_{f,d,\varepsilon})(x)| \leq \varepsilon B d^p (1+|x|)^{pq},
	\ee
	\be\label{eq:delta_F}
	|F(t,x)-\mathcal{R}(\Phi_{F,d,\varepsilon})(t,x)| \leq \varepsilon B d^p (1+t+|x|)^{pq},
	\ee
	\item[(d)] A priori estimates on the number of parameters of the \emph{DNN} approximation of the initial data.
	\be\label{eq:cota_param_phiu}
	\mathcal{P}(\Phi_{f,d,\varepsilon}) \leq Bd^p \varepsilon^{-\alpha}, \hspace{1cm} \mathcal{P}(\Phi_{F,d,\varepsilon}) \leq Bd^p \varepsilon^{-\alpha},
	\ee
	\item[(e)] A priori estimates on the number of hidden layers of the \emph{DNN} approximation of the initial data.
	\be\label{eq:dim_dnn_encontradas}
	\mathcal{H}(\Phi_{f,d,\varepsilon}) \leq Bd^p \varepsilon^{-\beta}, \hspace{1cm} \mathcal{H}(\Phi_{F,d,\varepsilon}) \leq Bd^p \varepsilon^{-\beta}.
	\ee
	\item[(f)] Finally, let $\nu$ be a probability measure on $\R^d$ such that 
	\be\label{eq:nu}
	\left(\int_{\R^d} |x|^{2pq} \nu(dx)\right)^{\frac{1}{2pq}} \leq Bd^r.
	\ee
	\end{enumerate}
\end{assumptions}

Assumptions \ref{ass1} are standard in the literature, and represent the fact that one cannot expect a good approximation of the solution if the (initial) data of the problem is not well-approximated a priori. Assumptions (c) are good in the light cone region $|x|<t$, $0\leq t\leq T$, after this they become useless (specially in space), but polynomial growth bounds are required to apply assumption (f). Finally, assumption (f) is also required because universality is a concept deeply related to compactness: usually one cannot approximate functions in unbounded domains unless weighted estimates or compact subsets are used, exactly as in the Stone-Weierstrass theorem. 

\begin{theorem}[Universal approximation of linear waves]\label{thrm:main_linear}
Under Assumptions \ref{ass1}, the following is satisfied: for all $t \in [0,T]$, $\varepsilon \in (0,1]$ there exists a \emph{DNN} $\Phi_{d,t,\varepsilon} \in \emph{\textbf{N}}$, and constants $\overline{B}:=\overline{B}(p,q,B,T,\alpha,\beta)>0$ and $\eta:=\eta(p,q,r,\alpha,\beta) >0$, such that $\mathcal{R}(\Phi_{d,t,\varepsilon}) \in C(\R^d,\R)$ and
	\be\label{eq:main_eps}
	\left(\int_{\R^d} |U(t,x) - \mathcal{R}(\Phi_{d,t,\varepsilon})(x)|^2 \nu(dx) \right)^{\frac 12} \leq \varepsilon.
	\ee
	Moreover, the number of nonzero parameters of the \emph{DNN} $\Phi_{d,t,\varepsilon}$ satisfies the bound
	\be\label{eq:main_param}
	\mathcal{P}(\Phi_{d,t,\varepsilon}) \leq \overline{B} d^{\eta} \varepsilon^{-\eta}, \quad \forall t\in [0,T],
	\ee
	and the number of hidden layers of the \emph{DNN} $\Phi_{d,t,\varepsilon}$ satisfies the bound
	\be\label{eq:main_hidden}
	\mathcal H(\Phi_{d,t,\varepsilon}) \leq (B+2)d^p\delta^{-\beta}, \quad \forall t \in [0,T].
	\ee
\end{theorem}

The rest of this section is devoted to the proof of Theorem \ref{thrm:main_linear}. Let $d=1,2,3$, and $u=u(t,x)$ be defined as in \eqref{eq:u_homo}. Let $\delta \in (0,1]$ and let $\Phi_{f,d,\delta}, \Phi_{F,d,\delta} \in {\bf N}$ satisfy hypotheses \eqref{eq:R_f}-\eqref{eq:dim_dnn_encontradas}. We will divide this proof in several steps for a better structure of the ideas.

\medskip
\noindent
{\bf Step 1.} For any $t \in [0,T]$ and $x \in \R^d$ define
\[
v(t,x) := \E \left[{\bf 1}_{\{\tau \geq t\}}\frac{t}{\overline\rho(t)}\mathcal R(\Phi_{f,d,\delta})(x+tZ) + {\bf 1}_{\{\tau < t\}} \frac{\tau}{\rho(\tau)}\mathcal R(\Phi_{F,d,\delta})(t-\tau,x+\tau Z)\right].
\]
We want to estimate $|U(t,x)-v(t,x)|$. 

\begin{lemma}
One has
\begin{equation}\label{eq:step1-2}
|U(t,x)-v(t,x)| \leq Bd^p(1+|x|+t)^{pq} \delta t \left(1 + \frac t2\right).
\end{equation}
\end{lemma}

\begin{proof}
Notice by triangle inequality that
\[
\begin{aligned}
|U(t,x)-v(t,x)| &\leq \E\left[{\bf 1}_{\{\tau \geq t\}} \frac t{\overline \rho(t)} |f(x+tZ) - \mathcal R(\Phi_{f,d,\delta})(x+tZ)| \right] \\
&\hspace{.5cm} + \E\left[{\bf 1}_{\{\tau < t\}} \frac{\tau}{\rho(\tau)}|F(t-\tau,x+\tau Z) - \mathcal R(\Phi_{F,d,\delta})(t-\tau,x+\tau Z)|\right].
\end{aligned}
\]
Inequalities \eqref{eq:delta_f} and \eqref{eq:delta_F} imply that
\[
\begin{aligned}
|U(t,x)-v(t,x)| &\leq \frac{t}{\overline \rho(t)}Bd^p\delta\E\left[{\bf 1}_{\{\tau \geq t\}}(1+|x+tZ|)^{pq}\right] \\
&\hspace{.5cm}+ Bd^p \delta \E\left[{\bf 1}_{\{\tau < t\}}\frac{\tau}{\rho(\tau)}(1+t-\tau+|x+\tau Z|)^{pq}\right].
\end{aligned}
\]
Recall that $Z$ takes values into $B(0,1)$, then $|Z| \leq 1$. This and triangle inequality gives
\begin{equation}\label{eq:step1-1}
|U(t,x)-v(t,x)|\leq Bd^p(1+|x|+t)^{pq}\delta \left(\frac{t}{\overline \rho(t)} \E\left[{\bf 1}_{\{\tau \geq t\}}\right] + \E\left[{\bf 1}_{\{\tau < t\}} \frac{\tau}{\rho(\tau)}\right]\right).
\end{equation}
On the one hand, from the definition of $\tau$ it follows that
\[
\E\left[{\bf 1}_{\{\tau \geq t\}}\right] = \int_{t}^{\infty}\rho(s) ds = \overline \rho(s).
\]
On the other hand, 
\[
\E\left[{\bf 1}_{\{\tau < t\}}\frac{\tau}{\rho(\tau)}\right] = \int_0^t \frac{s}{\rho(s)}\rho(s)ds = \frac{t^2}{2}.
\]
These two computations can be replaced into \eqref{eq:step1-1} to obtain
\[
|U(t,x)-v(t,x)| \leq Bd^p(1+|x|+t)^{pq} \delta t \left(1 + \frac t2\right).
\]
This shows \eqref{eq:step1-2}.
\end{proof}

\medskip
\noindent
{\bf Step 2.} For $M \in \N$ define
\begin{equation}\label{eq:E_M}
\begin{aligned}
&\hbox{E}_{M,\delta}(t,x) \\
& := \frac 1M \sum_{i=1}^{M} \left({\bf 1}_{\{\tau_i \geq t\}}\frac{t}{\overline \rho(t)}\mathcal R(\Phi_{f,d,\delta})(x+tZ_{i}) + {\bf 1}_{\{\tau_i < t\}}\frac{\tau_i}{\rho(\tau_i)}\mathcal R(\Phi_{F,d,\delta})(t-\tau_i,x+\tau_iZ_i)\right),
\end{aligned}
\end{equation}
where $(\tau_{i})_{i=1}^{M}$ and $(Z_i)_{i=1}^{M}$ are i.i.d. copies of $\tau$ and $Z$, respectively. We want to estimate the quantity $\E\left[\left|v(t,x)-\hbox{E}_{M,\delta}(t,x)\right|^2\right]^{\frac 12}$. By properties of the expected value, one has that
\[
\begin{aligned}
&\E\left[\left|v(t,x)-\hbox{E}_{M,\delta}(t,x)\right|^2\right]^{\frac 12} \\
&\hspace{1cm}\leq \frac{1}{\sqrt{M}} \E\left[\left({\bf 1}_{\{\tau \geq t\}}\frac{t}{\overline\rho(t)}\mathcal R(\Phi_{f,d,\delta})(x+tZ) + {\bf 1}_{\{\tau < t\}} \frac{\tau}{\rho(\tau)}\mathcal R(\Phi_{F,d,\delta})(t-\tau,x+\tau Z)\right)^2\right]^{\frac 12}.
\end{aligned}
\]
By the inequality $(a+b)^2 \leq 2a^2 + 2b^2$, in order to bound $\E\left[\left|v(t,x)-\hbox{E}_{M,\delta}(t,x)\right|^2\right]^{\frac 12}$, we need to bound the following quantities
\[
\text{I}:= \E\left[{\bf 1}_{\{\tau\geq t\}}\frac{t^2}{\overline \rho(t)^2} \mathcal R(\Phi_{f,d,\delta})(x+tZ)^2\right],
\]
and
\[
\text{II} := \E\left[{\bf 1}_{\{\tau< t\}}\frac{\tau^2}{\rho(\tau)^2} \mathcal R(\Phi_{F,d,\delta})(t-\tau,x+\tau Z)^2\right].
\]
Using \eqref{eq:R_f}, \eqref{eq:R_F} and similar arguments used in Step 1 we have
\[
\text{I} \leq \frac{t^2}{\overline \rho(t)^2} B^2d^{2p}(1+|x|+t)^{2pq} \E\left[{\bf 1}_{\{\tau \geq t\}}\right], 
\]
and
\[
\text{II} \leq B^2d^{2p}(1+|x|+t)^{2pq} \E\left[{\bf 1}_{\{\tau <t\}} \frac{\tau^2}{\rho(\tau)^2} \right].
\]
For $\text{I}$ we can use that $\E\left[{\bf 1}_{\{\tau \geq t\}}\right] = \overline \rho(t)$. For $\text{II}$ notice that
\[
\E\left[{\bf 1}_{\{\tau < t\}}\frac{\tau^2}{\rho(\tau)^2}\right] = \int_0^t \frac{s^2}{\rho(s)} ds.
\]
Recall that $\rho$ is a continuous function on $[0,\infty)$ and in particular, continuous on the compact set $[0,t]$. Therefore $\rho(s)$ achieves it maximum and minimum on $[0,t]$. Define
\[
\underline \rho(t) := \min_{s\in[0,t]} \rho(s).
\]
Then
\[
\text{II} \leq B^2d^{2p}(1+|x|+t)^{2pq} \frac{t^3}{3\underline \rho(t)}.
\]
From $\text{I}$ and $\text{II}$ we conclude that
\[
\begin{aligned}
\E\left[\left|v(t,x)-\hbox{E}_{M,\delta}(t,x)\right|^2\right]^{\frac 12} &\leq \frac{\sqrt 2}{\sqrt M} (\text{I} + \text{II})^{\frac 12}\\
&\leq \frac{\sqrt 2}{\sqrt M} B d^p (1+|x|+t)^{pq} t \left( \frac{1}{\overline \rho(t)} + \frac{t}{3\underline \rho(t)} \right)^{\frac 12}.
\end{aligned}
\]

\medskip
\noindent
{\bf Step 3.} Steps 1 and 2, and triangle inequality allow us to bound $\E\left[\left| U(t,x)-\hbox{E}_M(t,x)\right|^2\right]^{\frac 12}$. Indeed,
\[
\begin{aligned}
\E \left[\left| U(t,x)-\hbox{E}_{M,\delta}(t,x)\right|^2 \right]^{\frac 12} &\leq |u(t,x) - v(t,x)| + \E \left[\left|v(t,x)-\hbox{E}_{M,\delta}(t,x)\right|^2 \right]^{\frac 12}\\
&\leq Bd^p(1+|x|+t)^{pq} \delta t \left(1 + \frac t2\right) \\
&\hspace{.5cm}+ \frac{\sqrt 2}{\sqrt M} B d^p (1+|x|+t)^{pq} t \left( \frac{1}{\overline \rho(t)} + \frac{t}{3\underline \rho(t)} \right)^{\frac 12}.
\end{aligned}
\]
By integrating over $\R^d$ with the measure $\nu$ that satisfies \eqref{eq:nu} and by Fubini we have that
\[
\begin{aligned}
\E\left[\int_{\R^d}\left| U(t,x)-\hbox{E}_{M,\delta}(t,x)\right|^2 \nu(dx)\right]&\leq  2B^2d^{2p}\delta^2 t^2\left(1+\frac t2\right)^2 \int_{\R^d} (1+|x|+t)^{2pq} \nu(dx)\\
&\hspace{.1cm}+ B^2d^{2p}\frac{4}{M} t^2 \left( \frac{1}{\overline \rho(t)} + \frac{t}{3\underline \rho(t)}\right)\int_{\R^d} (1+|x|+t)^{2pq} \nu(dx)\\
&=:{\bf error}_u^2.
\end{aligned}
\]
To bound ${\bf error}_u^2$ notice from Minkowski's inequality and \eqref{eq:nu} that
\[
\begin{aligned}
\int_{\R^d} (1+|x|+t)^{2pq} \nu(dx)&\leq \left(1 + \left(\int_{\R^d} |x|^{2pq}\nu(dx)\right)^{\frac{1}{2pq}} + t\right)^{2pq}\\
&\leq (1+Bd^r + t)^{2pq}.
\end{aligned}
\]
Then
\[
{\bf error}_u^2 \leq B^2d^{2p}(1+Bd^r+t)^{2pq}t\left(2\delta^2\left(1+\frac t2\right)^2 + \frac{4}{M}\left( \frac{1}{\overline \rho(t)} + \frac{t}{3\underline \rho(t)}\right)\right).
\]
Choosing $M := \lceil \delta^{-2}\rceil$, the smallest integer larger than $\delta^{-2}$, we have that
\[
{\bf error}_u^2 \leq \delta^2 d^{2(p+rpq)} B^2(1+B+t)^{2pq}t\left(2\left(1+\frac t2\right)^2 + 4\left( \frac{1}{\overline \rho(t)} + \frac{t}{3\underline \rho(t)}\right)\right).
\]
For simplicity, we denote $C$ as the constant: 
\[
C^2 := C(T,B,p,q)^2 = \sup_{t \in [0,T]} \left[B^2(1+B+t)^{2pq}t\left(2\left(1+\frac t2\right) + 4\left( \frac{1}{\overline \rho(t)} + \frac{t}{3\underline \rho(t)}\right)\right)\right].
\]
The bound on ${\bf error}_u^2$ implies that
\[
\E\left[\int_{\R^d}\left| U(t,x)-\hbox{E}_{M,\delta}(t,x)\right|^2 \nu(dx)\right] \leq \delta^2 d^{2(p+pqr)} C^2.
\]
Then there exist values $(\tau_i)_{i=1}^{M} \subset \R$ and vectors $(Z_i)_{i=1}^{M} \subset B(0,1)$ such that
\[
\left(\int_{\R^d}\left| U(t,x)-\overline{\hbox{E}}_{M,\delta}(t,x)\right|^2 \nu(dx)\right)^{\frac 12} \leq \delta d^{p+pqr} C.
\]
Where $\overline{\hbox{E}}_{M,\delta}(t,x)$ is defined as \eqref{eq:E_M}, with the fixed values $(\tau_i)_{i=1}^{M}$, vectors $(Z_i)_{i=1}^{M}$, and the properly choice of $M$. Fixed $\varepsilon \in (0,1]$, let
\[
\delta = \frac{\varepsilon}{d^{p+pqr}C}.
\]
We conclude that
\begin{equation}\label{eq:step3-1}
\left(\int_{\R^d}\left| U(t,x)-\overline{\hbox{E}}_{M,\delta}(t,x)\right|^2 \nu(dx)\right)^{\frac 12} \leq \varepsilon.
\end{equation}

\medskip
\noindent
{\bf Step 4. Counting DNN sizes.} To conclude the proof, we will show that $\overline{\hbox{E}}_{M,\delta}(t,x)$, with the choices of $M$ and $\delta$ defines for any $t \in [0,T]$ and $\varepsilon \in (0,1]$ a DNN $\Psi_{t,d,\varepsilon}$ such that for all $x \in \R^d$
\[
\mathcal R(\Psi_{t,d,\varepsilon})(x) = \overline{\hbox{E}}_{M,\delta}(t,x).
\]
First, notice by Lemma \ref{lemma:DNN_affine} that for all $i = 1,\ldots,M$,
\begin{equation}\label{eq:step4-1}
{\bf 1}_{\{\tau_i \geq t\}}\frac{t}{\overline \rho(t)}\mathcal R(\Phi_{f,d,\delta})(\cdot+tZ_{i}) \in \mathcal R\left(\left\{\Phi \in {\bf N}: \mathcal D(\Phi) = \mathcal D(\Phi_{f,d,\delta})\right\}\right).
\end{equation}
On the other hand side, for any $t \in \R_+$, $i =1,\ldots,M$ we have by Lemmas \ref{lemma:DNN_ext_t} and \ref{lemma:DNN_affine} that
\begin{equation}\label{eq:step4-2}
\begin{aligned}
& {\bf 1}_{\{\tau_i < t\}}\frac{\tau_i}{\rho(\tau_i)}\mathcal R(\Phi_{F,d,\delta})(t-\tau_i,\cdot+\tau_iZ_i) \\
& \quad \in \mathcal R\Big(\Big\{ \Phi \in {\bf N}: \mathcal D(\Phi) = \mathcal D(\Phi_{F,d,\delta})\odot (d,2d+1,d+1)\Big\}\Big).
\end{aligned}
\end{equation}
For $i=1,\ldots,M$ we denote by $\Psi_{1,i}:=\Psi^{1,i}_{d,t,\delta}$ and $\Psi_{2,i}:= \Psi_{d,t,\delta}^{2,i}$ the DNNs with realizations as in \eqref{eq:step4-1} and \eqref{eq:step4-2}, respectively. It is not direct that we can sum $\Psi_{1,i}$ and $\Psi_{2,i}$, because
\[
\mathcal H(\Psi_{1,i}) = \mathcal H(\Phi_{f,d,\delta}), \hspace{1cm} \hbox{and} \hspace{1cm} \mathcal H(\Psi_{2,i}) = \mathcal H(\Phi_{F,d,\delta}) + 2,
\]
and therefore $\Psi_{1,i}$ and $\Psi_{2,i}$ may not have the same length. We then consider the three possible cases:
\begin{itemize}
\item $\mathcal H(\Psi_{1,i}) = \mathcal H(\Psi_{2,i})$: There is no problem in sum both DNNs. Define $\Psi_i := \Psi^i_{d,t,\delta}$ from its realization
\[
\mathcal R(\Psi_i) \equiv \mathcal R(\Psi_{1,i}) + \mathcal R(\Psi_{2,i}).
\]
By Lemma \ref{lemma:DNN_suma} we have that $\mathcal D(\Psi_{i}) = \mathcal D(\Psi_{1,i}) \boxplus \mathcal D(\Psi_{2,i})$ and
\[
\begin{aligned}
\mathcal P(\Psi_i) \leq &~{} \mathcal P(\Psi_{1,i}) + \mathcal P(\Psi_{2,i}) \\
\leq &~{} \mathcal P(\Phi_{f,d,\delta}) + 2\mathcal P(\Phi_{F,d,\delta}) + 4(2d+1).
\end{aligned}
\]
\item $\mathcal H(\Psi_{1,i}) < \mathcal H(\Psi_{2,i})$: We need to extend $\Psi_{1,i}$ to have $H(\Psi_{2,i})$ hidden layers. Lemma \ref{lem:DNN_extension} implies that
\[
\mathcal R\big( \Sigma_{\mathcal H(\Psi_{2,i})-\mathcal H(\Psi_{1,i})-1} \big) \circ \mathcal R(\Psi_{1,i}) \in \mathcal R \Big( \Big\{\Phi \in{\bf N}: \mathcal D(\Phi) = \mathfrak n_{\mathcal H(\Psi_{2,i})-\mathcal H(\Psi_{1,i})-1}\odot\mathcal D(\Psi_{1,i})\Big\} \Big).
\]
Define $\overline \Psi_{1,i}$ the DNN with previous realization. This DNN satisfy $\mathcal H(\overline \Psi_{1,i})$ and therefore we can use Lemma \ref{lemma:DNN_suma} to obtain
\[
\mathcal R(\overline \Psi_{1,i}) + \mathcal R(\Psi_{2,i}) \in \mathcal R\big( \{\Phi \in{\bf N}: \mathcal D(\Phi) = \mathcal D(\overline \Psi_{1,i}) \boxplus \mathcal D(\Psi_{2,i})\} \big).
\]
Define $\Psi_i:= \Psi^i_{d,t,\delta}$ from its realization
\[
\mathcal R(\Psi_i) \equiv \mathcal R(\overline \Psi_{1,i}) + \mathcal R(\Psi_{2,i}) \equiv \mathcal R(\Psi_{1,i}) + \mathcal R(\Psi_{2,i}).
\]
Then 
\[
\begin{aligned}
\mathcal P(\Psi_i) &\leq \mathcal P(\overline \Psi_{1,i}) + \mathcal P(\Psi_{2,i})\\
& \leq 2\mathcal P(\Phi_{f,d,\delta}) + 4\left(\mathcal H(\Psi_{2,i})-\mathcal H(\Psi_{1,i})\right) + 2\mathcal P(\Phi_{F,d,\delta}) + 4(2d+1).
\end{aligned}
\]
\item $\mathcal H(\Psi_{1,i}) > \mathcal H(\Psi_{2,i})$: Similar to previous case. We extend $\Psi_{2,i}$ to have $H(\Psi_{1,i})$ hidden layers. Lemma \ref{lem:DNN_extension} implies that
\[
\mathcal R(\Sigma_{\mathcal H(\Psi_{1,i})-\mathcal H(\Psi_{2,i})-1}) \circ \mathcal R(\Psi_{2,i}) \in \mathcal R \Big( \Big\{\Phi \in{\bf N}: \mathcal D(\Phi) = \mathfrak n_{\mathcal H(\Psi_{1,i})-\mathcal H(\Psi_{2,i})-1}\odot\mathcal D(\Psi_{2,i})\Big\} \Big).
\]
Define $\overline \Psi_{2,i}$ the DNN with previous realization. This DNN satisfy $\mathcal H(\overline \Psi_{2,i})$ and therefore we can use Lemma \ref{lemma:DNN_suma} to obtain
\[
\mathcal R(\Psi_{1,i}) + \mathcal R(\overline \Psi_{2,i}) \in \mathcal R \Big( \Big\{\Phi \in{\bf N}: \mathcal D(\Phi) = \mathcal D(\Psi_{1,i}) \boxplus \mathcal D(\overline \Psi_{2,i})\Big\} \Big).
\]
Define $\Psi_i:= \Psi^i_{d,t,\delta}$ from its realization
\[
\mathcal R(\Psi_i) \equiv \mathcal R(\Psi_{1,i}) + \mathcal R(\overline \Psi_{2,i}) \equiv \mathcal R(\Psi_{1,i}) + \mathcal R(\Psi_{2,i}).
\]
Then 
\[
\begin{aligned}
\mathcal P(\Psi_i) &\leq \mathcal P(\Psi_{1,i}) + \mathcal P(\overline \Psi_{2,i})\\
& \leq \mathcal P(\Phi_{f,d,\delta}) + 4\mathcal P(\Phi_{F,d,\delta}) + 8(2d+1) + 4\left(\mathcal H(\Psi_{1,i})-\mathcal H(\Psi_{2,i})\right) .
\end{aligned}
\]
\end{itemize}

\medskip

In the three cases we have for all $i=1,\ldots,M$ a DNN $\Psi_i:=\Psi_{d,t,\delta}^i$ such that
\[
\mathcal R(\Psi_{i}) \equiv \mathcal R(\Psi_{1,i}) + \mathcal R(\Psi_{2,i}),
\]
and 
\[
\mathcal P(\Psi_i) \leq 2\mathcal P(\Phi_{f,d,\delta}) + 4\mathcal P(\Phi_{F,d,\delta}) + 4\max\{\mathcal H(\Psi_{1,i}),\mathcal H(\Psi_{2,i})\} + 8(2d+1).
\]
Notice also that for all $i=1,\ldots,M$,
\[
\mathcal H(\Psi_i) = \max\{\mathcal H(\Psi_{1,i}),\mathcal H(\Psi_{2,i})\}.
\]
It follows from Lemma \ref{lemma:DNN_suma} that
\[
\frac 1M \sum_{i=1}^{M} \mathcal R(\Psi_i) \in \mathcal R\left(\left\{\Phi \in {\bf N}: \mathcal D(\Phi) = \underset{i=1}{\overset{M}{\boxplus}} \mathcal D(\Psi_i)\right\}\right).
\]
Finally, for all $t \in [0,T]$, $\varepsilon \in (0,1]$, $d=1,2,3$ there exists $\Psi_{t,d,\varepsilon} \in {\bf N}$ such that
\[
\mathcal R(\Psi_{t,d,\varepsilon}) \equiv \frac 1M \sum_{i=1}^{M} \mathcal R(\Psi_{i}) \equiv \overline{\hbox{E}}_{M,\delta}(t,\cdot),
\] 
with the choices of $M$ and $\delta$ obtained in step 3. This and \eqref{eq:step3-1} proves \eqref{eq:main_eps}. Moreover
\[
\begin{aligned}
\mathcal P(\Psi_{t,d,\varepsilon}) \leq &~{} \sum_{i=1}^{M} \mathcal P(\Psi_i) \\
\leq &~{}  M\left(2\mathcal P(\Phi_{f,d,\delta}) + 4\mathcal P(\Phi_{F,d,\delta}) + 4\max\{\mathcal H(\Psi_{1,i}),\mathcal H(\Psi_{2,i})\} + 8(2d+1)
\right).
\end{aligned}
\]
By inequalities \eqref{eq:cota_param_phiu} and \eqref{eq:dim_dnn_encontradas} we have that
\[
\begin{aligned}
\mathcal P(\Psi_{t,d,\varepsilon}) &\leq M\left(6Bd^p\delta^{-\alpha} + 4(Bd^p\delta^{-\beta}+2) + 8(2d+1)\right)\\
&\leq M\delta^{-\gamma}d^p(6B + 4(B+2) + 24),
\end{aligned}
\]
with $\gamma := \max\{\alpha,\beta\}$. Recall the choices of $M$ and $\delta$:
\[
M = \lceil \delta^{-2} \rceil \leq 2\delta^{-2}, \hspace{1cm} \hbox{and} \hspace{1cm} \delta = \frac{\varepsilon}{Cd^{p+pqr}}.
\]
It follows from replacing these values on $\mathcal P(\Psi_{t,d,\varepsilon})$ that
\[
\mathcal P(\Psi_{t,d,\varepsilon}) \leq \varepsilon^{-2-\gamma}d^{(2+\gamma)(p+pqr)+p} 2C^{\gamma}(10B + 32).
\]
Finally, choosing 
\[
\overline B := 4C^{\gamma}(5B+16), \hspace{1cm} \hbox{and} \hspace{1cm} \eta := (2+\gamma)(p+pqr) + p,
\]
we conclude that 
\[
\mathcal P(\Psi_{t,d,\varepsilon}) \leq \overline B d^{\eta} \varepsilon^{-\eta}.
\]
This proves \eqref{eq:main_param}. Finally, \eqref{eq:main_hidden} is obtained.

\newpage

\section{Linear perturbation of waves}\label{sec:5}

Now we consider the general case in linear and nonlinear waves, Theorem \ref{MT1}. Theorem \ref{thrm:main_linear} can be recast as a preparatory result in the most general situation considered in Theorem \ref{MT1}, hinting the right direction to follow, but not directly helping to consider the nonlinear situation, because of not completely related stochastic representations.  

\medskip

Recall $F$ given in \eqref{eq:source-nonlinear}, and the branching mechanism introduced in Definition \ref{BM}.  Notice that according to \eqref{def_rho}, we have a gauge freedom in the choice of $\rho$. Along this section we will consider $\rho(t) = \lambda e^{-\lambda t}$.  We consider the first case in Theorem \ref{MT1}, where 
\[
F(t,x,u) = c(t,x)u(t,x),
\]
with $c: \R_+ \times \R^d \to \R$ is a nontrivial bounded continuous function. This source term represents the situation where \emph{every particle in the stochastic representation branches in exactly one offspring particle}. Moreover, the set $K_t$ will have one particle and $\overline K_t \setminus K_t$ will have $N$ particles, with 
\begin{equation}\label{eq:N_nl}
N = \min\{n \in \N: T_n^t = t\} \geq 0.
\end{equation}

\subsection{Preliminaries}

We are no longer working with the representation for the solution given in \eqref{eq:u_homo}. Instead, as mentioned in Remark \ref{rem:special_cases}, the solution $u$ from Theorem \ref{prop:sol-nonlinear} becomes
\begin{equation}\label{eq:sol_NL1}
U(t,x) = \E\left[ {\bf 1}_{\{T_N^t\geq t\}}\frac{\Delta T_N^t}{\overline \rho (\Delta T_N^t)} f\left(X_{T_N^t}^N\right) \prod_{k=0}^{N-1}{\bf 1}_{\{T_k^t < t\}} \frac{\Delta T_k^t}{\rho (\Delta T_k^t)}c\left(t-T_k^t,X_{T_k^t}^k\right) \right].
\end{equation}
The characteristic functions ${\bf 1}_{\{T_N^t \geq t\}}$ and ${\bf 1}_{\{T_n^t < t\}}$ (with $n \in \{0,\ldots,N-1\}$) do not give additional information, but they will be useful when we are going to consider $N$ as a fixed value, instead of the random variable defined as in \eqref{eq:N_nl}. In addition, as we said in Remark \ref{rem:special_cases}, the particles $k$ can be indexed over the natural numbers, instead of $\N$-valued vectors, then $k \in \{0,\ldots,N\}$ will be the particles of the branching mechanism.

\medskip

Given this simplest indexation, notice that for any particle $k \in \{1,\ldots,N\}$, its parent particle $k_-$ is just $k-1 \in \{0,\ldots,N-1\}$. In addition, one can simplify the following elements of the branching process:
\begin{enumerate}
\item As in Remark \ref{rem:Z}, $Z_{T_k^t}^k$ defined in \eqref{eq:DTk} can be seen as the multiplication of $\Delta T_k^t$ and the random variable $Z$ from \ref{rem:Z}.
\item Recall the random variable $X_k^t := X_{T_k^t}^k$ from \eqref{eq:XTk}. From this definition, $X_k^t$ is defined recursively as the sum of $Z_{T_k^t}^k$ and the position of its parent particle, $X_{k_-}^t$. Inductively, and given that the particle $0$ starts at position $x$, this recursion yields to
\[
X_{k}^t = x + \sum_{i=0}^{k} Z_{T_i^t}^i.
\]
From the form of $Z_{T_k^t}^k$ mentioned in previous item, we have that
\[
X_k^t = x + \sum_{i=0}^k \Delta T_i^t Z_i,
\]
where $Z_i$ are i.i.d. copies of $Z$ defined in Remark \ref{rem:Z}. Then, for all $n \in \{0,\ldots,N\}$ one can obtain the following bound for $X_n^t$:
\end{enumerate}

\begin{remark}\label{rem:cota_Xnt}
For any $k \in \{0,\ldots,N\}$ we have $|X_k^t| \leq |x|+T_k^t$. Indeed, from \eqref{eq:XTk}, and the fact that  $X_k^t = x+\sum_{i=0}^{k} \Delta T_i^t Z_i$, one gets 
\[
|X_k^t| \leq |x|+\sum_{i=0}^{k} \Delta T_i^t |Z_i| \leq |x|+\sum_{i=0}^{k} \Delta T_i^t = |x|+T_k^t.
\]
Here, we have used the fact that $\Delta T_i^t\geq 0$, and $|Z_i|\leq 1$ for dimensions $d=1,2,3.$
\end{remark}

In the proof of Theorem \ref{thrm:main_linear}, involving the classical linear wave, we saw that it was necessary to estimate terms involving the random variable $\tau$ in the stochastic representation of the solution $U$. The current case has some similarities; in the sense that we will look for estimates for the term containing the random times $T_k^t$ in the solution \eqref{eq:sol_NL1}. Recall that we are assuming $\rho(t) = \lambda e^{-\lambda t}$.

%

\begin{lemma}[Computation of expectations, linear perturbative case]\label{lemma:Expectation_NL} For every $n \in \N$, one has
\begin{enumerate} 
\item Polynomial expectation:
\be\label{eq:prod_Tn}
 \E\left[{\bf 1}_{\{T_n^t\geq t\}}\frac{\Delta T_n^t}{\overline\rho(\Delta T_n^t)}\prod_{k=0}^{n-1}{\bf 1}_{\{T_k^t<t\}}\frac{\Delta T_k^t}{\rho(\Delta T_k^t)}\right] = \frac{t^{2n+1}}{(2n+1)!},
\ee
\item Exponential second moment:
\be\label{eq:prod_Tn2}
 \E\left[\left({\bf 1}_{\{T_n^t\geq t\}}\frac{\Delta T_n^t}{\overline\rho(\Delta T_n^t)}\prod_{k=0}^{n-1}{\bf 1}_{\{T_k^t<t\}}\frac{\Delta T_k^t}{\rho(\Delta T_k^t)}\right)^2\right] = \frac{2^{n+1}e^{\lambda t}}{\lambda^n}\frac{t^{3n+2}}{(3n+2)!}.
\ee
\end{enumerate}
\end{lemma}

\begin{proof} We will prove \eqref{eq:prod_Tn} and \eqref{eq:prod_Tn2} by induction on $n \in \N$. 
\begin{enumerate}
\item Proof of \eqref{eq:prod_Tn}. Define
\[
\mathcal I_n(t) = \E\left[{\bf 1}_{\{T_n^t\geq t\}}\frac{\Delta T_n^t}{\overline\rho(\Delta T_n^t)}\prod_{k=0}^{n-1}{\bf 1}_{\{T_k^t<t\}}\frac{\Delta T_k^t}{\rho(\Delta T_k^t)}\right].
\]
If $n=1$, $\mathcal I_n(t)$ is calculated as
\[
\begin{aligned}
\mathcal I_1(t) = &~{}  \E\left[{\bf 1}_{\{\tau_1 + \tau_0 \geq t\}} \frac{t-\tau_0}{\overline \rho(t-\tau_0)}{\bf 1}_{\{\tau_0 < t\}} \frac{\tau_0}{\rho(\tau_0)}\right] = \int_{0}^{t} \int_{t-s}^{\infty} \frac{t-s}{\overline \rho(t-s)} \frac{s}{\rho(s)} \rho(s) \rho(r)dr ds.
\end{aligned}
\]
Notice that from \eqref{rho_bar},
\[
\int_{t-s}^{\infty} \rho(r)dr = \overline \rho(t-s).
\]
Then
\[
\mathcal I_1(t) = \int_0^t (t-s)s ds = \left(\frac {t^3}2 - \frac {t^3} 3\right) = \frac {t^3}6.
\]
As $3! = 6$, equation \eqref{eq:prod_Tn} is valid for $n=1$. Now suppose that for $n \in \N$, the equality \eqref{eq:prod_Tn} holds. For $n+1$ one has
\[
\begin{aligned}
& \mathcal I_{n+1}(t) \\
&= \E \left[ {\bf 1}_{\left\{\sum_{i=0}^{n+1} \tau_i \geq t \right\}}\frac{t-\sum_{i=0}^{n} \tau_i}{\overline \rho\left(t-\sum_{i=0}^{n} \tau_i\right)} \prod_{k=0}^{n} {\bf 1}_{\left\{\sum_{i=0}^{k}\tau_i < t\right\}}
 \frac{\tau_i}{\rho(\tau_i)}\right]\\
 &=\displaystyle\int_0^t \E \left[ {\bf 1}_{\left\{\sum_{i=1}^{n+1} \tau_i \geq t-s \right\}}\frac{t-s-\sum_{i=1}^{n} \tau_i}{\overline \rho\left(t-s-\sum_{i=1}^{n} \tau_i\right)} \prod_{k=1}^{n} {\bf 1}_{\left\{\sum_{i=1}^{k}\tau_i < t-s\right\}}
 \frac{\tau_i}{\rho(\tau_i)}\right]\frac{s}{\rho(s)}\rho(s)ds.
\end{aligned} 
\]
By the fact that $(\tau_i)_{i=0}^{n}$ are i.i.d. random variables, it follows that the expectation inside the integral is equal to $\mathcal I_n(t-s)$. By a change of variables $r = t-s$ and the inductive hypothesis, one has that
\[
\mathcal I_{n+1}(t) = \int_0^t \mathcal I_n(r) (t-r)dr = \frac 1{(2n+1)!}\int_0^t r^{2n+1}(t-r)dr.
\]
It follows from a simple integral calculation that
\[
\mathcal I_{n+1}(t) = \frac{1}{(2n+1)!} \left(\frac{t^{2n+3}}{2n+2} - \frac{t^{2n+3}}{2n+3}\right) = \frac{t^{(2n+1)+1}}{(2(n+1)+1)!}.
\]
This proves \eqref{eq:prod_Tn}.

\item Proof of \eqref{eq:prod_Tn2}. Define
\begin{equation}\label{def:J_n}
\mathcal J_n(t) = \E\left[{\bf 1}_{\{T_n^t\geq t\}}\frac{(\Delta T_n^t)^2}{\overline\rho^2(\Delta T_n^t)}\prod_{k=0}^{n-1}{\bf 1}_{\{T_k^t<t\}}\frac{(\Delta T_k^t)^2}{\rho^2(\Delta T_k^t)}\right].
\end{equation}
If $n=1$, $\mathcal J_n(t)$ is calculated as
\[
\begin{aligned}
\mathcal J_1(t) = &~{}  \E\left[{\bf 1}_{\{\tau_1 + \tau_0 \geq t\}} \frac{(t-\tau_0)^2}{\overline \rho^2(t-\tau_0)}{\bf 1}_{\{\tau_0 < t\}} \frac{\tau_0^2}{\rho^2(\tau_0)}\right] = \int_{0}^{t} \int_{t-s}^{\infty} \frac{(t-s)^2}{\overline \rho^2(t-s)} \frac{s^2}{\rho^2(s)} \rho(s) \rho(r)dr ds.
\end{aligned}
\]
Notice that
\[
\int_{t-s}^{\infty} \rho(r)dr = \overline \rho(t-s).
\]
On the other hand side, it follows from the assumption $\rho(t) = \lambda e^{-\lambda t}$ that
\[
\overline \rho(t-s) \rho(s) = e^{-\lambda(t-s)} \lambda e^{-\lambda s} = \lambda e^{-\lambda t}.
\]
Then
\[
\mathcal J_1(t) = \frac{e^{\lambda t}}{\lambda}\int_0^t (t-s)^2s^2 ds = \frac{e^{\lambda t}}{\lambda}\left(\frac {t^5}3 - \frac {2t^5}4 + \frac{t^5} 5\right) = \frac{2e^{\lambda t}}{\lambda}\frac {t^5}{3\cdot 4\cdot 5}.
\]
As $3 \cdot 4 \cdot 5 = \frac{5!}{2}$, equation \eqref{eq:prod_Tn2} is valid for $n=1$. Now suppose that for $n \in \N$, the equality \eqref{eq:prod_Tn2} holds. For $n+1$ one has
\[
\begin{aligned}
& \mathcal J_{n+1}(t) \\
&\quad = \E \left[ {\bf 1}_{\left\{\sum_{i=0}^{n+1} \tau_i \geq t \right\}}\frac{\left(t-\sum_{i=0}^{n} \tau_i\right)^2}{\overline \rho^2\left(t-\sum_{i=0}^{n} \tau_i\right)} \prod_{k=0}^{n} {\bf 1}_{\left\{\sum_{i=0}^{k}\tau_i < t\right\}}
 \frac{\tau_i^2}{\rho^2(\tau_i)}\right]\\
 &\quad  =\displaystyle\int_0^t \E \left[ {\bf 1}_{\left\{\sum_{i=1}^{n+1} \tau_i \geq t-s \right\}}\frac{\left(t-s-\sum_{i=1}^{n} \tau_i\right)^2}{\overline \rho^2\left(t-s-\sum_{i=1}^{n} \tau_i\right)} \prod_{k=1}^{n} {\bf 1}_{\left\{\sum_{i=1}^{k}\tau_i < t-s\right\}}
 \frac{\tau_i^2}{\rho^2(\tau_i)}\right]\frac{s^2}{\rho^2(s)}\rho(s)ds.
\end{aligned} 
\]
By the fact that $(\tau_i)_{i=0}^{n}$ are i.i.d. random variables, and \eqref{eq:DTk}, it follows 
 from \eqref{def:J_n} the expectation inside the last integral is equal to $\mathcal J_n(t-s)$. By a change of variables $r = t-s$ and the inductive hypothesis, one has that
\[
\begin{aligned}
\mathcal J_{n+1}(t) =&~{} \int_0^t \mathcal J_n(r) \frac{(t-r)^2}{\rho(t-r)}dr \\
=&~{} \frac{2^{n+1}}{\lambda^n}\frac 1{(3n+2)!}\int_0^t e^{\lambda r}r^{3n+2}(t-r)^2 \frac{1}{\lambda e^{-\lambda(t-r)}}dr.
\end{aligned}
\]
Notice that the terms $e^{\lambda r}$ are simplified, then
\[
\mathcal J_{n+1}(t) = \frac{2^{n+1}e^{\lambda t}}{\lambda^{n+1}} \frac{1}{(3n+2)!} \int_0^t (t-r)^2 r^{3n+2} dr.
\]
It follows from a simple integral calculation that
\[
\begin{aligned}
\int_0^t (t-r)^2 r^{3n+2} dr = &~{}  \left(\frac{t^{3n+5}}{3n+3} - \frac{2t^{3n+5}}{3n+4} + \frac{t^{3n+5}}{3n+5}\right) = \frac{2t^{3n+5}}{(3n+3)(3n+4)(3n+5)}.
\end{aligned}
\]
Then
\[
\mathcal J_{n+1}(t) = \frac{2^{n+2}e^{\lambda t}}{\lambda^{n+1}} \frac{t^{3(n+1)+2}}{(3(n+1)+2)!}.
\]
This finally proves \eqref{eq:prod_Tn2}.
\end{enumerate}
\end{proof}

\subsection{Proof of Theorem \ref{MT1}, linear perturbative case} The rest of this Section is devoted to the proof of Theorem \ref{MT1}. This proof will be divided in several steps for a better organization of ideas.

\medskip
\noindent
{\bf Step 1:} For any $t \in [0,T]$ and $x \in B(0,t) \subseteq \R^d$ denote
\begin{equation}\label{def_v}
v(t,x) := \E\left[ {\bf 1}_{\{T_N^t\geq t\}}\frac{\Delta T_N^t}{\overline \rho (\Delta T_N^t)} \mathcal R(\Phi_{f,d,t,\delta})(X_N^t) \prod_{k=0}^{N-1}{\bf 1}_{\{T_k^t < t\}} \frac{\Delta T_k^t}{\rho (\Delta T_k^t)}\mathcal R(\Phi_{c,d,t,\delta})(t-T_k^t,X_k^t) \right],
\end{equation}
\[
\xi_n(t,x) := f(X_n^t) \prod_{k=0}^{n-1} c(t-T_k^t,X_k^t),
\]
and
\[
\overline \xi_{n}(t,x) := \mathcal R(\Phi_{f,d,t,\delta})(X_n^t) \prod_{k=0}^{n-1} \mathcal R(\Phi_{c,d,t,\delta})(t-T_k^t,X_k^t).
\]
This step is committed to bound $|U(t,x) - v(t,x)|$. Let  $n \in \N$ satisfy $T_n^t = t$ and $T_{n-1} < t$ (i.e, $n$ will take the place of the random variable $N$). For all $x \in B(0,t)$ and $k = 0,\ldots,n$, using Remark \ref{rem:cota_Xnt},
\[
|X_k^t| \leq |x| + T_k^t \leq 2t.
\]
Therefore for all $x \in B(0,t)$, the random variables $(X_k^t)_{k=0}^{n}$ lie in the ball centered at $0$ with radius $2t$. Assumptions \eqref{eq:approx_f1} and \eqref{eq:approx_c1} imply that
\[
|f(X_n^t) - \mathcal R(\Phi_{f,d,t,\delta})(X_n^t)| \leq \varepsilon,
\]
and for all $k= 0,\ldots,n-1$,
\[
\big| c(t-T_k^t,X_k^t) - \mathcal R(\Phi_{c,d,t,\delta})(t-T_k^t,X_k^t) \big| \leq \varepsilon.
\]
We apply Lemma \ref{lemma:mult_DNNs}  on $C_n$, to have
\begin{equation}\label{problemas}
\begin{aligned}
\left|\xi_n(t,x)-\overline \xi_{n}(t,x)\right| \leq &~{}  (\varepsilon + \|f\|_\infty + n \|c\|_\infty)^n \varepsilon \\
\leq &~{}  (n \overline B)^n \left( 1+\frac2n \right)^n \varepsilon , \quad \overline{B}:= \max\{ \|c\|_\infty, \|f\|_\infty, \varepsilon\} \\
\leq &~{}  e^2 (n \overline B)^n \varepsilon.
\end{aligned}
\end{equation}
Of course estimate \eqref{problemas} is just brute force, but it represents that smallnes will be key to avoid possible blow-up.
Given \eqref{eq:prod_Tn} in Lemma \ref{lemma:Expectation_NL} we have that for any $n \in \N$,
\[
\E\left[{\bf 1}_{\{T_n^t \geq t\}}\frac{\Delta T_n^t}{\overline \rho(\Delta T_n^t)}\left(\prod_{k=0}^{n-1} {\bf 1}_{\{T_k^t < t\}} \frac{\Delta T_k^t}{\rho(\Delta T_k^t)}\right)\left|\xi_n(t,x) - \overline \xi_n(t,x)\right|\right]\leq \frac{t^{2n+1}}{(2n+1)!} e^2(n\overline B)^n \varepsilon.
\]
Now, by the use of the tower property (or law of total expectation) we have that
\[
\begin{aligned}
& \left| U(t,x) - v(t,x) \right| \\
&\leq \E\left[\left.\E\left[{\bf 1}_{\{T_N^t \geq t\}}\frac{\Delta T_N^t}{\overline \rho(\Delta T_N^t)}\left(\prod_{k=0}^{N-1} {\bf 1}_{\{T_k^t < t\}} \frac{\Delta T_k^t}{\rho(\Delta T_k^t)}\right)\left|\xi_N(t,x) - \overline \xi_N(t,x)\right|\right|N\right]\right]\\
&\leq e^2 \E\left[\frac{t^{2N+1}}{(2N+1)!} (N\overline B)^N \varepsilon \right].
\end{aligned}
\]
With the assumption $\rho(t) = \lambda e^{-\lambda t}$, the random times $T_{n}^t$ are sum of exponential random variables and then $N=\min\{k \in \N: T_k^t = t\}$ is a Poisson counting process, i.e., for $n \in \N$
\begin{equation}\label{eq:PN_nl}
\mathbb P(N=n) = e^{-\lambda t} \frac{(\lambda t)^n}{n!}.
\end{equation}
This and the definition of the expected value implies that ($0^0 =1$ by convention)
\[
\begin{aligned}
\left| U(t,x) - v(t,x) \right| &\leq \varepsilon e^2 \sum_{n\geq 0} \frac{t^{2n+1}}{(2n+1)!}  n^n \overline B^{n}  \PP(N=n) \\
&=  \varepsilon e^2 e^{-\lambda t} \sum_{n\geq 0} \frac{t^{2n+1}}{(2n+1)!}   \frac{( n \overline B \lambda t )^{n}}{n!} =  \varepsilon e^2 t e^{-\lambda t} \sum_{n\geq 0} \frac{n^n  (\overline B \lambda t^3 )^{n} }{(2n+1)! n!}.
\end{aligned}
\]
The classical quotient ratio test reveals that $\lim_{n\to +\infty} |a_{n+1}/a_n| =0$ with $a_n:= \frac{n^n  (\overline B \lambda t^3 )^{n} }{(2n+1)! n!}$. Consequently, the series has infinite radius of convergence.  To bound this series, we use the Stirling formula:
\begin{equation}\label{Stirling}
\lim_{n\to +\infty} \frac{n!}{\left( \frac{n}{e}\right)^n\sqrt{2\pi n}} =1, \quad \hbox{or } \left( \frac{n}{e}\right)^n\sqrt{2\pi n} e^{1/(12n+1)}<n! <\left( \frac{n}{e}\right)^n\sqrt{2\pi n} e^{1/(12n)}.
\end{equation}
Therefore, thanks to \eqref{Stirling},
\[
\begin{aligned}
\left| U(t,x) - v(t,x) \right| \leq &~{} C \varepsilon  t e^{-\lambda t} \sum_{n\geq 0} \frac{  (e\overline B \lambda t^3 )^{n} }{(2n+1)! } \\
\leq &~{}  \frac{Ct \varepsilon}{\sqrt{\overline B \lambda t^3}}   e^{-\lambda t} \sum_{n\geq 0} \frac{  (\sqrt{e\overline B \lambda t^3} )^{2n+1} }{(2n+1)! }   \leq   \frac{C\varepsilon}{\sqrt{\overline B \lambda t}}   e^{-\lambda t} \sinh (\sqrt{e\overline B \lambda t^3}).
\end{aligned}
\]
We have used that for any $n \geq 0$:
\begin{equation}\label{eq:step1_nl}
\frac{t^{2n+1}}{(2n+1)!} \leq \sinh(t) \leq e^t.
\end{equation}
Finally, we conclude
\begin{equation}\label{step1}
\left| U(t,x) - v(t,x) \right| \leq   \frac{C\varepsilon}{\sqrt{\overline B \lambda t}}   e^{-\lambda t} \sinh (\sqrt{e\overline B \lambda t^3}). 
\end{equation}

\medskip
\noindent
{\bf Step 2:} For all $t \in [0,T]$ we want to approximate $v(t,x)$ in \eqref{def_v} by a DNN. For $n \in \N$ we can find a DNN that approximates $\overline \xi_n(t,x)$ for all $t \in [0,T]$, $x \in B(0,t)$. Indeed, we have from \eqref{eq:approx_f1},
\[
|\mathcal R(\Phi_{f,d,t,\delta})(X_n^t)| \leq \varepsilon + \|f\|_\infty, 
\]
and for all $k = 0,\ldots,n-1$, using \eqref{eq:approx_c1},
\[
|\mathcal R(\Phi_{c,d,t,\delta})(t-T_k^t,X_k^t)| \leq  \varepsilon + \|c\|_\infty. 
\]
Therefore, from Corollary \ref{lemma:mult_dnn_final}, for all $\gamma \in \left(0,\frac 12\right)$, $n \in \N$ there exists $\Psi_{\overline B,\gamma}^n \in {\bf N}$ such that for all $x \in B(0,t)$,
\[
\left|\overline \xi_{n}(t,x) - \mathcal R\left(\Psi_{\overline B,\gamma}^{n}\right)(x)\right| \leq \gamma.
\]
This bound and Lemma \ref{lemma:Expectation_NL}, equation \eqref{eq:prod_Tn} imply that
\[
\E\left[{\bf 1}_{\{T_n^t \geq t\}}\frac{\Delta T_n^t}{\overline \rho(\Delta T_n^t)}  \left( \prod_{k=0}^{n-1}{\bf 1}_{\{T_k^t < t\}} \frac{\Delta T_k^t}{\rho(\Delta T_k^t)}\right) \left|\overline{\xi}_n(t,x) - \mathcal R\left(\Psi_{\overline B,\gamma}^{n}\right)(x)\right|  \right] \leq \frac{t^{2n+1}}{(2n+1)!} \gamma.
\]
Moreover, it follows from tower property that
\[
\begin{aligned}
&\left|v(t,x) - \E\left[{\bf 1}_{\{T_N^t \geq t\}}\frac{\Delta T_N^t}{\overline \rho(\Delta T_N^t)}  \left( \prod_{k=0}^{n-1}{\bf 1}_{\{T_k^t < t\}} \frac{\Delta T_k^t}{\rho(\Delta T_k^t)} \right)\mathcal R\left(\Psi_{\overline B,\gamma}^{N}\right)(x)\right]\right| \\
&\hspace{.5cm}\leq \E\left[\E\left[\left.{\bf 1}_{\{T_N^t \geq t\}}\frac{\Delta T_N^t}{\overline \rho(\Delta T_N^t)}  \left( \prod_{k=0}^{N-1}{\bf 1}_{\{T_k^t < t\}} \frac{\Delta T_k^t}{\rho(\Delta T_k^t)}\right) \left|\overline{\xi}_N(t,x) - \mathcal R\left(\Psi_{\overline B,\gamma}^{N}\right)(x)\right|  \right| N\right] \right]\\
&\hspace{.5cm} \leq \gamma\E\left[\frac{t^{2N+1}}{(2N+1)!}\right].
\end{aligned}
\]
Finally, from definition of the expectation and bound \eqref{eq:step1_nl} one has
\begin{equation}\label{step2}
\begin{aligned}
& \left|v(t,x) - \E\left[{\bf 1}_{\{T_N^t \geq t\}}\frac{\Delta T_N^t}{\overline \rho(\Delta T_N^t)}  \left( \prod_{k=0}^{n-1}{\bf 1}_{\{T_k^t < t\}} \frac{\Delta T_k^t}{\rho(\Delta T_k^t)} \right)\mathcal R\left(\Psi_{\overline B,\gamma}^{N}\right)(x)\right]\right| \\
&\quad \leq \gamma \sum_{n\geq 0} \frac{t^{2n+1}}{(2n+1)!} e^{-\lambda t} \frac{(\lambda t)^n}{n!}\\
&\quad \leq \gamma e^{t} \sum_{n\geq 0} e^{-\lambda t} \frac{(\lambda t)^n}{n!} = \gamma e^t.
\end{aligned}
\end{equation}

\medskip
\noindent
{\bf Step 3:} For $n \in \N$, $M \in \N$ define $\Phi_{\gamma}^{i,n} \in {\bf N}$, $i=1,\ldots,M$ by its realization
\[
\mathcal R\left(\Phi_{t,\gamma}^{i,n}\right)(x) = {\bf 1}_{\{T_{i,n}^t \geq t\}} \frac{\Delta T_{i,n}^t}{\rho(\Delta T_{i,n}^t)} \prod_{k=0}^{n-1} {\bf 1}_{\{T_{i,k}^t < t\}} \frac{\Delta T_{i,k}^t}{\rho(\Delta T_{i,k}^t)}\mathcal R\left(\Psi_{\overline B,\gamma}^{i,n}\right)(x),
\]
and
\begin{equation}\label{eq:MC_nl}
\hbox{E}_{M,\gamma}(t,x) := \frac 1M \sum_{i=1}^{M} \mathcal R\left(\Phi_{t,\gamma}^{i,N_i}\right)(x),
\end{equation}
where $(T_{i,k}^t)_{k=0}^{N_i}$ and $N_i$ are $M$ i.i.d. copies of $(T_k^t)_{k=0}^{N}$ and $N$ respectively, and $\Psi_{\overline B,\gamma}^{i,N_i}$ consider $(T_{i,k}^t)_{k=0}^{N_{i}}$ and $M$ i.i.d. copies of $(X_{k}^t)_{k=0}^{N}$, namely $(X_{i,k}^t)_{k=0}^{N_i}$. Now we want to bound
\[
\E\left[\left| \E\left[\mathcal R\left(\Phi_{t,\gamma}^N\right)(x)\right] - \hbox{E}_{M,\gamma}(t,x) \right|^2\right]^{\frac 12}.
\]
By properties of the expected value, first we have that
\[
\begin{aligned}
\E\left[\left| \E\left[\mathcal R\left(\Phi_{t,\gamma}^N\right)(x)\right] - \hbox{E}_{M,\gamma}(t,x)  \right|^2\right]^{\frac 12} &= \frac 1{\sqrt M} \left(\E\left[ \mathcal R\left(\Phi_{t,\gamma}^{N}\right)^2(x) \right]-\E\left[ \mathcal R\left(\Phi_{t,\gamma}^{N}\right)(x)\right]^2\right)^{\frac 12}\\
&\leq \frac 1{\sqrt M} \E\left[ \mathcal R(\Phi_{t,\gamma}^{N})^2(x) \right]^{\frac 12}.
\end{aligned}
\]
Then notice that for $n \in \N$,
\[
\left|\mathcal R\left(\Psi_{\overline B,\gamma}^n\right)(x)\right| \leq 1+\overline B^n,
\]
and
\[
\begin{aligned}
\E\left[\mathcal R\left(\Phi_{t,\gamma}^n\right)^2(x) \right] &= \mathcal  \E\left[\left({\bf 1}_{\{T_n^t\geq t\}}\frac{\Delta T_n^t}{\overline\rho(\Delta T_n^t)}\prod_{k=0}^{n-1}{\bf 1}_{\{T_k^t<t\}}\frac{\Delta T_k^t}{\rho(\Delta T_k^t)}\right)^2\mathcal R\left(\Psi_{\overline B,\gamma}^{k}\right)^2(x)\right]\\
&\leq \left(1+\overline B^n\right)^2 \E\left[\left({\bf 1}_{\{T_n^t\geq t\}}\frac{\Delta T_n^t}{\overline\rho(\Delta T_n^t)}\prod_{k=0}^{n-1}{\bf 1}_{\{T_k^t<t\}}\frac{\Delta T_k^t}{\rho(\Delta T_k^t)}\right)^2\right].
\end{aligned}
\]
By Lemma \ref{lemma:Expectation_NL} and \eqref{eq:prod_Tn2}, we have
\[
\E\left[\mathcal R\left(\Phi_{t,\gamma}^n\right)^2(x) \right]  \leq \left(1+\overline B^n\right)^2 \frac{2^{n+1} e^{\lambda t}}{\lambda^n} \frac{t^{3n+2}}{(3n+2)!}.
\]
Then using tower property, we have
\[
\begin{aligned}
\E\left[\mathcal R\left(\Phi_{t,\gamma}^N\right)^2(x)\right] &= \E\left[\E\left[\left. \mathcal R\left(\Phi_{t,\gamma}^N\right)^2(x)\right|N\right]\right] \leq \E\left[4\overline B^{2N} \frac{2^{N+1}e^{\lambda t}}{\lambda^N}\frac{t^{3N+2}}{(3N+2)!}\right].
\end{aligned}
\]
The definition of the expected value and \eqref{eq:PN_nl} imply that
\[
\E\left[\mathcal R\left(\Phi_{t,\gamma}^N\right)^2(x)\right]  \leq \sum_{n\geq 0} 4\overline B^{2n}\frac{2^{n+1}e^{\lambda t}}{\lambda^n} \frac{t^{3n+2}}{(3n+2)!} e^{-\lambda t} \frac{(\lambda t)^n}{n!}.
\]
Here, some terms can be cancelled. In addition, given \eqref{eq:step1_nl} it follows for all $n \in \N$ that
\[
t^{2(2n+1)} \leq (2n+1)! e^{t^2}.
\]
Therefore
\[
\E\left[\mathcal R\left(\Phi_{t,\gamma}^N\right)^2(x)\right] \leq 8e^{t^2} \sum_{n\geq 0} \frac{\left(2\overline B^2\right)^n}{n!} \frac{(2n+1)!}{(3n+2)!}.
\]
For all $n \in \N$ it is true that
\[
\frac{(2n+1)!}{(3n+2)!} < 1,
\]
and consequently
\[
\E\left[\mathcal R\left(\Phi_{t,\gamma}^N\right)^2(x)\right]  \leq 8 e^{t^2+2\overline B^2}.
\]
Finally by previous calculus we have
\begin{equation}\label{step3}
\E\left[\left| \E\left[\mathcal R\left(\Phi_{t,\gamma}^N\right)(x)\right] - \hbox{E}_{M,\gamma}(t,x) \right|^2\right]^{\frac 12} \leq \frac{2\sqrt 2}{\sqrt M} e^{\frac{t^2+2\overline B^2}{2}}.
\end{equation}

\medskip
\noindent
{\bf Step 4:} We want to bound the term
\[
\E \left[\left|\E[N] - \frac 1M \sum_{i=1}^M N_i\right|^2\right]^{\frac 12} = \frac 1{\sqrt M} \left(\E[N^2]-\E[N]^2\right)^{\frac 12}.
\]
Note that
\[
\E[N] = \sum_{n \geq 0} n e^{-\lambda t} \frac{(\lambda t)^n}{n!} = \lambda t \sum_{n \geq 0} e^{-\lambda t} \frac{(\lambda t)^n}{n!} = \lambda t,
\]
and
\[
\E[N^2] = \sum_{n \geq 0} n^2 e^{-\lambda t} \frac{(\lambda t)^n}{n!} = \lambda t \sum_{n\geq 0} (n+1) e^{-\lambda t}\frac{(\lambda t)^n}{n!} = \lambda^2 t^2 + \lambda t.
\]
We conclude
\begin{equation}\label{step4}
\E \left[\left|\E[N] - \frac 1M \sum_{i=1}^M N_i\right|^2\right]^{\frac 12} = \frac {\sqrt{\lambda t}}{\sqrt M}.
\end{equation}

\medskip
\noindent
{\bf Step 5:} Recall that $x\in \overline{B}(0,t)$. By combining \eqref{step1}, \eqref{step2}, \eqref{step3} and \eqref{step4} we have
\[
\begin{aligned}
& \E\left[ \frac{1}{|B(0,t)|} \int_{B(0,t)}\big| U(t,x) - \hbox{E}_{M,\gamma}(t,x) \big|^2 dx + \left|\E[N] - \frac 1M \sum_{i=1}^M N_i\right|^2 \right] \\
&\hspace{1cm} \leq   3\delta^2(1+\lambda t \overline B)^2e^{2t(1+\lambda \overline B - \lambda)} + 3\gamma^2 e^{2t} + \frac {24}M e^{t^2+2\overline B^2} + \frac{\lambda t}{M}=: {\bf error}_u^2.
\end{aligned}
\]
To bound ${\bf error}_u^2$ we will choose $M = \lceil \delta^{-2} \rceil$ and $\gamma = \delta$. Then
\[
{\bf error}_u^2 \leq \delta^2 \left(3(1+\lambda t \overline B)^2 e^{2t(1+\lambda \overline B - \lambda)} + 3e^{2t} + 24e^{t^2+2\overline B^2}+\lambda t\right).
\]
For simplicity, define $C$ as the constant:
\[
C:=C(T) = \max_{t \in [0,T]}\left(\sqrt 3(1+\lambda t \overline B)e^{t(1+\lambda \overline B - \lambda)} + \sqrt 3e^{t} + 2\sqrt 6 e^{\frac{t^2+3\overline B^2}2} + \sqrt{\lambda t}\right).
\]
Then by definition,
\[
{\bf error}_u \leq \delta C,
\]
and
\[
\E\left[ \frac{1}{|B(0,t)|} \int_{B(0,t)}\left|U(t,x) - \hbox{E}_{M,\gamma}(t,x) \right|^2 dx + \left|\E[N] - \frac 1M \sum_{i=1}^M N_i\right|^2 \right] \leq \delta^2C^2.
\]
This implies that there exist $(\overline N_i)_{i=1}^{M} \subset \N$ and vectors $(\overline X_n^t)_{n=0}^{\overline N_i} \subset B(0,2t)$ such that
\[
\sup_{|x|\leq t}\left|U(t,x) -  \overline{\hbox{E}}_{M,\gamma}(t,x) \right|^2+  \frac{1}{|B(0,t)|}\int_{B(0,t)}\left|U(t,x) -  \overline{\hbox{E}}_{M,\gamma}(t,x) \right|^2 dx + \left|\E[N] - \frac 1M \sum_{i=1}^M \overline N_i\right|^2 \leq \delta^2 C^2.
\]
Where $\overline{\hbox{E}}_{M,\gamma}(t,x)$ is defined as in \eqref{eq:MC_nl}, with the fixed values of $(\overline N_i)_{i=1}^M$, $(\overline X_n^t)_{n=0}^{\overline N_i}$ and the choice of $M$.

\medskip
\noindent
{\bf Step 6:} By choosing $\delta = \frac{\varepsilon}{C}$, we have
\[
\sup_{|x|\leq t}\left|U(t,x) -  \overline{\hbox{E}}_{M,\gamma}(t,x) \right|^2 + \left(\frac{1}{|B(0,t)|}\int_{B(0,t)}\left| U(t,x) -  \overline{\hbox{E}}_{M,\gamma}(t,x) \right|^2 dx\right)^{\frac 12} \leq \varepsilon.
\]
Moreover
\[
\sum_{i=1}^{M} \overline N_i \leq M \left(\delta C + \E[N]\right)  \leq M (1 + \lambda T).
\]

\medskip
\noindent
{\bf Step 7:} To conclude the proof, we will show that $\overline{\hbox{E}}_{M,\gamma}(t,x)$, with the previous choices of $M$, $\delta$ and $\gamma$ defines for any $t \in [0,T]$ and $\varepsilon \in (0,1]$ a DNN $\Psi_{t,d,\varepsilon} \in \N$ such that for all $x \in B(0,t)$:
\[
\mathcal R(\Psi_{t,d,\varepsilon})(x) = \overline{\hbox{E}}_{M,\gamma}(t,x).
\]
First notice that for any $i \in \{1,\ldots,M\}$ and $k \in \{0,\ldots,\overline N_i - 1\}$, we can use Corollary \ref{cor:DNN_fixed_comp} to obtain that
\[
\mathcal R(\Phi_{c,d,t,\delta})(t-T_k^t,\cdot) \in \mathcal R\left(\left\{\Phi \in {\bf N}: \mathcal D(\Phi) = \mathcal D(\Phi_{c,d,t,\varepsilon}) \odot (d,2d+1,d+1)\right\}\right).
\]
In addition, we can see $\mathcal R(\Phi_{c,d,t,\delta})(t-T_k^t,\overline X_{k}^t)$ and $\mathcal R(\Phi_{f,d,t,\delta})(\overline X_{\overline N_i}^t)$ as the realization of DNNs with input $x \in B(0,t)$ and the same dimensions of $\mathcal R(\Phi_{c,d,t,\delta})(t-T_k^t,\cdot)$ and $\mathcal R(\Phi_{f,d,t,\delta})$, respectively, providing that $\left(\overline X_{k}^t\right)_{k=0}^{\overline N_i}$ are translations of $x$ and the use of Lemma \ref{lemma:DNN_affine}.

\subsection{Computation of DNN parameters} Now, we want that the DNNs with realization $\mathcal R(\Phi_{c,d,t,\delta})(t-T_k^t,\overline X_{k}^t)$ and $\mathcal R(\Phi_{f,d,t,\delta})(\overline X_{\overline N_i}^t)$, namely $\overline \Phi_{c,d,t,\delta}^k$ and $\overline \Phi_{f,d,t,\delta}$, respectively, have the same number of hidden layers for Lemma \ref{lemma:mult_dnn_final}. Notice that for $k \in \{0,\ldots,\overline N_i - 1\}$,
\[
\begin{aligned}
\mathcal H\left(\overline \Phi_{c,d,t,\delta}^k\right) = &~{} \mathcal H(\Phi_{c,d,t,\delta}) + 2 =: H_1,\\
\mathcal H\left(\overline \Phi_{f,d,t,\delta}\right) = &~{} \mathcal H(\Phi_{f,d,t,\delta}) =: H_2.
\end{aligned}
\]
The three cases of study, namely $H_1 < H_2$, $H_1 > H_2$ and $H_1 = H_2$ can be treated in the same way as in the proof of Theorem \ref{thrm:main_linear}, Step 4. Then the DNNs $\overline \Phi_{c,d,t,\delta}^{k}$ and $\overline \Phi_{f,d,t,\delta}$ satisfy
\[
\mathcal P\left(\overline \Phi_{c,d,t,\delta}^k\right) \leq 2\mathcal P(\Phi_{f,d,t,\delta}) + 4\mathcal P(\Phi_{c,d,t,\delta}) + 4\max\{H_1,H_2\} + 8(2d+1),
\]
and
\[
\mathcal P(\overline \Phi_{f,d,t,\delta}) \leq 2\mathcal P(\Phi_{f,d,t,\delta}) + 4\mathcal P(\Phi_{c,d,t,\delta}) + 4\max\{H_1,H_2\} + 8(2d+1).
\]
From Assumptions \ref{ass2}, \eqref{eq:cota_param_phiu_nl} and \eqref{eq:dim_dnn_encontradas_nl}, one has that
\[
\begin{aligned}
\mathcal P\left(\overline \Phi_{c,d,t,\delta}^k\right) &\leq 6Bd^p\delta^{-\alpha} + 4(Bd^p\delta^{-\beta}+2) + 8(2d+1) \leq (10B+32)d^p\delta^{-\overline \gamma},
\end{aligned}
\]
with $\overline \gamma := \max\{\alpha,\beta\}$, and
\[
\mathcal P(\overline \Phi_{f,d,t,\delta}) \leq (10B+32)d^p\delta^{-\overline \gamma}.
\]
Therefore Lemma \ref{lemma:mult_dnn_final} implies that for all $i \in \{1,\ldots,M\}$, the DNN $\Psi_{\overline B,\gamma}^{i,\overline N_i}$ satisfy $\mathcal H\left(\Psi_{\overline B,\gamma}^{i,\overline N_i}\right) = \max\{H_1,H_2\}=: H$ and
\[
\begin{aligned}
\mathcal P\left(\Psi_{\overline B,\gamma}^{i,\overline N_i}\right)&\leq 2C'(\overline N_i + 1)^4 \left(\log\lceil \gamma^{-1}\rceil + 1 + \log \left\lceil \, \overline B \, \right\rceil\right) + 2 \sum_{k=0}^{\overline N_i-1} \mathcal P\left(\overline \Phi_{c,d,t,\delta}^n\right) + 2 \mathcal P(\overline \Phi_{f,d,t,\delta})\\
&\leq 2C'(\overline N_i + 1)^4 \left(\log\lceil \gamma^{-1}\rceil + 1 + \log \left\lceil \, \overline B \, \right\rceil \right) + 2(\overline N_i + 1)(10B+32)d^p\delta^{-\overline\gamma},
\end{aligned}
\]
for some $C'>0$. In addition, Lemma \ref{lemma:DNN_suma} implies that there exists a DNN $\Psi_{t,d,\varepsilon} \in {\bf N}$ such that
\[
\mathcal R(\Psi_{t,d,\varepsilon}) \equiv \frac 1M \sum_{i=1}^M \mathcal R\left(\Phi_{t,\gamma}^{i,\overline N_i}\right) \equiv \overline{\hbox{E}}_{M,\gamma}(t,\cdot),
\]
with the choices of $M$, $\delta$ and $\gamma$ in steps 5 and 6. This proves \eqref{eq:main_eps_nl}. Moreover $\mathcal H(\Psi_{t,d,\varepsilon}) = H$ and
\[
\mathcal P(\Psi_{t,d,\varepsilon}) \leq \sum_{i=1}^{M} \mathcal P\left(\Phi_{t,\gamma}^{i,\overline N_i}\right) = \sum_{i=1}^M \mathcal P\left(\Psi_{\overline B,\gamma}^{i,\overline N_i}\right).
\]
Then
\[
\mathcal P(\Psi_{t,d,\varepsilon}) \leq 2C'\left(\log\lceil \gamma^{-1}\rceil + 1 + \log \left\lceil \, \overline B \, \right\rceil \right)\sum_{i=1}^M(\overline N_i + 1)^4  + 2(10B+32)d^p\delta^{-\overline\gamma}\sum_{i=1}^M(\overline N_i + 1).
\]
The bound on the sum of the $\overline N_i$'s implies that
\[
\sum_{i=1}^{M} (\overline N_i + 1)^4 \leq \left(\sum_{i=1}^{M} \overline N_i + M\right)^4 \leq M^4(2+\lambda T)^{4},
\]
and therefore
\[
\begin{aligned}
\mathcal P(\Psi_{t,d,\varepsilon}) \leq &~{}  2C'\left(\log\lceil \gamma^{-1}\rceil + 1 + \log \left\lceil \, \overline B \, \right\rceil \right)M^4(2+\lambda T)^4  \\
&~{}+ 2(10B+32)d^p\delta^{-\overline\gamma}M(2+\lambda T).
\end{aligned}
\]
From the choices
\[ M = \lceil \delta^{-2}\rceil \leq 2\delta^{-2} \hspace{.5cm}\hbox{and} \hspace{.5cm}\gamma = \delta,
\]
we have
\[
\mathcal P(\Psi_{t,d,\varepsilon}) \leq 32C'\left(\delta^{-1} + 1 + \log \left\lceil \, \overline B \, \right\rceil \right)\delta^{-8}(2+\lambda T)^4  + 4(10B+32)d^p\delta^{-\overline\gamma}\delta^{-2}(2+\lambda T).
\]
Finally, the choice $\delta = \varepsilon / C$ implies that
\[
\mathcal P(\Psi_{t,d,\varepsilon}) \leq \varepsilon^{-9}32C'C^9\left(2 + \log \left\lceil \, \overline B \, \right\rceil \right)(2+\lambda T)^4  + \varepsilon^{-\overline\gamma-2}4C^{\overline \gamma + 2}(10B+32)d^p(2+\lambda T).
\]
By defining
\[
\eta := \eta(p,\alpha,\beta) = \max\{9,\overline \gamma + 2,p\},
\]
and
\[
\widetilde B :=\widetilde B(B,T,\alpha,\beta,\lambda) = 32C'C^9\left(2+\log \left\lceil \, \overline B \, \right\rceil \right)(2+\lambda T)^4 + 4C^{\overline \gamma + 2}(10B+32)(2+\lambda T),
\]
one can conclude \eqref{eq:main_param_nl}.

\newpage
 
{\color{black} 
\section{Nonlinear perturbation of waves}\label{sec:6}

Let $p \in \N$, $p \geq 2$ and consider now the case of a pure power nonlinearity, such that now the nonlinear wave equation becomes
\begin{equation}\label{eq:nlwk}
	\begin{cases}
		\partial_{tt} u - \Delta u = c(t,x)u^p & \hbox{ in } \R\times \R^d,\\
		u(0,\cdot) = 0 & \hbox{ in } \R^d,\\
		\partial_t u(0,\cdot) = f & \hbox{ in } \R^d,
	\end{cases}
\end{equation}
where $c : [0,T] \times \R^d \longrightarrow \R$ is a bounded continuous function. The branching diffusion representation of the associated solution $u$ \eqref{eq:nlwk} is given by
\begin{equation}\label{eq:u-p}
U(t,x) = \E\left[\prod_{i=1}^{(p-1)N_t^p+1} {\bf 1}_{\{T_{k_i}^t\geq t\}}\frac{\Delta T_{k_i}^t}{\overline \rho(\Delta T_{k_i}^t)}f\left(X_{T_{k_i}^t}^{k_i}\right)\prod_{i=1}^{N_t^p} {\bf 1}_{\{T_{\bar k_i}^t < t\}} \frac{\Delta T_{\bar k_i}^t}{\rho(\Delta T_{\bar k_i}^t)}c\left(t-T_{\bar k_i}^t,X_{T_{\bar k_i}^t}^{\bar k_i}\right)\right].
\end{equation}
Here, $N_t^p$ denotes the number of branches prior to time $t$ given each particle branches into exactly $p$ offspring particles. The following Lemma states the probability measure of $N_t^p$.
\begin{lemma}\label{prob-Ntp} For every $n,p \in \N$, $p \geq 2$ one has
\begin{equation}\label{eq:prob-Ntp}
\PP(N_t^p = n) = \begin{pmatrix} -\frac{1}{p-1} \\ n \end{pmatrix}(-1)^n e^{-\lambda t} \left(1-e^{-\lambda t (p-1)}\right)^n.
\end{equation}
\end{lemma}
\begin{remark}
For all $p \in \N$, $p \geq 2$, the measure given in \eqref{eq:prob-Ntp} is indeed a probability measure, given that for all $\alpha \in \R$, $|x| < 1$,
\begin{equation}\label{eq:gen-bin}
(1+x)^{\alpha} = \sum_{n=0}^{\infty} \begin{pmatrix} \alpha \\ n \end{pmatrix} x^{n}.
\end{equation}
Then 
\[
\begin{aligned}
\sum_{n=0}^{\infty} \PP(N_t^p = n) &= e^{-\lambda t} \sum_{n=0}^{\infty} \begin{pmatrix} -\frac{1}{p-1} \\ n \end{pmatrix} (e^{-\lambda t(p-1)} - 1)^n \\
&= e^{-\lambda t} \left(1+e^{-\lambda t(p-1)} - 1\right)^{-\frac{1}{p-1}}  = e^{-\lambda t} e^{\lambda t} = 1.
\end{aligned}
\]
\end{remark}
\begin{proof}[Proof of Lemma \ref{prob-Ntp}]
Let $p \in \N$, $p \geq 2$. To prove the Lemma \ref{prob-Ntp} first we will prove by induction that
\begin{equation}\label{eq:prob-Ntp2}
\PP \left( N_{t}^p = n \right) = q_n e^{-\lambda t} \left(1-e^{-\lambda t (p-1)}\right)^{n},
\end{equation}
with $(q_n)_{n \in \N} := (q_n(p))_{n \in \N}$ satisfying $q_0 = 0$ and for $n \geq 1$,
\[
q_n = \frac{1}{n(p-1)}\sum_{\substack{ i_1,\ldots,i_p \in \{0,...,n-1\} \\ i_1 +\cdots+i_p = n-1} } q_{i_1} \cdots q_{i_p}.
\]
Indeed, for $n = 0$ notice that the event represented by $\{N_t^p = 0\}$ is equivalent to the event which the first particle, namely $0$, is alive at time $t$, that is to say, the event $\{\tau_0 \geq t\}$. Then
\[
\PP \left( N_t^p = 0 \right) = \PP(\tau_0 \geq t) = \overline \rho(t) = e^{-\lambda t}.
\]
As $q_0 = 1$, the case $n=0$ is true.

\medskip
Suppose that \eqref{eq:prob-Ntp2} holds for some $n \in \N$ and we want to compute $\PP(N_{t}^p = n + 1)$. It follows from applying law of total probability, conditioning on the branching time of the first particle that
\begin{equation}\label{eq:TP-Ntp}
\PP\left( N_t^p = n+1 \right) = \int_0^t \PP \left( N_t^p = n+1 \big| \tau_0 = s \right) \rho(s) ds.
\end{equation}
Given $\tau_0 = s$, at time $s$ the first particle branches into $p$ offspring particles, generating $p$ independent branching processes whose number of particles follows the law of $N_{t-s}^p$, namely $\tilde N_1, \ldots, \tilde N_p$. Therefore the event of being $n+1$ branches up to time $t$, given that $\tau_0 = s$ is equivalent to the union of all the events such that the sum of $\tilde N_1, \ldots, \tilde N_p$ is equal to $n$, implying that
\[
\begin{aligned}
\PP \left( N_t^p = n+1 | \tau_0 = s \right) &= \sum_{\substack{ i_1,\ldots,i_p \in \{0,...,n\} \\ i_1 +\cdots+i_p = n} } \PP \left( \tilde N_1 = i_1, \ldots \tilde N_p = i_p \right) \\
&=\sum_{\substack{ i_1,\ldots,i_p \in \{0,...,n\} \\ i_1 +\cdots+i_p = n} } \PP(\tilde N_1 = i_1) \cdots \PP(\tilde N_p = i_p),
\end{aligned}
\]
where the last equality comes from the independence of such a branching processes. Then it follows from the inductive hypothesis that
\[
\begin{aligned}
&\sum_{\substack{ i_1,\ldots,i_p \in \{0,...,n\} \\ i_1 +\cdots+i_p = n} } \PP(\tilde N_1 = i_1) \cdots \PP(\tilde N_p = i_p) \\
&\hspace{1cm}= \sum_{\substack{ i_1,\ldots,i_p \in \{0,...,n\} \\ i_1 +\cdots+i_p = n} } q_{i_1}e^{-\lambda (t-s)}\left(1-e^{-\lambda (t-s)(p-1)}\right)^{i_1} \cdots q_{i_p}e^{-\lambda (t-s)}\left(1-e^{-\lambda (t-s)(p-1)}\right)^{i_p} \\
&\hspace{1cm}= e^{-\lambda (t-s) p}\left(1-e^{-\lambda (t-s) (p-1)}\right)^n \sum_{\substack{ i_1,\ldots,i_p \in \{0,...,n\} \\ i_1 +\cdots+i_p = n} } q_{i_1} \cdots q_{i_p}.
\end{aligned}
\]
Combining previous equality, \eqref{eq:TP-Ntp} and the definition of $q_n$ one gets
\[
\begin{aligned}
\PP(N_t^p = n+1) &= \sum_{\substack{ i_1,\ldots,i_p \in \{0,...,n\} \\ i_1 +\cdots+i_p = n} } q_{i_1} \cdots q_{i_p} \int_{0}^{t} e^{-\lambda(t-s)p}\left(1-e^{-\lambda(t-s)(p-1)}\right)^n \lambda e^{-\lambda s} ds \\
& =  \sum_{\substack{ i_1,\ldots,i_p \in \{0,...,n\} \\ i_1 +\cdots+i_p = n} } q_{i_1} \cdots q_{i_p} \int_{0}^{t} e^{-\lambda s p}\left(1-e^{-\lambda s(p-1)}\right)^n \lambda e^{-\lambda (t-s)} ds, \\
\end{aligned}
\]
where the last equality comes form a change of variables $s \mapsto t-s$. This integral can be computed using that
\[
\frac{d}{ds} \left(\left(1-e^{-\lambda s (p-1)}\right)^{n+1}\right) = \lambda (n+1)(p-1) \left(1-e^{-\lambda s (p-1)}\right)^n e^{-\lambda s(p-1)}.
\]
Finally,
\[
\begin{aligned}
\PP(N_t^p=n+1) &= \frac{1}{(n+1)(p-1)} \sum_{\substack{ i_1,\ldots,i_p \in \{0,...,n\} \\ i_1 +\cdots+i_p = n} } q_{i_1} \cdots q_{i_p} e^{-\lambda t} \left.\left(1-e^{-\lambda s (p-1)}\right)^{n+1}\right|_{0}^{t} \\
&= q_{n+1}e^{-\lambda t} (1-e^{-\lambda t(p-1)})^{n+1},
\end{aligned}
\]
which proves \eqref{eq:prob-Ntp2}. 

\medskip
To conclude the Lemma \ref{prob-Ntp}, we will prove that for all $n \in \N$,
\[
q_n = \begin{pmatrix} -\frac{1}{p-1} \\ n\end{pmatrix} (-1)^n.
\]
For this consider the generating function of $(q_n)_{n \in \N}=(q_n(p))_{n \in \N}$:
\[
Q_p(x) = \sum_{n \geq 0} q_n x^n.
\]
By properties on the multiplication of series, one has
\[
Q_p(x)^p = \sum_{n\geq 0} \sum_{\substack{ i_1,\ldots,i_p \in \{0,...,n\} \\ i_1 +\cdots+i_p = n} } q_{i_1} \cdots q_{i_p} x^n = \sum_{n \geq 0} (n+1)(p-1) q_{n+1} x^n.
\]
On the other hand side note that
\[
Q_p'(x) = \sum_{n \geq 1} nq_nx^{n-1} = \sum_{n\geq 0} (n+1) q_{n+1} x^n.
\]
Those computations yields to the following ODE
\[
Q_p(x)^p = (p-1)Q'_p(x),
\]
with initial condition $Q_p(0) = q_0 = 1$. The solution of previous ODE can be obtained explicitly, and it is
\[
Q_p(x) = \frac 1{(1-x)^{\frac{1}{p-1}}}.
\]
Thanks to the explicit value of $Q_p(x)$, for any $n \in \N$, $q_n$ will be simply $n$-th coefficient of the Taylor expansion of $Q_p(x)$ centered at $x_0 = 0$, that is,
\[
q_n = \frac{Q_p^{(n)}(0)}{n!} = \frac{(-1)^n\left(-\frac 1{p-1}\right)\left(-\frac 1{p-1}-1\right)\cdots\left(-\frac 1{p-1} - (n-1)\right)}{n!} = (-1)^n \begin{pmatrix} -\frac 1{p-1} \\ n \end{pmatrix}.
\]
This completes the proof of Lemma \ref{prob-Ntp}.
\end{proof}
The form of the law of $N_t^p$ in \eqref{eq:prob-Ntp} allow us to compute its first and second moment, as we see in lemma \ref{lem:moments-Ntp}
\begin{lemma}[First and second moments of $N_t^p$]\label{lem:moments-Ntp}
Let $p \in \N$, $p \geq 2$ and $t \in [0,T]$. Then
\begin{equation} \label{eq:E-Ntp}
\E\left[N_t^p\right] = \frac{e^{\lambda t (p-1)}-1}{p-1},
\end{equation}
and
\begin{equation}\label{eq:E-Ntp2}
\E\left[(N_t^p)^2\right] = \frac{e^{\lambda t(p-1)}-1}{p-1} + p\left(\frac{e^{\lambda t (p-1)}-1}{p-1}\right)^2.
\end{equation}
\end{lemma}
\begin{proof}
{\bf First moment:} Notice from definition of the expectation that
\[
\E\left[N_t^p\right] = \sum_{n=0}^{\infty} n \PP(N_t^p = n) = \sum_{n=0}^{\infty} n \begin{pmatrix} -\frac{1}{p-1} \\ n \end{pmatrix}(-1)^n e^{-\lambda t} \left(1-e^{-\lambda t (p-1)}\right)^n.
\]
It follows from properties on the binomial coefficient that
\begin{equation}\label{eq:binom}
\begin{pmatrix} -\frac{1}{p-1} \\ n \end{pmatrix} = -\frac{1}{(p-1)n} \begin{pmatrix} -\frac{1}{p-1}-1 \\ n-1 \end{pmatrix} = -\frac{1}{(p-1)n} \begin{pmatrix} -\frac{p}{p-1} \\ n-1 \end{pmatrix} .
\end{equation}
Therefore 
\[
\begin{aligned}
\E\left[N_t^p\right] &= -\frac{e^{-\lambda t}}{p-1}\sum_{n=1}^{\infty}  \begin{pmatrix} -\frac{p}{p-1} \\ n-1 \end{pmatrix} 
\left(e^{-\lambda t (p-1)}-1\right)^n \\
&= \frac{e^{-\lambda t}}{p-1}\left(1-e^{-\lambda t(p-1)}\right) \sum_{n=0}^{\infty}\begin{pmatrix} -\frac{p}{p-1} \\ n \end{pmatrix} \left(e^{-\lambda t (p-1)}-1\right)^n.
\end{aligned}
\]
Previous series has the form of \eqref{eq:gen-bin}, then
\[
\begin{aligned}
\E\left[N_t^p\right] &=  \frac{e^{-\lambda t}}{p-1}\left(1-e^{-\lambda t(p-1)}\right) \left(1+e^{-\lambda t (p-1)} - 1\right)^{-\frac{p}{p-1}}\\
&= \frac{e^{\lambda t(p-1)}}{p-1} \left(1-e^{-\lambda t(p-1)}\right) = \frac{e^{\lambda t(p-1)}-1}{p-1}.
\end{aligned}
\]
This concludes \eqref{eq:E-Ntp}.

\medskip
\noindent
{\bf Second moment:} As equal as in the first moment, we will use the definition of the expectation
\[
\E\left[\left(N_t^p\right)^2\right] = \sum_{n=0}^{\infty} n^2 \begin{pmatrix} -\frac{1}{p-1} \\ n \end{pmatrix}(-1)^n e^{-\lambda t} \left(1-e^{-\lambda t (p-1)}\right)^n.
\]
Now, by using \eqref{eq:binom}
\[
\begin{aligned}
\E\left[\left(N_t^p\right)^2\right] &= -\frac{e^{-\lambda t}}{p-1}\sum_{n=1}^{\infty} n \begin{pmatrix} -\frac{p}{p-1} \\ n-1 \end{pmatrix} \left(e^{-\lambda t (p-1)}-1\right)^n \\
&= \frac{e^{-\lambda t}}{p-1} \left(1-e^{-\lambda t (p-1)}\right)\sum_{n=0}^{\infty} (n+1) \begin{pmatrix} -\frac{p}{p-1} \\ n \end{pmatrix} \left(e^{-\lambda t (p-1)}-1\right)^n.
\end{aligned}
\]
On the one hand side, the computation for the first moment gives
\[
\sum_{n=0}^{\infty} \begin{pmatrix} -\frac{p}{p-1} \\ n \end{pmatrix} \left(e^{-\lambda t (p-1)}-1\right)^n = e^{\lambda t p}.
\]
On the other hand side, by using properties on the binomial coefficient
\[
\begin{aligned}
\sum_{n=0}^{\infty} n \begin{pmatrix} -\frac{p}{p-1} \\ n \end{pmatrix} \left(e^{-\lambda t (p-1)}-1\right)^n &= -\frac{p}{p-1} \sum_{n=1}^{\infty} \begin{pmatrix} -\frac{p}{p-1}-1 \\ n-1 \end{pmatrix}\left(e^{-\lambda t (p-1)}-1\right)^n \\
&= \frac{p}{p-1}\left(1-e^{-\lambda t (p-1)}\right) \sum_{n=0}^{\infty} \begin{pmatrix} -\frac{2p-1}{p-1} \\ n \end{pmatrix}\left(e^{-\lambda t (p-1)}-1\right)^n \\
&= \frac{p}{p-1}\left(1-e^{-\lambda t (p-1)}\right)e^{\lambda t (2p-1)}.
\end{aligned}
\]
Therefore
\[
\begin{aligned}
\E\left[\left(N_t^p\right)^2\right] &= \frac{e^{-\lambda t}}{p-1} \left(1-e^{-\lambda t (p-1)}\right) \left(e^{\lambda t p} + \frac{p}{p-1} \left(1-e^{-\lambda t (p-1)}\right)e^{\lambda t (2p-1)}\right)\\
&= \frac{e^{\lambda t (p-1)}-1}{p-1}\left(1 + p \frac{e^{\lambda t (p-1)}-1}{p-1}\right).
\end{aligned}
\]
Finally, we conclude \eqref{eq:E-Ntp2}.
\end{proof}

For $n,p \in \N$, $p \geq 2$ define
\begin{equation}\label{def_I_np}
\mathcal I_{n,p}(t) := \E\left[\prod_{i=1}^{(p-1)n+1} {\bf 1}_{\{T_{k_i}^t \geq t\}} \frac{\Delta T_{k_i}^t}{\overline \rho(\Delta T_{k_i}^t)} \prod_{i=1}^{n} {\bf 1}_{\{T_{\bar k_i}^t<t\}}\frac{\Delta T_{\bar k_i}^t}{\rho(\Delta T_{\bar k_i}^t)}\right],
\end{equation}
and
\begin{equation}\label{def_J_np}
\mathcal J_{n,p}(t) := \E\left[\prod_{i=1}^{(p-1)n+1}{\bf 1}_{\{T_{k_i}^t \geq t\}} \frac{(\Delta T_{k_i}^t)^2}{\overline \rho(\Delta T_{k_i}^t)^2}\prod_{i=1}^{n}{\bf 1}_{\{T_{\bar k_i}^t < t\}}\frac{(\Delta T_{\bar k_i}^t)^2}{\rho(\Delta T_{\bar k_i}^t)^2}\right].
\end{equation}

Next Lemma states a recurrence for the functions $\mathcal I_{n,p}(t)$ and $\mathcal J_{n,p}(t)$.

\begin{lemma}
For any $p \in \N$, $p \geq 2$ and $t \in [0,T]$, the following are satisfied: first,
\[
\mathcal I_{0,p}(t) = t, \hspace{.5cm} \hbox{and} \hspace{.5cm} \mathcal J_{0,p}(t) = e^{\lambda t} t^2.
\]
Second, for all $n \in \N$, $n \geq 1$,
\[
\mathcal I_{n,p}(t) = \int_0^t (t-s) \sum_{\substack{ i_1,\ldots,i_p \in \{0,...,n-1\} \\ i_1 +\cdots+i_p = n-1} }\mathcal I_{i_1,p}(s) \cdots \mathcal I_{i_p,p}(s) ds,
\]
and
\[
\mathcal J_{n,p}(t) = \int_0^t \frac{(t-s)^2}{\rho(t-s)} \sum_{\substack{ i_1,\ldots,i_p \in \{0,...,n-1\} \\ i_1 +\cdots+i_p = n-1} }\mathcal J_{i_1,p}(s) \cdots \mathcal J_{i_p,p}(s) ds.
\]
\end{lemma}
\begin{proof}
Notice from definitions of $\mathcal I_{n,p}(t)$ and $\mathcal J_{n,p}(t)$ in \eqref{def_I_np} and \eqref{def_J_np} respectively that
\[
\mathcal I_{0,p}(t) = \E\left[\frac{t}{\overline \rho(t)} {\bf 1}_{\{\tau_0 \geq t\}}\right] = \frac{t}{\overline \rho(t)} \overline \rho(t) = t,
\]
and
\[
\mathcal J_{0,p}(t) = \E\left[\frac{t^2}{\overline \rho(t)^2} {\bf 1}_{\{\tau_0 \geq t\}}\right] = \frac{t^2}{\overline \rho(t)^2} \overline \rho(t)^2 = t^2 e^{\lambda t}.
\]
Let $n \in \N$, $n \geq 1$. As we know that are at least one branch, therefore the first particle has lifetime $\tau_0$ less than $t$. Conditioning in the value of $\tau_0$ one has
\[
\mathcal I_{n,p}(t) = \int_0^t s \E\left[\left.\prod_{i=1}^{(p-1)n+1} {\bf 1}_{\{T_{k_i}^{t-s} \geq t-s\}} \frac{\Delta T_{k_i}^{t-s}}{\overline \rho(\Delta T_{k_i}^{t-s})} \prod_{i=1}^{n-1} {\bf 1}_{\{T_{\bar k_i}^{t-s}<{t-s}\}}\frac{\Delta T_{\bar k_i}^{t-s}}{\rho(\Delta T_{\bar k_i}^{t-s})}\right| \tau_0 = s\right] ds.
\]
Given $\tau_0 = s$, at time $s$ the first particle branches into $p$ offspring particles, generating $p$ independent branching processes whose number of particles follows the law of $N_{t-s}^p$, namely $\tilde N_1, \ldots, \tilde N_p$. Therefore conditioning into $\tau_0 = s$ is equivalent to conditioning into $\tilde N_{1} + \ldots +\tilde N_p = n-1$.

\medskip
Notice additionally that each branching process has $(p-1)\tilde N_{i} + 1$ particles alive in time $t-s$ and $\tilde N_i$ particles not alive in time $t-s$. Therefore, conditioning into $\tilde N_{1} + \ldots + \tilde N_p = n-1$ one has
\[
\sum_{i=1}^{p} ((p-1)\tilde N_i + 1) = (p-1)n +1.
\]
Moreover,
\[
\begin{aligned}
& \E\left[\left.\prod_{i=1}^{(p-1)n+1} {\bf 1}_{\{T_{k_i}^{t-s} \geq t-s\}} \frac{\Delta T_{k_i}^{t-s}}{\overline \rho(\Delta T_{k_i}^{t-s})} \prod_{i=1}^{n-1} {\bf 1}_{\{T_{\bar k_i}^{t-s}<{t-s}\}}\frac{\Delta T_{\bar k_i}^{t-s}}{\rho(\Delta T_{\bar k_i}^{t-s})}\right| \tau_0 = s\right] \\
& = \E\left[\left. \prod_{j=1}^{p} \left(\prod_{i=1}^{(p-1)\tilde N_j+1} {\bf 1}_{\{T_{k_i}^{t-s} \geq t-s\}} \frac{\Delta T_{k_i}^{t-s}}{\overline \rho(\Delta T_{k_i}^{t-s})} \prod_{i=1}^{\tilde N_j } {\bf 1}_{\{T_{\bar k_i}^{t-s}<{t-s}\}}\frac{\Delta T_{\bar k_i}^{t-s}}{\rho(\Delta T_{\bar k_i}^{t-s})}\right)\right| \tilde N_1 + \ldots + \tilde N_p = n-1\right].
\end{aligned}
\]
Using the tower property
\[
\begin{aligned}
&\E\left[\left.\prod_{i=1}^{(p-1)n+1} {\bf 1}_{\{T_{k_i}^{t-s} \geq t-s\}} \frac{\Delta T_{k_i}^{t-s}}{\overline \rho(\Delta T_{k_i}^{t-s})} \prod_{i=1}^{n-1} {\bf 1}_{\{T_{\bar k_i}^{t-s}<{t-s}\}}\frac{\Delta T_{\bar k_i}^{t-s}}{\rho(\Delta T_{\bar k_i}^{t-s})}\right| \tau_0 = s\right] \\
&\hspace{1cm}= \E \left[\left. \mathcal I_{\tilde N_1,p}(t-s) \cdots \mathcal I_{\tilde N_p,p}(t-s) \right| \tilde N_1 + \ldots + \tilde N_p = n-1\right] \\
&\hspace{1cm} = \sum_{i_1,\ldots,i_p \in \{0,\ldots,n-1\}} \mathcal I_{i_1,p}(t-s) \cdots \mathcal I_{i_p,p}(t-s) \PP\left(\left.\tilde N_1 = i_1, \ldots, \tilde N_p = i_p \right| \tilde N_1 + \ldots + \tilde N_p = n-1\right).
\end{aligned}
\]
Notice that the probability $\PP\left(\left.\tilde N_1 = i_1, \ldots, \tilde N_p = i_p \right| \tilde N_1 + \ldots + \tilde N_p = n-1\right)$ is equal to 1 when $i_1 + \ldots + i_p = n-1$ and 0 otherwise. Therefore
\[
\begin{aligned}
\mathcal I_{n,p}(t) &= \int_0^t s  \sum_{\substack{ i_1,\ldots,i_p \in \{0,...,n-1\} \\ i_1 +\cdots+i_p = n-1} } \mathcal I_{i_1,p}(t-s) \cdots \mathcal I_{i_p,p}(t-s)ds \\
&= \int_0^t (t-s)  \sum_{\substack{ i_1,\ldots,i_p \in \{0,...,n-1\} \\ i_1 +\cdots+i_p = n-1} } \mathcal I_{i_1,p}(s) \cdots \mathcal I_{i_p,p}(s)ds,
\end{aligned}
\]
where the last inequality comes from a change of variables $s \mapsto t-s$. The equality for $\mathcal J_{n,p}(t)$ uses the same arguments, an by noting that for $n \geq 1$:
\[
\mathcal J_{n,p}(t) = \int_0^t \frac{s^2}{\rho(s)} \E\left[\left.\prod_{i=1}^{(p-1)n+1} {\bf 1}_{\{T_{k_i}^{t-s} \geq t-s\}} \frac{\left(\Delta T_{k_i}^{t-s}\right)^2}{\overline \rho(\Delta T_{k_i}^{t-s})^2} \prod_{i=1}^{n-1} {\bf 1}_{\{T_{\bar k_i}^{t-s}<{t-s}\}}\frac{\left(\Delta T_{\bar k_i}^{t-s}\right)^2}{\rho(\Delta T_{\bar k_i}^{t-s})^2}\right| \tau_0 = s\right] ds.
\]
The proof is complete.
\end{proof}

\begin{cor}
For any $n,p \in \N$, $p \geq 2$ and $t \in [0,T]$
\begin{equation}\label{eq:Inp}
\mathcal I_{n,p}(t) = a_{n,p} t^{(p+1)n+1},
\end{equation}
and
\begin{equation}\label{eq:Jnp}
\mathcal J_{n,p}(t) \leq \left(\frac 2{\lambda}\right)^n b_{n,p} t^{(2p+1)n+2}e^{\lambda t ((p-1)n+1)},
\end{equation}
where $a_{0,p}=b_{0,p}=1$ and for any $n \geq 1$
\begin{equation}\label{eq:rec-anp}
a_{n,p} = \frac{1}{(p+1)n((p+1)n+1)}\sum_{\substack{ i_1,\ldots,i_p \in \{0,...,n-1\} \\ i_1 +\cdots+i_p = n-1} } a_{i_1,p} \cdots a_{i_p,p},
\end{equation}
and
\begin{equation}\label{eq:rec-bnp}
b_{n,p} = \frac{1}{(2p+1)n((2p+1)n+1)((2p+1)n+2)}\sum_{\substack{ i_1,\ldots,i_p \in \{0,...,n-1\} \\ i_1 +\cdots+i_p = n-1} } b_{i_1,p} \cdots b_{i_p,p}.
\end{equation}
\end{cor}

\begin{proof}
First we prove the Lemma for $\mathcal I_{n,p}(t)$. Fix $p \geq 2$ and recall that $\mathcal I_{0,p}(t) = a_{0,p}t$.

Assume \eqref{eq:Inp} for some $n \in \N$. For the case $n+1$ one has
\[
\begin{aligned}
\mathcal I_{n+1,p}(t)&= \int_{0}^t (t-s) \sum_{\substack{ i_1,\ldots,i_p \in \{0,...,n\} \\ i_1 +\cdots+i_p = n} }\mathcal I_{i_1,p}(s) \cdots \mathcal I_{i_p,p}(s) ds.\\
&= \int_0^t (t-s)\sum_{\substack{ i_1,\ldots,i_p \in \{0,...,n\} \\ i_1 +\cdots+i_p = n} }a_{i_1,p}s^{(p+1)i_1+1} \cdots a_{i_p,p}s^{(p+1)i_p+1} ds.\\
&= \sum_{\substack{ i_1,\ldots,i_p \in \{0,...,n\} \\ i_1 +\cdots+i_p = n} }a_{i_1,p}\cdots a_{i_p,p} \int_0^t (t-s)s^{(p+1)n+p} ds \\
&= \sum_{\substack{ i_1,\ldots,i_p \in \{0,...,n\} \\ i_1 +\cdots+i_p = n} }a_{i_1,p}\cdots a_{i_p,p} t^{(p+1)(n+1)+1}  \left(\frac{1}{(p+1)(n+1)}-\frac{1}{(p+1)(n+1)+1}\right) \\
&= \frac{1}{(p+1)(n+1)((p+1)(n+1)+1)} \sum_{\substack{ i_1,\ldots,i_p \in \{0,...,n\} \\ i_1 +\cdots+i_p = n} }a_{i_1,p}\cdots a_{i_p,p} t^{(p+1)(n+1)+1} \\
&= a_{n+1,p} t^{(p+1)(n+1)+1}.
\end{aligned} 
\]
Then \eqref{eq:Inp} is true for all $n \in \N$.

\medskip
Now we will prove \eqref{eq:Jnp}.
\end{proof}

\begin{lemma}\label{lem:cotas-conv}
For all $n,p \in \N$, $p \geq 2$,
\begin{equation}\label{eq:cota-conv-anp}
\sum_{\substack{ i_1,\ldots,i_p \in \{0,...,n\} \\ i_1 +\cdots+i_p = n} } a_{i_1,p}\cdots a_{i_p,p} \leq \frac{1}{(p+2)^n},
\end{equation}
and
\begin{equation}\label{eq:cota-conv-bnp}
\sum_{\substack{ i_1,\ldots,i_p \in \{0,...,n\} \\ i_1 +\cdots+i_p = n} } b_{i_1,p}\cdots b_{i_p,p} \leq \frac{1}{(2(2p+1)(2p+3))^n}.
\end{equation}

\end{lemma}

\begin{proof}
Let $p \in \N$, $p \geq 2$. First we will prove \eqref{eq:cota-conv-anp}. For $n = 0$ the sum on the left hand side of \eqref{eq:cota-conv-anp} is simply $a_0^p = 1$, then $\eqref{eq:cota-conv-anp}$ is true for $n=0$. Suppose that \eqref{eq:cota-conv-anp} holds for all $i \leq n \in \N$. The $k$-fold convolution satisfies the following recurrence \cite{DePril}
\begin{equation}\label{rec-conv}
\sum_{\substack{ i_1,\ldots,i_p \in \{0,...,n+1\} \\ i_1 +\cdots+i_p = n+1} } a_{i_1,p}\cdots a_{i_p,p} = \sum_{i=1}^{n+1} \left((p+1)\frac i{n+1} - 1\right) a_{i,p} \sum_{\substack{ i_1,\ldots,i_p \in \{0,...,n+1-i\} \\ i_1 +\cdots+i_p = n+1-i} } a_{i_1,p}\cdots a_{i_p,p}.
\end{equation}
An inductive hypothesis implies that for all $1 \leq i \leq n+1$
\[
\sum_{\substack{ i_1,\ldots,i_p \in \{0,...,n+1-i\} \\ i_1 +\cdots+i_p = n+1-i} } a_{i_1,p}\cdots a_{i_p,p} \leq \frac{1}{(p+2)^{n+1-i}}.
\]
In addition, by definition of $a_{n,p}$ in \eqref{eq:rec-anp}, for all $1 \leq i \leq n+1$,
\[
\begin{aligned}
a_{i,p} &= \frac{1}{(p+1)i((p+1)i+1)} \sum_{\substack{ i_1,\ldots,i_p \in \{0,...,i-1\} \\ i_1 +\cdots+i_p = i-1} } a_{i_1,p}\cdots a_{i_p,p}\\
&\leq \frac{1}{(p+1)i((p+1)i+1)}\frac{1}{(p+2)^{i-1}}.
\end{aligned}
\]
Therefore we obtain a bound for \eqref{rec-conv}
\[
\begin{aligned}
\sum_{\substack{ i_1,\ldots,i_p \in \{0,...,n+1\} \\ i_1 +\cdots+i_p = n+1} } a_{i_1,p}\cdots a_{i_p,p} & \leq \frac{p+1}{n+1} \sum_{i=1}^{n-1} i a_{i,p} \sum_{\substack{ i_1,\ldots,i_p \in \{0,...,n+1-i\} \\ i_1 +\cdots+i_p = n+1-i} } a_{i_1,p}\cdots a_{i_p,p}\\
 &\leq \frac{p+1}{(p+2)^{n}(n+1)}\sum_{i=1}^{n+1} i \frac{1}{(p+1)i((p+1)i+1)}\\
 &= \frac{1}{(p+2)^n(n+1)} \sum_{i=1}^{n+1} \frac{1}{(p+1)i+1}.
\end{aligned}
\]
Finally one concludes \eqref{eq:cota-conv-anp} using that
\[
\frac{1}{n+1} \sum_{i=1}^{n+1} \frac{1}{(p+1)i+1} \leq \frac{1}{(p+2)(n+1)} \sum_{i=1}^{n+1} 1 = \frac{1}{p+2},
\]
which implies
\[
\sum_{\substack{ i_1,\ldots,i_p \in \{0,...,n+1\} \\ i_1 +\cdots+i_p = n+1} } a_{i_1,p}\cdots a_{i_p,p} \leq \frac{1}{(p+2)^{n+1}}.
\]
The proof of \eqref{eq:cota-conv-bnp} is similar given the definition of $b_{n,p}$ in \eqref{eq:rec-bnp} and the recurrence \eqref{rec-conv} for $b_{n,p}$.
\end{proof}

\begin{cor}\label{cor:cotas-rec}
For all $p \geq 2$, $a_{0,p}= b_{0,p} = 1$ and for all $n \geq 1$
\[
a_{n,p} \leq \frac{1}{(p+1)n}\frac{1}{(p+2)^{n}},
\]
and
\[
b_{n,p} \leq \frac{1}{(p+1)n} \frac{1}{(2(2p+1)(2p+3))^n}.
\]
\end{cor}

\begin{proof}
Corollary \ref{cor:cotas-rec} comes from Lemma \ref{lem:cotas-conv} and the definition of $a_{n,p}$ and $b_{n,p}$ in \eqref{eq:rec-anp} and \eqref{eq:rec-bnp}, respectively.
\end{proof}


\subsection{Theorem for general branches} Now we are ready to conclude the proof of Theorem \ref{MT1} in the case of a $p$ power nonlinearity. {\color{black} Suppose that $f,c$ are bounded functions satisfying
\[
\max\{\norm{f}_{\infty},\norm{c}_{\infty}\} < \frac{1}{2T} \left(\frac{\lambda(2p+1)(2p+3)}{T\left(e^{\lambda T(p-1)}-1\right)}\right)^{\frac 1{2p}}.
\]
See Assumptions \ref{ass2} $(e)$ will be used now. Notice that the above is a bound telling us that $f$ and $c$ cannot be too large compared with a fixed time $T$. Or, if $f$ and $c$ are arbitrary in size, $T$ is sufficiently small.  
}

\medskip
\noindent
{\bf Step 1:} Recall \eqref{eq:u-p}
\[
U(t,x) = \E\left[\prod_{i=1}^{(p-1)N_t^p+1} {\bf 1}_{\{T_{k_i}^t\geq t\}}\frac{\Delta T_{k_i}^t}{\overline \rho(\Delta T_{k_i}^t)}f\left(X_{T_{k_i}^t}^{k_i}\right)\prod_{i=1}^{N_t^p} {\bf 1}_{\{T_{\bar k_i}^t < t\}} \frac{\Delta T_{\bar k_i}^t}{\rho(\Delta T_{\bar k_i}^t)}c\left(t-T_{\bar k_i}^t,X_{T_{\bar k_i}^t}^{\bar k_i}\right)\right].
\]
{\color{black}For $\delta \in \left(0,\frac 12\right)$, $\delta <  \frac{1}{2T} \left(\frac{\lambda(2p+1)(2p+3)}{T\left(e^{\lambda T(p-1)}-1\right)}\right)^{\frac 1{2p}}$ to be chosen later,} define
\begin{equation}\label{def_v_2}
\begin{aligned}
&v(t,x) \\
& = \E\left[\prod_{i=1}^{(p-1)N_t^p+1} {\bf 1}_{\{T_{k_i}^t\geq t\}}\frac{\Delta T_{k_i}^t}{\overline \rho(\Delta T_{k_i}^t)}\mathcal R(\Phi_{f,d,\delta})\left(X_{T_{k_i}^t}^{k_i}\right)\prod_{i=1}^{N_t^p} {\bf 1}_{\{T_{\bar k_i}^t < t\}} \frac{\Delta T_{\bar k_i}^t}{\rho(\Delta T_{\bar k_i}^t)}\mathcal R(\Phi_{c,t,d,\delta})\left(t-T_{\bar k_i}^t,X_{T_{\bar k_i}^t}^{\bar k_i}\right)\right].
\end{aligned}
\end{equation}
As in the previous section, we want to bound $|U(t,x) - v(t,x)|$. For this, define the following quantities:
\begin{enumerate}
\item $\displaystyle\xi_{n}^p(t,x) = \prod_{i=1}^{(p-1)n+1} f\left(X_{T_{k_i}^t}^t\right) \prod_{i=1}^{n} c\left(t - T_{\bar k_i}^t, X_{T_{\bar k_i}^t}^{\bar k_i}\right)$,
\item $\displaystyle\overline \xi_{n}^p(t,x) = \prod_{i=1}^{(p-1)n+1} \mathcal R(\Phi_{f,d,\delta})\left(X_{T_{k_i}^t}^t\right) \prod_{i=1}^{n} \mathcal R(\Phi_{c,t,d,\delta})\left(t - T_{\bar k_i}^t, X_{T_{\bar k_i}^t}^{\bar k_i}\right)$.
\end{enumerate}

First of all, from \eqref{eq:u-p} and \eqref{def_v_2},
\[
\begin{aligned}
& U(t,x)-v(t,x) \\
& = \E\Bigg[\prod_{i=1}^{(p-1)N^p_t+1}{\bf 1}_{\{T_{k_i}^t \geq t\}}\frac{\Delta T_{k_i}^t}{\overline \rho(\Delta T_{k_i}^t)} \prod_{i=1}^{N_t^p} {\bf 1}_{\{T_{\bar k_i}^t < t\}} \frac{\Delta T_{\bar k_i}^t}{\rho(\Delta T_{\bar k_i}^t)} \\
& \hspace{0.5cm} \times\Bigg( \prod_{i=1}^{(p-1)N_t^p+1} f\left(X_{T_{k_i}^t}^{k_i}\right)  \prod_{i=1}^{N_t} c\left(t-T_{\bar k_i}^t,X_{T_{\bar k_i}^t}^{\bar k_i}\right) - \prod_{i=1}^{(p-1)N_t^p+1} \mathcal R(\Phi_{f,d,\delta})\left(X_{T_{k_i}^t}^{k_i}\right) \prod_{i=1}^{N_t}  \mathcal R(\Phi_{c,t,d,\delta})\left(t-T_{\bar k_i}^t,X_{T_{\bar k_i}^t}^{\bar k_i}\right) \Bigg)\Bigg]\\
&= \E\Bigg[\prod_{i=1}^{(p-1)N_t^p+1}{\bf 1}_{\{T_{k_i}^t \geq t\}}\frac{\Delta T_{k_i}^t}{\overline \rho(\Delta T_{k_i}^t)} \prod_{i=1}^{N_t} {\bf 1}_{\{T_{\bar k_i}^t < t\}} \frac{\Delta T_{\bar k_i}^t}{\rho(\Delta T_{\bar k_i}^t)}  \left(\displaystyle\xi_{N^p_t}^p(t,x) -\displaystyle\overline \xi_{N^p_t}^p(t,x) \right) \Bigg].
\end{aligned}
\]
Then we have
\[
\begin{aligned}
&|U(t,x)-v(t,x)| \\
&\hspace{.5cm}\leq \sum_{n = 0}^{\infty} \E\left[\prod_{i=1}^{(p-1)n+1}{\bf 1}_{\{T_{k_i}^t \geq t\}}\frac{\Delta T_{k_i}^t}{\overline \rho(\Delta T_{k_i}^t)} \prod_{i=1}^{n} {\bf 1}_{\{T_{\bar k_i}^t < t\}} \frac{\Delta T_{\bar k_i}^t}{\rho(\Delta T_{\bar k_i}^t)} \left| \xi_n^p(t,x)-\overline \xi_n^p(t,x) \right|\right] \PP(N_t^p = n).
\end{aligned}
\]
Recall \eqref{eq:approx_f1} and \eqref{eq:approx_c1}. From Lemma \ref{lemma:mult_DNNs} applied with $k=pn+1$ and \eqref{ck-2}, we conclude that
\begin{equation}\label{chuuuuu}
\left| \xi_n^p(t,x) - \overline \xi_n^p(t,x) \right| \leq (pn+1) \overline B^{pn} \delta,
\end{equation}
with $\overline B:= 2\max\{\delta, \norm{f}_{\infty},\norm{c}_{\infty}\}$. {\color{black} Notice that
\begin{equation}\label{por_demostrar}
\overline B < \frac{1}{T} \left(\frac{\lambda (2p+1)(2p+3)}{T\left(e^{\lambda T(p-1)}-1\right)}\right)^{\frac 1{2p}} \leq \frac{1}{T} \left(\frac{p+2}{T\left(1-e^{-\lambda T(p-1)}\right)}\right)^{\frac 1{p}}.
\end{equation}
} 
Then, using the definition of $\mathcal I_{n,p}(t)$ in \eqref{def_I_np}, the probability measure of $N_t^p$ in \eqref{eq:prob-Ntp} and \eqref{chuuuuu} one has
\[
\begin{aligned}
|U(t,x) - v(t,x)| &\leq  \delta e^{-\lambda t} \sum_{n=0}^{\infty} \begin{pmatrix} - \frac 1{p-1} \\ n \end{pmatrix} \mathcal I_{n,p}(t) \overline B^{pn}(pn+1) (-1)^n \left(1-e^{-\lambda t(p-1)}\right)^{n}.
\end{aligned}
\]
Moreover, \eqref{eq:Inp} and Corollary \ref{cor:cotas-rec} implies that
\[
\begin{aligned}
|U(t,x) - v(t,x)| &\leq \delta e^{-\lambda t} t \left(1+\sum_{n=1}^{\infty} \begin{pmatrix} - \frac 1{p-1} \\ n \end{pmatrix} \frac{t^{(p+1)n}}{(p+1)n}\frac{1}{(p+2)^n} \overline B^{pn}(pn+1) (-1)^n\left(1-e^{-\lambda t(p-1)}\right)^{n}\right) \\
&\leq \delta e^{-\lambda t} t \left(1+\sum_{n=1}^{\infty} \begin{pmatrix} - \frac 1{p-1} \\ n \end{pmatrix} \left( \frac{t^{p+1}\overline B^p \left(e^{-\lambda t(p-1)}-1 \right)}{p+2} \right)^n\right)\\
&= \delta e^{-\lambda t} t \sum_{n=0}^{\infty} \begin{pmatrix} - \frac 1{p-1} \\ n \end{pmatrix} \left( \frac{t^{p+1}\overline B^p \left(e^{-\lambda t(p-1)}-1\right)}{p+2} \right)^n.
\end{aligned}
\]
It follows from the bound for $\overline B$ in \eqref{por_demostrar} that for all $t \in [0,T]$
\[
\frac{t^{p+1}\overline B^p \left(1-e^{-\lambda t(p-1)}\right)}{p+2} < 1,
\]
which implies
\[
\sum_{n=0}^{\infty} \begin{pmatrix} - \frac 1{p-1} \\ n \end{pmatrix} \left( \frac{t^{p+1}\overline B^p (e^{-\lambda t(p-1)}-1)}{p+2} \right)^n = \left(1 - \frac{t^{p+1}\overline B^p \left(1-e^{-\lambda t(p-1)}\right)}{p+2} \right)^{-\frac{	1}{p-1}}.
\]
Therefore
\[
|U(t,x) - v(t,x)| \leq \delta e^{-\lambda t} t  \left(1 - \frac{t^{p+1}\overline B^p \left(1-e^{-\lambda t(p-1)}\right)}{p+2} \right)^{-\frac{	1}{p-1}}.
\]

\medskip
\noindent
{\bf Step 2:} For $n \in \N$, we want to approximate $\overline \xi_n^p(t,x)$ by a DNN. Notice that for any $k_i \in K_t$, $\bar k_i \in \overline K_t \setminus K_t$
\[
\left|\mathcal R(\Phi_{f,d,\delta})\left(X_{T_{k_i}^t}^{k_i}\right)\right| \leq \overline B, \hspace{.5cm} \hbox{and} \hspace{.5cm} \left|\mathcal R(\Phi_{c,t,d,\delta})\left(t - T_{\overline k_i}^t, X_{T_{\bar k_i}^t}^{\bar k_i}\right)\right| \leq \overline B.
\]
Therefore from Lemma \ref{lemma:DNN_mult_k} applied for $k = pn + 1$, for all $\gamma \in \left(0,\frac 12\right)$ there exists $\Psi_{\overline B,\gamma}^{n,p} \in {\bf N}$ such that
\[
\left| \overline \xi_n^p(t,x) - \mathcal R\left(\Psi_{\overline B,\gamma}^{n,p}\right)(x)\right| \leq \gamma pn \overline B^{pn+1}.
\]
Define $\Gamma \in {\bf N}$ by its realization
\[
\mathcal R(\Gamma)(x) := \prod_{i=1}^{(p-1)N_t^p+1} {\bf 1}_{\{T_{k_i}^t\geq t\}}\frac{\Delta T_{k_i}^t}{\overline \rho(\Delta T_{k_i}^t)}\prod_{i=1}^{N_t^p} {\bf 1}_{\{T_{\bar k_i}^t < t\}} \frac{\Delta T_{\bar k_i}^t}{\rho(\Delta T_{\bar k_i}^t)}\mathcal R\left(\Psi_{\overline B,\gamma}^{N_t^p,p}\right)(x).
\] 
This DNN satisfies
\[
\begin{aligned}
\big| v(t,x) - \E[\mathcal R(\Gamma)(x)] \big| &\leq \gamma p \sum_{n=0}^{\infty} \mathcal I_{n,p}(t)  n \overline B^{pn+1} \PP(N_t^p = n) \\
&\leq \gamma e^{-\lambda t} t \overline B\sum_{n=0}^{\infty} \begin{pmatrix} - \frac 1{p-1} \\ n \end{pmatrix} \left( \frac{t^{p+1}\overline B^p \left(e^{-\lambda t(p-1)}-1\right)}{p+2} \right)^n.
\end{aligned}
\]
This bound comes from the same arguments used in Step 1. Finally, one has that
\[
\big| v(t,x) - \E[\mathcal R(\Gamma)(x)] \big| \leq \gamma e^{-\lambda t} t\overline B \left(1 - \frac{t^{p+1}\overline B^p \left(1-e^{-\lambda t(p-1)}\right)}{p+2} \right)^{-\frac{	1}{p-1}}.
\]

\medskip
\noindent
{\bf Step 3:} Notice that $\Gamma$ is a DNN that involves the random variables $T_k^t$ and the corresponding $X_{T_k^t}^k$.  denote by $\Gamma_\ell$ the DNN with the same structure than $\Gamma$ but that takes $T_{k,\ell}^t$ and $X_{T_{k,\ell}^t,\ell}^k$, independent copies of $T_k^t$ and $X_{T_k^t}^k$. Let $M \in \N$. We want to establish the bound
\[
\E\left[\left|\E[\mathcal R(\Gamma)(x)] - \frac 1M \sum_{\ell=1}^{M} \mathcal R(\Gamma_{\ell})(x) \right|^2\right]^{\frac 12} \leq \frac{1}{\sqrt M} \E\left[\mathcal R(\Gamma)(x)^2\right]^{\frac 12} .
\]
By the definition of $\Gamma$ we have
\[
\begin{aligned}
& \E\left[\mathcal R(\Gamma)(x)^2\right] \\
&\hspace{.5cm}= \sum_{n=0}^{\infty} \E\left[\left(\prod_{i=1}^{n+1} {\bf 1}_{\{T_{k_i}^t \geq t\}}\frac{\Delta T_{k_i}^t}{\overline \rho(\Delta T_{k_i}^t)} \prod_{j=1}^{n} {\bf 1}_{\{T_{k_j}^t < t\}} \frac{\Delta T_{k_j}^t}{\rho \Delta T_{k_j}^t}\right)^2 \mathcal R\left(\Psi_{\overline B,\gamma}^{n,p}\right)(x)^2\right]\PP(N_t^p = n) \\
&\hspace{.5cm}\leq \sum_{n=0}^{\infty} \begin{pmatrix} - \frac{1}{p-1} \\ n \end{pmatrix}\mathcal J_{n,p}(t)\left(pn+1\right)^2 \overline B^{2(pn+1)} (-1)^ne^{-\lambda t}\left(1-e^{-\lambda t(p-1)}\right)^n.
\end{aligned}
\]
It follows from \eqref{eq:Jnp} and \eqref{eq:cota-conv-bnp} that
\[
\begin{aligned}
&\E\left[\mathcal R(\Gamma)(x)^2\right] \\
&\leq J_{0,p}(t)\overline B^2 e^{-\lambda t}+ t^2\overline B^2\sum_{n=1}^{\infty} \begin{pmatrix} -\frac{1}{p-1} \\ n\end{pmatrix}\left(\frac{t^{2p+1}\overline B^{2p}\left(1-e^{\lambda t(p-1)}\right)}{\lambda (2p+1)(2p+3)}\right)^n \frac{\left(pn+1\right)^2}{n(p+1)}\\
&\leq t^2 \overline B^2 + t^2 \overline B^2 \sum_{n=1}^{\infty}\begin{pmatrix} -\frac{1}{p-1} \\ n\end{pmatrix}\left(\frac{t^{2p+1}\overline B^{2p}\left(1-e^{\lambda t(p-1)}\right)}{\lambda (2p+1)(2p+3)}\right)^n (pn+1)\\
&= t^2 \overline B^2\sum_{n=0}^{\infty} \begin{pmatrix} -\frac{1}{p-1} \\ n\end{pmatrix}\left(\frac{t^{2p+1}\overline B^{2p}\left(1-e^{\lambda t(p-1)}\right)}{\lambda (2p+1)(2p+3)}\right)^n (pn+1).
\end{aligned}
\]
Recall \eqref{por_demostrar}. Then for all $t \in [0,T]$
\[
\left|\frac{t^{2p+1}\overline B^{2p}\left(1-e^{\lambda t(p-1)}\right)}{\lambda (2p+1)(2p+3)}\right| < 1,
\]
and therefore
\[
\sum_{n=0}^{\infty} \begin{pmatrix} -\frac{1}{p-1} \\ n\end{pmatrix}\left(\frac{t^{2p+1}\overline B^{2p}(1-e^{\lambda t(p-1)})}{\lambda (2p+1)(2p+3)}\right)^n = \left( 1 -  \frac{t^{2p+1}\overline B^{2p}\left(e^{\lambda t(p-1)}-1\right)}{\lambda (2p+1)(2p+3)}\right)^{-\frac{1}{p-1}}.
\]
In addition,
\[
\begin{aligned}
&\sum_{n=0}^{\infty} \begin{pmatrix} -\frac{1}{p-1} \\ n\end{pmatrix}\left(\frac{t^{2p+1}\overline B^{2p}\left(1-e^{\lambda t(p-1)}\right)}{\lambda (2p+1)(2p+3)}\right)^n n \\
&= \frac{1}{p-1}\left(\frac{t^{2p+1}\overline B^{2p}\left(e^{\lambda t(p-1)}-1\right)}{\lambda (2p+1)(2p+3)}\right)\sum_{n=0}^{\infty} \begin{pmatrix} -\frac{p}{p-1} \\ n\end{pmatrix}\left(\frac{t^{2p+1}\overline B^{2p}\left(1-e^{\lambda t(p-1)}\right)}{\lambda (2p+1)(2p+3)}\right)^n\\
&=\frac{1}{p-1}\left(\frac{t^{2p+1}\overline B^{2p}\left(e^{\lambda t(p-1)}-1\right)}{\lambda (2p+1)(2p+3)}\right)\left( 1 -  \frac{t^{2p+1}\overline B^{2p}\left(e^{\lambda t(p-1)}-1\right)}{\lambda (2p+1)(2p+3)}\right)^{-\frac{p}{p-1}}.
\end{aligned}
\]
Therefore
\[
\begin{aligned}
&\E\left[\mathcal R(\Gamma)(x)^2\right]\\
& \leq t^2\overline B^2 \left( 1 -  \frac{t^{2p+1}\overline B^{2p}\left(e^{\lambda t(p-1)}-1\right)}{\lambda (2p+1)(2p+3)}\right)^{-\frac{1}{p-1}} \\
&\hspace{.3cm}+ \frac{p}{p-1}t^2\overline B^2\left(\frac{t^{2p+1}\overline B^{2p}\left(e^{\lambda t(p-1)}-1\right)}{\lambda (2p+1)(2p+3)}\right)\left( 1 -  \frac{t^{2p+1}\overline B^{2p}\left(e^{\lambda t(p-1)}-1\right)}{\lambda (2p+1)(2p+3)}\right)^{-\frac{p}{p-1}}.
\end{aligned}
\]

\medskip
\noindent
{\bf Step 4:} By combining the results obtained in previous steps implies
\[
\begin{aligned}
&\E\left[\left| U(t,x) - \frac 1M \sum_{\ell=1}^{M} \mathcal R(\Gamma_{\ell})(x) \right|^2\right]^{\frac{1}{2}} \\
&\hspace{.5cm}\leq \left|U(t,x) - v(t,x)\right| + |v(t,x) - \E\left[\mathcal R(\Gamma)(x)\right]|+ \E\left[\left| \E\left[\mathcal R(\Gamma)(x)\right] - \frac 1M \sum_{\ell = 1}^{M} \mathcal R(\Gamma_{\ell})(x)\right|^2\right]^{\frac 12} \\
&\hspace{.5cm}\leq e^{-\lambda t} t\left(\delta + \gamma \overline B\right) \left(1- \frac{t^{p+1}\overline B^p \left(1-e^{-\lambda t(p-1)}\right)}{p+2}\right)^{-\frac 1{p-1}} \\
&\hspace{1cm}+ \frac{t\overline B}{\sqrt M} \left(\left(1- \frac{t^{2p+1}\overline B^{2p} \left(e^{\lambda t(p-1)}-1\right)}{\lambda (2p+1)(2p+3)}\right)^{-\frac 1{p-1}} \right.\\
&\hspace{2cm}\left.+ \frac{p}{p-1}\left(\frac{t^{2p+1}\overline B^{2p}\left(e^{\lambda t(p-1)}-1\right)}{\lambda (2p+1)(2p+3)}\right)\left( 1 -  \frac{t^{2p+1}\overline B^{2p}\left(e^{\lambda t(p-1)}-1\right)}{\lambda (2p+1)(2p+3)}\right)^{-\frac{p}{p-1}}\right)^{\frac{1}{2}}.
\end{aligned}
\]

\medskip
\noindent
{\bf Step 5:} From Lemma \ref{lem:moments-Ntp} it follows that
\[
\begin{aligned}
\E\left[\left|\E[N_t]-\frac 1M \sum_{i=1}^M N_{t,\ell}\right|^2\right]^{\frac 12} &= \frac{1}{\sqrt M} (E[N_t^2] - E[N_t]^2)^{\frac 12} \\
&= \frac{1}{\sqrt M} \left(\frac{e^{\lambda t(p-1)}-1}{p-1} + p \left(\frac{e^{\lambda t(p-1)}-1}{p-1}\right)^{2} - \left(\frac{e^{\lambda t(p-1)}-1}{p-1}\right)^2\right)^{\frac 12} \\
&= \frac{1}{\sqrt M} \left(\frac{e^{\lambda t(p-1)}-1}{p-1}e^{\lambda t(p-1)}\right)^{\frac 12}.
\end{aligned}
\]
{\color{black}
\medskip
\noindent
{\bf Step 6:} Let us choose $\gamma = \delta$, $M = \lceil \delta^{-2}\rceil$, and
\[
\delta = \frac{\varepsilon}{2} \min\left\{1,\frac{1}{T}\left(\frac{\lambda(2p+1)(2p+3)}{T\left(e^{\lambda T(p-1)}-1\right)}\right)^{\frac 1{2p}},\frac{2}{C(T)}\right\} \in \left(0,\frac{1}{2}\right),
\]
where
\[
\begin{aligned}
C(T):=\max_{t \in [0,T]} &\left\{\sqrt{2}e^{-\lambda t}t(1+\overline B)\left(1-\frac{t^{p+1}\overline B^p \left(1-e^{-\lambda t(p-1)}\right)}{p+2}\right)^{-\frac1{p-1}} + \left(\frac{e^{\lambda t(p-1)}-1}{p-1}e^{\lambda t(p-1)}\right)^{\frac 12}\right. \\
&\hspace{.2cm}+ \sqrt{2}t\overline B \left(\left(1- \frac{t^{2p+1}\overline B^{2p} \left(e^{\lambda t(p-1)}-1\right)}{\lambda (2p+1)(2p+3)}\right)^{-\frac 1{p-1}} \right.\\
&\hspace{1cm}\left.\left.+ \frac{p}{p-1}\left(\frac{t^{2p+1}\overline B^{2p}\left(e^{\lambda t(p-1)}-1\right)}{\lambda (2p+1)(2p+3)}\right)\left( 1 -  \frac{t^{2p+1}\overline B^{2p}\left(e^{\lambda t(p-1)}-1\right)}{\lambda (2p+1)(2p+3)}\right)^{-\frac{p}{p-1}}\right)^{\frac{1}{2}}\right\}. 
\end{aligned}
\]
Therefore,
\[
\sup_{|x|\leq t} \left| U(t,x) - \frac 1M \sum_{\ell = 1}^{M} \mathcal R(\Gamma_{\ell})(x)\right|+ \left(\frac{1}{|B(0,t)|} \int_{B(0,t)} \left| U(t,x) - \frac 1M \sum_{\ell = 1}^{M} \mathcal R(\Gamma_{\ell})(x)\right|^2dx\right)^{\frac 12} \leq \varepsilon. 
\]
Moreover
\[
\begin{aligned}
\left| \sum_{\ell = 1}^{M} \overline N_{t,\ell}^p\right| &\leq M\left(\left|\E[N_t^p]-\frac{1}{M} \sum_{i=1}^{M} \overline N_{t,\ell}^{p}\right| + \E[N_t^p]\right) \\
&\leq M \left({\bf error}_{u} + \E\left[N_t^p\right]\right) \leq M \left(1 + \frac{e^{\lambda t(p-1)}-1}{p-1} \right).
\end{aligned}
\]

\subsection{Computation of DNN parameters}
In this section we will show that 
\[
\frac{1}{M} \sum_{\ell=1}^{M} \mathcal R(\Gamma_{\ell})(\cdot),
\]
is the realization of some DNN $\Phi_{d,t,\varepsilon} \in {\bf N}$. For any $\ell \in \{1,\ldots,M\}$, $k_{i,\ell} \in K_{t,\ell}$, $\bar k_{i,\ell} \in \overline K_{t,\ell} \setminus K_{t,\ell}$ define
\[
\mathcal R(\phi_{i,\ell})(x) = \mathcal R\left(\Phi_{f,d,\delta}\right)\left(x + X_{T^t_{k_{i,\ell},\ell},\ell}^{k_{i,\ell}}\right), \hspace{.5cm} i \in \left\{1,\ldots,(p-1)\overline N_{t,\ell}^p+1\right\},
\]
and
\[
\mathcal R(\overline \phi_{i,\ell})(x) = \mathcal R \left(\Phi_{c,t,d,\delta}\right)\left(t-T_{\bar k_{i,\ell}}^t, x + X_{T^t_{\bar k_{i,\ell},\ell},\ell}^{\bar k_{i,\ell}}\right), \hspace{.5cm} i \in \left\{1,\ldots,\overline N_{t,\ell}^{p}\right\}.
\]
Lemmas \ref{lemma:DNN_affine} and \ref{cor:DNN_fixed_comp} implies for all $\ell \in \{1,\ldots,M\}$
\[
\mathcal P(\phi_{i,\ell}) = \mathcal P\left(\Phi_{f,d,\delta}\right), \hspace{.5cm} \mathcal H(\phi_{i,\ell}) = \mathcal H\left(\Phi_{f,d,\delta}\right), \hspace{.5cm} \hbox{for all} \hspace{.5cm} i \in \{1,\ldots,(p-1)\overline N_{t,\ell}^{p}+1\},
\]
and
\[
\mathcal P(\overline \phi_{i,\ell}) \leq 2 \mathcal P\left(\Phi_{c,t,d,\delta}\right) + 4(2d+1), \hspace{.5cm} \mathcal H(\overline \phi_{i,\ell}) = \mathcal H\left(\Phi_{c,t,d,\delta}\right) + 2, \hspace{.5cm} \hbox{for all} \hspace{.5cm} i \in \{1,\ldots,\overline N_{t,\ell}^{p}\}.
\]
Notice that for all $\ell \in \{1,\ldots,M\}$ may not necessarily have the same number of hidden layers, then we will separate in three cases:
\begin{enumerate}
\item If $\mathcal H\left(\Phi_{f,d,\delta}\right) = \mathcal H\left(\Phi_{c,t,d,\delta}\right) + 2$, we can parallelize the DNNs $\phi_{i,\ell}$, $\overline \phi_{j,\ell}$, $i \in \{1,\ldots,(p-1)\overline N_{t,\ell}^p + 1\}$, $j \in \{1,\ldots,\overline N_{t,\ell}^p\}$ without any problem.

\item If $\mathcal H\left(\Phi_{f,d,\delta}\right) < \mathcal H\left(\Phi_{c,t,d,\delta}\right) + 2$, we extend each $\phi_{i,\ell}$, $i \in \{1,\ldots,(p-1)\overline N_{t,\ell}^p+1\}$ to have $\mathcal H\left(\Phi_{c,t,d,\delta}\right) + 2$ hidden layers. In this case, Lemma \ref{lem:DNN_extension} implies that $\mathcal H\left(\phi_{i,\ell}\right) =\mathcal H\left(\Phi_{c,t,d,\delta}\right) + 2$ and
\[
\mathcal P\left(\phi_{i,\ell}\right) \leq 2\mathcal P\left(\Phi_{f,d,\delta}\right) + 4\left(\mathcal H\left(\Phi_{c,t,d,\delta}\right) + 2 - \mathcal H\left(\Phi_{f,d,\delta}\right) \right).
\]
\item If $\mathcal H\left(\Phi_{c,t,d,\delta}\right) + 2 < \mathcal H \left(\Phi_{f,d,\delta}\right)$, we extend each $\overline \phi_{i,\ell}$, $i \in \{1,\ldots,\overline N_{t,\ell}^p\}$ to have $\mathcal H\left(\Phi_{f,d,\delta}\right)$ hidden layers. Lemma \ref{lem:DNN_extension} implies that  $\mathcal H\left(\overline \phi_{i,\ell}\right) =\mathcal H\left(\Phi_{f,d,\delta}\right)$ and
\[
\mathcal P\left(\overline \phi_{i,\ell}\right) \leq 4\mathcal P\left(\Phi_{c,t,d,\delta}\right) + 8(2d+1) + 4\left(\mathcal H\left(\Phi_{f,d,\delta}\right) - \mathcal H\left(\Phi_{c,t,d,\delta}\right) - 2\right).
\]
\end{enumerate}
In the three cases one has for $i \in \{1,\ldots,(p-1)\overline N_{t,\ell}^p + 1\}$, $j \in \{1,\ldots,\overline N_{t,\ell}^p\}$
\[
\mathcal H\left(\phi_{i,\ell}\right) = \mathcal H\left(\overline \phi_{j,\ell}\right) = \max\{\mathcal H\left(\Phi_{f,d,\delta}\right),\mathcal H\left(\Phi_{c,t,d,\delta}\right) + 2\},
\]
\[
\mathcal P\left(\phi_{i,\ell}\right) \leq 2\mathcal P\left(\Phi_{f,d,\delta}\right) + 4\max\{\mathcal H\left(\Phi_{f,d,\delta}\right),\mathcal H\left(\Phi_{c,t,d,\delta}\right) + 2\},
\]
and
\[
\mathcal P\left(\overline \phi_{j,\ell}\right) \leq 4\mathcal P\left(\Phi_{c,t,d,\delta}\right) + 8(2d+1)+ 4\max\{\mathcal H\left(\Phi_{f,d,\delta}\right),\mathcal H\left(\Phi_{c,t,d,\delta}\right) + 2\}.
\]

\medskip
For all $\ell \in \{1,\ldots,M\}$ define $\varphi_{\ell} \in {\bf N}$ by its realization
\[
\mathcal R(\varphi_{\ell})(x) := \left(\mathcal R\left(\phi_{1,\ell}\right)(x),\ldots,\mathcal R\left(\phi_{(p-1)\overline N_{t,\ell}^p + 1,\ell}\right)(x),\mathcal R\left(\overline \phi_{1,\ell}\right)(x),\ldots,\mathcal R\left(\overline \phi_{\overline N_{t,\ell}^p,\ell}\right)(x)\right).
\] 
Thanks to Lemma \ref{lemma:DNN_para}, $\varphi_{\ell}$ is indeed a DNN and $\mathcal R (\varphi_{\ell}) \in C\left(\R^{d},\R^{p\overline N_{t,\ell}^p + 1}\right)$. Moreover, $\varphi_{\ell}$ satisfies
\[
\begin{aligned}
\mathcal P\left(\varphi_{\ell}\right) &= \sum_{i=1}^{(p-1)\overline N_{t,\ell}^p + 1} \mathcal P\left(\phi_{i,\ell}\right) + \sum_{i=1}^{\overline N_{t,\ell}^p} \mathcal P\left(\overline \phi_{i,\ell}\right) \\
&\leq 2\left((p-1)\overline N_{t,\ell}^{p}+1\right)\mathcal P\left(\Phi_{f,d,\delta}\right) + 4 \overline N_{t,\ell}^p \mathcal P\left(\Phi_{c,t,d,\delta}\right) + 8(2d+1)\overline N_{t,\ell}^p \\
&\hspace{.5cm} + 4 \left(p\overline N_{t,\ell}^p + 1\right)\max\{\mathcal H\left(\Phi_{f,d,\delta}\right),\mathcal H\left(\Phi_{c,t,d,\delta}\right) + 2\},
\end{aligned}
\]
and
\[
\mathcal H(\varphi_{\ell}) = \max\left\{\mathcal H\left(\Phi_{f,d,\delta}\right), \mathcal H\left(\Phi_{c,t,d,\delta}\right)+2\right\}.
\]
On the other hand side, notice that the neural network $\Psi_{\overline B,\gamma,\ell}^{\overline N_{t,\ell}^p,p}$ is a composition between $\varphi_{\ell}$ and the DNN $\Phi_{\overline B,\gamma}^{p\overline N_{t,\ell}^p+1}$ found in Lemma \ref{lemma:DNN_mult_k}. Lemma \ref{lemma:DNN_comp} implies that 
\[
\mathcal P \left(\Psi_{\overline B,\gamma,\ell}^{\overline N_{t,\ell}^p,p}\right) \leq 2\mathcal P\left(\Phi_{\overline B,\gamma}^{p\overline N_{t,\ell}^p+1}\right) + 2 \mathcal P \left(\varphi_{\ell}\right),
\]
and
\[
\mathcal H \left(\Psi_{\overline B,\gamma,\ell}^{\overline N_{t,\ell}^p,p}\right) = \mathcal H\left(\Phi_{\overline B,\gamma}^{p\overline N_{t,\ell}^p+1}\right) + \mathcal H \left(\varphi_{\ell}\right) + 1.
\]
Here, for $\ell \in \{1,\ldots,M\}$, the DNNs $\Phi_{\overline B,\gamma}^{p\overline N_{t,\ell}^p+1}$ may not necessarily have the same number of hidden layers. Therefore from Lemma \ref{lem:DNN_extension} one can extend each DNN to have
\[
\max_{\ell = 1,\ldots,M} \mathcal H\left(\Phi_{\overline B,\gamma}^{p\overline N_{t,\ell}^p+1}\right)
\]
hidden layer, with each DNN satisfying
\[
\mathcal P\left(\Phi_{\overline B,\gamma}^{p\overline N_{t,\ell}^p+1}\right) \leq 2C' \left(p\overline N_{t,\ell}^p+1\right)^4 \left(\log \lceil \gamma^{-1}\rceil + 1 + \log\lceil R^{-1} \rceil \right) + 4\max_{\ell = 1,\ldots,M} \mathcal H\left(\Phi_{\overline B,\gamma}^{p\overline N_{t,\ell}^p+1}\right),
\]
and therefore, the DNNs $\Psi_{\overline B,\gamma,\ell}^{\overline N_{t,\ell}^p,p}$ will have the same number of hidden layers, equal to
\[
\mathcal H \left(\Psi_{\overline B,\gamma,\ell}^{\overline N_{t,\ell}^p,p}\right) = \max_{\ell = 1,\ldots,M}\mathcal H\left(\Phi_{\overline B,\gamma}^{p\overline N_{t,\ell}^p+1}\right) + \mathcal H \left(\varphi_{\ell}\right) + 1.
\]

Recall that for all $\ell \in \{1,\ldots,M\}$
\[
\mathcal R(\Gamma_{\ell})(x) = \prod_{i=1}^{(p-1) \overline N_{t,\ell}^p+1} {\bf 1}_{\{T_{k_i,\ell}^t\geq t\}}\frac{\Delta T_{k_i,\ell}^t}{\overline \rho(\Delta T_{k_i,\ell}^t)}\prod_{i=1}^{\overline N_{t,\ell}^p} {\bf 1}_{\{T_{\bar k_i,\ell}^t < t\}} \frac{\Delta T_{\bar k_i,\ell}^t}{\rho(\Delta T_{\bar k_i,\ell}^t)}\mathcal R\left(\Psi_{\overline B,\gamma,\ell}^{\overline N_{t,\ell}^p,p}\right)(x).
\]
By applying again Lemma \ref{lemma:DNN_affine}, one has that
\[
\mathcal P\left(\Gamma_{\ell}\right) = \mathcal P \left(\Psi_{\overline B,\gamma,\ell}^{\overline N_{t,\ell}^p,p}\right), \hspace{.5cm} \mathcal H\left(\Gamma_{\ell}\right) = \mathcal H \left(\Psi_{\overline B,\gamma,\ell}^{\overline N_{t,\ell}^p,p}\right).
\]
Finally, define $\Phi_{t,d,\varepsilon}$ the DNN with realization
\[
\mathcal R\left(\Phi_{t,d,\varepsilon}\right)(x) = \frac{1}{M} \sum_{\ell=1}^{M} \mathcal R\left(\Gamma_{\ell}\right)(x).
\]
We conclude:
\[
\sup_{|x|\leq t} \left| U(t,x) - \mathcal R\left(\Phi_{t,d,\varepsilon}\right)(x) \right|+ \left(\frac{1}{|B(0,t)|} \int_{B(0,t)} \left| U(t,x) -\mathcal R\left(\Phi_{t,d,\varepsilon}\right)(x)  \right|^2dx\right)^{\frac 12} \leq \varepsilon,
\]
proving \eqref{eq:main_eps_nl} in this case. From Lemma \ref{lemma:DNN_suma}, $\Phi_{t,d,\varepsilon}$ satisfies
\[
\mathcal P\left(\Phi_{t,d,\varepsilon}\right)(x) \leq \sum_{\ell=1}^{M} \mathcal P\left(\Gamma_{\ell}\right), \hspace{.5cm} \mathcal H\left(\Phi_{t,d,\varepsilon}\right) = \mathcal H \left(\Psi_{\overline B,\gamma,\ell}^{\overline N_{t,\ell}^p,p}\right).
\]
By combining the bounds, one has
\[
\begin{aligned}
\sum_{\ell=1}^{M} \mathcal P\left(\Gamma_{\ell}\right) &\leq 4C' \left(\log \lceil \gamma^{-1}\rceil + 1 + \log\lceil \overline B^{-1} \rceil \right) \sum_{\ell=1}^M\left(p\overline N_{t,\ell}^p+1\right)^4 + 8 M \max_{\ell = 1,\ldots,M} \mathcal H\left(\Phi_{\overline B,\gamma}^{p\overline N_{t,\ell}^p+1}\right) \\
&\hspace{.5cm}+ 2\left((p-1)\sum_{\ell=1}^M\overline N_{t,\ell}^{p}+M\right)\mathcal P\left(\Phi_{f,d,\delta}\right) + 4 \sum_{\ell=1}^M\overline N_{t,\ell}^p \mathcal P\left(\Phi_{c,t,d,\delta}\right) + 8(2d+1)\sum_{\ell=1}^M\overline N_{t,\ell}^p \\
&\hspace{.5cm} + 4 \left(p\sum_{\ell=1}^M\overline N_{t,\ell}^p + M\right)\max\{\mathcal H\left(\Phi_{f,d,\delta}\right),\mathcal H\left(\Phi_{c,t,d,\delta}\right) + 2\}.
\end{aligned}
\]
We bound each quantity. First recall that
\[
\sum_{\ell=1}^{M} \overline N_{t,\ell}^{p} \leq M\left(1+ \frac{e^{\lambda t(p-1)}-1}{p-1}\right),
\]
therefore
\[
\begin{aligned}
\sum_{\ell=1}^{M} \left(p\overline N_{t,\ell}^p + 1\right)^4 &\leq \left(p\sum_{\ell=1}^{M} \overline N_{t,\ell}^{p} + M\right)^4 \\
&\leq \left(pM\left(1+ \frac{e^{\lambda t(p-1)}-1}{p-1}\right) + M\right)^4 \\
&= M^4 \left(p+1+p\frac{e^{\lambda t(p-1)}-1}{p-1}\right)^4.
\end{aligned}
\]
Additionally, for all $\ell \in \{1,\ldots,M\}$,
\[
 \mathcal H\left(\Phi_{\overline B,\gamma}^{p\overline N_{t,\ell}^p+1}\right) \leq C \left(p\overline N_{t,\ell}^p+1\right) \left( \log\lceil \gamma^{-1}\rceil + \log \left(p\overline N_{t,\ell}^p+1\right) + \left(p\overline N_{t,\ell}^p+1\right) \log \lceil \overline B^{-1} \rceil \right).
\]
Then,
\[
\begin{aligned}
\max_{\ell=1,\ldots,M} \mathcal H\left(\Phi_{\overline B,\gamma}^{p\overline N_{t,\ell}^p+1}\right) &\leq C\left(\log\lceil \gamma^{-1}\rceil + 1 +\log \lceil \overline B^{-1} \rceil \right) \sum_{\ell=1}^{M} \left(p\overline N_{t,\ell}^p+1\right)^2 \\
&\leq CM^2\left(\log\lceil \gamma^{-1}\rceil + 1 +\log \lceil \overline B^{-1} \rceil \right) \left(p+1+p\frac{e^{\lambda t(p-1)}-1}{p-1}\right)^2.
\end{aligned}
\]
This implies that
\[
\begin{aligned}
& \sum_{\ell=1}^{M} \mathcal P\left(\Gamma_{\ell}\right) \\
&\quad \leq  \widetilde C M \left( \gamma^{-1} M^3 + \gamma^{-1}M^2 + \mathcal P\left(\Phi_{f,d,\delta}\right) +\mathcal P\left(\Phi_{c,t,d,\delta}\right) + d + \max\{\mathcal H\left(\Phi_{f,d,\delta}\right),\mathcal H\left(\Phi_{c,t,d,\delta}\right) + 2\}\right),
\end{aligned}
\]
where $\widetilde C = \widetilde C\left(\lambda, T, p, \norm{f}_{\infty}, \norm{c}_{\infty}\right)$ is some positive constant. In addition, assumptions \eqref{eq:cota_param_phiu_nl} and \eqref{eq:dim_dnn_encontradas_nl} implies that
\[
\begin{aligned}
\sum_{\ell=1}^{M} \mathcal P\left(\Gamma_{\ell}\right) &\leq  \widetilde C M \left(\gamma^{-1} M^3 + \gamma^{-1}M^2 + 2Bd^p\delta^{-\alpha} + d + Bd^p\delta^{-\beta}+2\right),
\end{aligned}
\]
and from the choices of $M,\delta,\gamma$ it follows that
\[
\sum_{\ell=1}^{M} \mathcal P\left(\Gamma_{\ell}\right) \leq  2\widetilde C  \varepsilon^{-2} C(T)^2 \left(\varepsilon^{-7} 2^3C(T)^7 + 2^2 \varepsilon^{-5}C(T)^5 + 2Bd^p\varepsilon^{-\alpha}C(T)^{\alpha} + d + Bd^p\varepsilon^{-\beta}C(T)^{\beta}+2\right).
\]
Finally
\[
\sum_{\ell=1}^{M} \mathcal P\left(\Gamma_{\ell}\right) \leq Cd^p \varepsilon^{-\eta},
\]
with $\eta = 2 + \max\{\alpha,\beta,7\}$ and $C = C\left(\lambda, T, p, \norm{f}_{\infty}, \norm{c}_{\infty}\right)$ some positive constant.
}

}

\newpage

\bibliographystyle{alpha}

\newcommand{\etalchar}[1]{$^{#1}$}

\end{document}